\documentclass[10pt]{amsart}
\usepackage[toc,page]{appendix}
\usepackage[english]{babel}
\usepackage
{hyperref}
\usepackage{amsmath}
\usepackage{amssymb,amsthm,amscd}
\usepackage[latin1]{inputenc}
\usepackage[textwidth=18cm,textheight=24cm]{geometry}
\usepackage{mathrsfs}
\usepackage[thinlines]{easybmat}
\usepackage{esint}
\newcommand{\diag}{\operatorname{diag}}

\newcommand{\C}{\mathbb{C}}

\newcommand{\R}{\mathbb{R}}

\newcommand{\Z}{\mathbb{Z}}

\newcommand{\T}{\mathbb{T}}

\renewcommand{\d}{\mathrm{d}}

\newcommand{\Exp}[1]{\operatorname{e}^{#1}}

\renewcommand{\L}{\mathscr L}

\newtheorem{pro}{Proposition}
\newtheorem{lemma}{Lemma}
\newtheorem{definition}{Definition}
\newtheorem{theorem}{Theorem}
\newtheorem{cor}{Corollary}

\newcommand{\Res}{\operatorname{Res}}
\newcommand{\Langle}{\langle\!\langle}
\newcommand{\Rangle}{\rangle\!\rangle}

\begin{document}
\title[MOLPUC and Toda hierarchy]{Matrix Orthogonal Laurent Polynomials on the Unit Circle and Toda Type Integrable Systems}
\author{Gerardo Ariznabarreta}\address{Departamento de F\'{\i}sica Te\'{o}rica II (M\'{e}todos Matem\'{a}ticos de la F\'{\i}sica), Universidad Complutense de Madrid, 28040-Madrid, Spain}
\email{gariznab@ucm.es}
\thanks{GA thanks economical support from the Universidad Complutense de Madrid  Program ``Ayudas para Becas y Contratos Complutenses Predoctorales en Espa\~{n}a 2011"}
\author{Manuel Ma\~{n}as}
\email{manuel.manas@ucm.es}
\thanks{MM thanks economical support from the Spanish ``Ministerio de Econom\'{\i}a y Competitividad" research project MTM2012-36732-C03-01,  \emph{Ortogonalidad y aproximacion; teoria y aplicaciones}}
\maketitle
\begin{abstract}
Matrix orthogonal Laurent polynomials in the unit circle and the theory of Toda-like integrable systems are connected using the Gauss--Borel factorization of two,  left and a right, Cantero--Morales--Vel\'{a}zquez block  moment matrices, which are constructed using a quasi-definite matrix measure. A block Gauss--Borel factorization problem of these moment matrices  leads to two sets of biorthogonal matrix orthogonal Laurent polynomials and matrix Szeg\H{o} polynomials, which can be expressed in terms of Schur complements of bordered truncations of the block moment matrix. The corresponding block extension of the Christoffel--Darboux theory is derived. Deformations of the quasi-definite matrix measure leading to integrable systems of Toda type are studied. The integrable theory is given in this matrix scenario; wave and adjoint wave functions, Lax and Zakharov--Shabat equations, bilinear equations and discrete flows --connected with Darboux transformations--. We generalize the
integrable flows of the Cafasso's matrix  extension of the Toeplitz lattice for the Verblunsky coefficients of Szeg\H{o} polynomials.
An analysis of the Miwa shifts allows for the finding of interesting connections between Christoffel--Darboux kernels  and Miwa shifts of the  matrix orthogonal Laurent polynomials.
\end{abstract}
\tableofcontents

\section{Introduction}

In this paper we extend previous results on orthogonal Laurent polynomials in the unit circle (OLPUC) \cite{carlos} to the matrix  realm (MOLPUC).  To explain better our aims and results we need  a brief account on orthogonal polynomials,  Laurent orthogonal polynomials and their matrix extensions, and also some facts about integrable systems.

\subsection{Historical background}
\subsubsection{Szeg\H{o} polynomials} We will denote the unit circle by $\T:=\{z \in \C:
|z|=1\}$ and $\mathbb{D}:=\{z \in \C : |z|<1\}$ stands for the  unit disk; when $z\in\T$ we will use the parametrization $z=\text{e}^{\text{i}\theta}$ with $\theta\in[0,2\pi)$. In the scalar case, one deals with a complex Borel measure $\mu$ supported in $\T$  that is said to be positive definite if it maps measurable sets into non-negative numbers, that in the  absolutely continuous situation (with respect to the Lebesgue measure  $\d\theta$) has the form $w({\theta}) \d \theta$. For the positive definite situation the orthogonal polynomials in the unit circle (OPUC) or Szeg\H{o} polynomials are defined as those  monic polynomials  $P_n$ of degree  $n$ that satisfy the following system of equations, called orthogonality relations, $\int_{\T}P_n(z) z^{-k} \d \mu(z)=0$, for $ k=0,1,\dots,n-1$, \cite{Szego}.
The connections between orthogonal polynomials on the real line (OPRL) supported in the interval  $[-1,1]$ and OPUC has been explored in the literature, see for example  \cite{Freud,Berriochoa}.  Let us observe that for this analysis the use of spectral theory techniques requires the study of the operator of multiplication by \emph{z}.  Recursion relations for OPRL and OPUC are well known; however, in the real case the three term recurrence laws provide a tridiagonal matrix, the so called Jacobi operator, while  in the unit circle support case the problem leads to a Hessenberg matrix \cite{Golub}, being a more involved scenario that the Jacobi one (as it is not a sparse matrix with a finite number of non vanishing diagonals). In fact, OPUC's recursion relation requires the introduction of  reciprocal or reverse Szeg\H{o} polynomials $P^*_l(z):=z^l \overline{P_l(\bar z^{-1})}$ and the reflection or Verblunsky (Schur parameters is another usual name) coefficients $\alpha_l:=P_l(0)$. The recursion relations for
the Szeg\H{
o} polynomials can be written as $\left(\begin{smallmatrix} P_l \\ P_{l}^* \end{smallmatrix}\right)=
\left(\begin{smallmatrix} z & \alpha_l \\ z \bar \alpha_l & 1 \end{smallmatrix}\right)
\left(\begin{smallmatrix} P_{l-1} \\ P_{l-1}^* \end{smallmatrix}\right)$. There exist numerous studies on the  zeroes of the OPUC,  \cite{Alfaro,Ambrolazde,Barrios-Lopez,Garcia,Godoy,Golinskii2,Mhaskar,Totik} with interesting applications to signal analysis theory \cite{Jones-1,Jones-2,Pan-1,Pan-2}. Despite the mentioned advances for the OPUC theory, the corresponding state of the art in the OPRL context is still much more developed. An issue to stress here is that Szeg\H{o} polynomials are, in general, not a dense set in the Hilbert space $L^2(\T,\mu)$;  Szeg\H{o}'s theorem implies for a nontrivial probability measure $\d\mu$  on $\T$ with
Verblunsky coefficients $\{\alpha_n\}_{n=0}^\infty$ that the corresponding Szeg\H{o}'s polynomials are dense in $L^2(\T,\mu)$ if and
only if $\prod_{n=0}^\infty (1-|\alpha_n)|^2)=0$. For an absolutely continuous probability measure  Kolmogorov's density theorem ensures that density in $L^2(\T,\mu)$ of the  OPUC holds iff   the so called Szeg\H{o}'s condition $\int_{\T}\log (w(\theta)\d\theta=-\infty$ is fulfilled, \cite{Simon-S}. We refer the reader to Barry Simon's books \cite{Simon-1} and \cite{Simon-2} for a very detailed studied of OPUC.

\subsubsection{Orthogonal Laurent  polynomials} Orthogonal Laurent polynomials on the real line (OLPRL), where introduced in \cite{Jones-3,Jones-4} in the context of the strong Stieltjes moment problem.
When this moment problem has a solution, there exist polynomials $\{Q_n\}$, kown as Laurent polynomials,  such that $\int_{\R}x^{-n+j}Q_n(x)\d \mu(x)=0$ for $j=0,\dots,n-1$.
 The theory of Laurent polynomials on the real line was developed in parallel with the theory of orthogonal polynomials, see \cite{Cochran,Diaz,Jones-5} and \cite{Njastad}. Orthogonal Laurent polynomials' theory was carried from the real line to the circle \cite{Thron} and subsequent works broadened the matter (e.g. \cite{Barroso-Vera,CMV,Barroso-Daruis,Barroso-Snake}), treating subjects like recursion relations, Favard's theorem, quadrature problems, and Christoffel--Darboux formulae. The Cantero--Moral--Vel\'{a}zquez  (CMV) \cite{CMV} representation is a hallmark in the study of certain aspects of Szeg\H{o} polynomials, as we mentioned already while the OLPUC are always dense in $L^2(\T,\mu)$ this is not true in general for the OPUC,   \cite{Bul} and \cite{Barroso-Vera}. The bijection between OLPUC in the CMV representation and the ordinary Szeg\H{o} polynomials implies the replacement of complicated recursion relations with  five-term relations similar to the OPRL situation.  Other papers have reviewed and
broadened the study of CMV matrices, see for example \cite{CMV-Simon,Killip} in particular alternative or generic orders in the base used to span the space of OLPUC can be found in \cite{Barroso-Snake}.
In particular, the reading of Simon's account of the CMV theory \cite{CMV-Simon} is illuminating. In fact,  the discovery of the advantages of the CMV ordering goes back to previous work \cite{watkins}.

\subsubsection{Matrix orthogonal polynomials} Orthogonal polynomials with matrix coefficients on the real line were considered in detail  by Krein \cite{krein1, krein2} in 1949,
and thereafter were studied sporadically until the last decade of the XX century. Some relevant papers on this subject are \cite{bere},  \cite{geronimo} and  \cite{nikishin}; in particular,  in \cite{nikishin} the scattering problem is solved
for a kind of discrete Sturm--Liouville operators that are equivalent to the recursion  equation for scalar orthogonal polynomials. They found that
polynomials that satisfy a  relation of the form
\begin{align*}
  xP_{k}(x)&=A_{k}P_{k+1}(x)+B_{k}P_{k}(x)+A_{k-1}^{*}P_{k-1}(x),& k&=0,1,...,
\end{align*}
 are orthogonal with respect to a positive definite measure. This is a  matrix version of Favard's theorem for scalar orthogonal polynomials.
 Then, in the 1990's and the 2000's some authors found that matrix orthogonal  polynomials (MOP) satisfy in certain cases some properties that satisfy
 scalar valued orthogonal polynomials; for example, Laguerre, Hermite and  Jacobi polynomials, i.e., the scalar-type Rodrigues' formula
 \cite{duran20051,duran20052,constin} and a second order differential equation  \cite{duran1997,duran2004,borrego}. Later on, it has been proven \cite{duran2008} that operators of the form
$D$=${\partial}^{2}F_{2}(t)$+${\partial}^{1}F_{1}(t)$+${\partial}^{0}F_{0}$ have as eigenfunctions different infinite families of MOP's. Moreover, in
\cite{borrego} a new family of MOP's satisfying second order differential equations whose coefficients do not behave asymptotically as the identity
matrix was found; see also \cite{cantero}. In \cite{cassatella} the Riemann--Hilbert problem for this matrix situation and the appearance of non-Abelian discrete versions of Painlev\'{e} I were explored , showing singularity confinement --see \cite{cassatella2};
for Riemann--Hilbert problems see also \cite{dominguez}. Let us mention that in \cite{miranian} and \cite{Cafasso} the MOP are expressed in terms of Schur complements that play the role of determinants in the standard scalar case.
For a survey on matrix orthogonal polynomials we refer the reader to \cite{Damanik}.

\subsubsection{Integrable hierarchies and the Gauss--Borel factorization} The seminal paper of M. Sato \cite{sato} and further developments performed by the Kyoto school \cite{date1}-\cite{date3} settled the Lie-group theoretical description of the integrable hierarchies. It was Mulase \cite{mulase} the one who made the connection between factorization problems, dressing procedures and integrability. In this context, Ueno and Takasaki \cite{ueno-takasaki} performed an analysis of the Toda-type hierarchies and their soliton-like solutions.  Adler and van Moerbeke \cite{adler}-\cite{adler-vanmoerbeke-5} have clarified  the connection between the Lie-group factorization, applied to Toda-type hierarchies --what they call discrete Kadomtsev--Petviashvilii (KP)-- and the Gauss--Borel factorization applied to a moment matrix that comes from  orthogonality problems; thus, the corresponding
orthogonal polynomials  are closely related to specific  solutions of the integrable hierarchy. See \cite{bergvelt}, \cite{felipe}, \cite{manas-martinez-alvarez} and \cite{cum} for further developments in relation with the factorization problem, multicomponent Toda lattices and generalized orthogonality. In \cite{Adler-Van-Moerbecke-Toeplitz} a profound study of the OPUC and the Toda type associated lattice, called the Toeplitz lattice (TL), was performed. A relevant reduction of the equations of the TL has been found by Golinskii \cite{Golinskii} in the context of Schur flows when the measure is invariant under conjugation, (also studied in \cite{Simon-Schur} and \cite{Fay1}), another interesting paper on this subject is \cite{Mukaihira}. The Toeplitz lattice was proven to be equivalent to the Ablowitz--Ladik lattice (ALL), \cite{a-l-1,a-l-2}, and that work has been generalized to the link between matrix orthogonal polynomials and the non-Abelian ALL in \cite{Cafasso}. Both of them have to deal
with the Hessenberg operator for the multiplication by $z$. Research about the integrable structure of Schur flows and its connection with ALL has been done (in recent and not so recent works) from a Hamiltonian point of view in \cite{Nenciu}, and other works also introduce connections with Laurent polynomials and $\tau$-functions, like \cite{Fay2}, \cite{Fay3} and \cite{Bertola}.

\subsection{Preliminary material}
\subsubsection{Semi-infinite block matrices}
For the matrix extension considered in the present  work we need to deal with block matrices and block Gauss--Borel factorizations.  For each $m\in \mathbb{N}$, the directed set of natural numbers, we consider  ring of the complex $m\times m$ matrices $\mathbb{M}_m:=\C^{m\times m}$, and its direct limit $\mathbb{M}_{\infty}:=\underset{\longrightarrow}{\lim}\,\mathbb M_m$,  the ring of semi-infinite complex matrices.  We will denote by $\operatorname{diag}_m\subset \mathbb M_m$ the set of diagonal matrices. For any $A \in \mathbb{M}_{\infty}$, $ A_{ij} \in \C$ denotes the $(i,j)$-th element of $A$, while $(A)_{ij} \in \mathbb{M}_m$ denotes the $(i,j)$-th block of it when subdivided into $m\times m$ blocks. We will denote by $G_\infty$ the group of invertible semi-infinite matrices of $\mathbb{M}_\infty$. In this paper two important subgroups are $\mathscr U$, the invertible upper triangular --by blocks--  matrices, and $\mathscr L$, the lower  triangular  --by blocks-- matrices with the
identity matrix along their block diagonal.  The corresponding restriction on invertible upper triangular
block  matrices  is denoted by $\widehat{\mathscr U}$.
Block diagonal matrices will be denoted by $\mathscr D=\{D \in \mathbb{M}_{\infty} : (D)_{i,j}= d_i \cdot \delta_{i,j} \mbox{ with } d_i \in\mathbb{M}_m\} $.
Given a semi-infinite matrix $A\in\mathbb{M}_\infty$ we consider its $l$-th block  leading submatrix
\begin{align*}
 A^{[l]}&=\begin{pmatrix}
   (A)_{0,0} &(A)_{0,1} & \dots &(A)_{0,l-1}\\
   (A)_{1,0} &  (A)_{1,1} & \dots & (A)_{1,l-1}\\
  \vdots&       &       & \vdots\\
   (A)_{l-1,0} & (A)_{l-1,1} & \dots &(A)_{l-1,l-1}
\end{pmatrix}\in\mathbb{M}_{ml}, & (A)_{i,j}&\in\mathbb M_m,
\end {align*}
and  we write
\begin{align}\label{block}
A&=\left(\begin{BMAT}{c|c}{c|c} A^{[l]} & A^{[l,\geq l]} \\ A^{[\geq l, l]} & A^{[\geq l]} \end{BMAT}\right),
\end{align}
for the corresponding block partition of a matrix $A$ where, for example,  $A^{[l,\geq l]}$ denotes all the $(A)_{i,j}$-th blocks of the matrix $A$ with $i<l,j\geq l $. Very much related to the block partition of a matrix $M$ are the Schur complements. The Schur complement
with respect to the upper left block of the block partition
\begin{align*}
 M&=\begin{pmatrix}
    A & B\\
    C & D
   \end{pmatrix}\,\,\in \mathbb{M}_{p+q},&A=(a_{i,j})&\in\mathbb{M}_{p}, D \in \mathbb{M}_q,
\end{align*}
is
\[
M\diagup A :=\operatorname{SC}\left(\begin{BMAT}{c|c}{c|c}
    (a_{i,j}) & B\\
    C & D
   \end{BMAT}\right):= D-C A^{-1} D,
\]
where we have assumed that $A$ is an invertible matrix.

 \subsubsection{Quasi-definiteness}
 Let us recall the reader that measures and linear functionals are closely connected;  given a linear functional $\mathcal L$ on $\Lambda_{[\infty]}$, the set of Laurent polynomials on the circle --or polynomial loops $L_\text{pol}\C$,  we define the corresponding moments of $\mathcal L$ as $c_n:=\mathcal L[z^n]$ for all the possible integer values of $n \in \Z$. The functional $\mathcal L$ is said to be Hermitian whenever $c_{-n}=\overline{c_n}$, $\forall n \in \Z$. Moreover, the functional $\mathcal L$ is defined as quasi-definite (positive definite) when the principal submatrices of the Toeplitz moment matrix $(\Delta_{i,j})$, $\Delta_{i,j}:=c_{i-j}$, associated to the sequence $c_n$ are non-singular (positive definite), i.e. $\forall n \in \Z, \Delta_n:=\det(c_{i-j})_{i,j=0}^n \neq 0 (>0)$. Some aspects on quasi-definite functionals and their perturbations are studied in \cite{alvarez,marcellan}.
It is known \cite{geronimus-2} that when the linear functional $\mathcal L$ is Hermitian and positive definite there exist a finite positive  Borel measure with a support lying on $\T$ such that $\mathcal L[f]=\int_{\T} f \d \mu$, $\forall f\in \Lambda_{[\infty]}$. In addition, a Hermitian positive definite linear functional $\mathcal L$ defines a sesquilinear form $\langle {\cdot},{\cdot}\rangle_{\mathcal L}: \Lambda_{[\infty]} \times \Lambda_{[\infty]} \mapsto \C$ as $\langle f,g\rangle_{\mathcal L}=\mathcal L[f \bar g]$, $\forall f ,g \in \Lambda_{[\infty]}$. Two Laurent polynomials $\{f,g\}\subset\Lambda_{[\infty]}$ are said to be orthogonal with respect to $\mathcal L$ if $\langle f,g \rangle_{\mathcal L}=0$. From the properties of $\mathcal L$ it is easy to see that $\langle {\cdot},{\cdot} \rangle_{\mathcal L} $ is a scalar product and if $\mu$ is the positive finite Borel measure associated to $\L$ we are lead to the corresponding Hilbert space $L^2(\T,\mu)$, the closure of $\Lambda_{[\infty]}$. The
more general setting when $\mathcal L$ is just quasi-definite is associated to a corresponding quasi-definite complex measure $\mu$, see \cite{gautschi}. As  before a sesquilinear form $\langle {\cdot},{\cdot}\rangle_{\mathcal L}$ is defined for any such linear functional $\mathcal L$; thus, we just have the linearity (in the first entry) and skew-linearity (in the second entry) properties. However,  we have no symmetry allowing the interchange of the two arguments. We formally broaden the notion of orthogonality and say that $f$ is orthogonal to $g$ if $\langle f,g \rangle_{\mathcal L}=0$, but we must be careful as in this general situation it could  happen that $ \langle f,g \rangle_{\mathcal L}=0 $ but $ \langle g,f \rangle_{\mathcal L} \neq 0 $.

\subsubsection{Matrix Laurent polynomials and orthogonality} A matrix valued measure $\mu=(\mu_{i,j})$ supported on  $\T$ is said  to be Hermitian and/or positive definite, if for every   Borel subset
 $\mathscr B$ of $\T$ the matrix $\mu(\mathscr B)$ is a Hermitian and/or  positive definite matrix. When the scalar measures $\mu_{i,j}$, $i,j=1,\dots,m$, are  absolutely continuous with respect to the Lebesgue
measure on the circle $\operatorname{d}\theta$, according to the Radon--Nikodym theorem, it can be always expressed using
complex weight (density or Radon--Nikodym derivative of the measure) functions $w_{i,j}$, $i,j=1,\dots,m$, so that
$\operatorname{d}\mu_{i,j}(\theta)=w_{i,j}(\theta)\operatorname{d}\theta$,
$\theta\in[0,2\pi)$.
If, in addition, the matrix measure $\mu$ is Hermitian and positive definite then the matrix
$(w_{i,j}(\theta))$
is a positive definite Hermitian matrix.
 For the sake of notational simplicity we will use, whenever it
is convenient,
 the complex notation $\operatorname{d}\mu(z) = \operatorname{i}\operatorname{e}^{\operatorname{i}\theta}\operatorname{d}
\mu(\theta)$.\

The moments of the matrix measure $\mu$ are
\begin{align*}
c_n&:= \frac{1}{2\pi}\oint_\T z^{-n} \frac{\d\mu(z)}{\operatorname{i}z}=\frac{1}{2\pi}\int_0^{2\pi} \Exp{-in\theta} \d\mu(\theta)\in\mathbb M_m,
\end{align*}
while  the Fourier series of the measure is
\begin{align}\label{moment}
F_{\mu}(u)&:=\sum_{n=-\infty}^\infty c_n u^n,
\end{align}
that for  absolutely continuous measures $\d\mu(\theta)=w(\theta)\d\theta$ satisfies
$F_\mu(\theta)=w(\theta)$.
 Let $D(0;r,R)=\{z\in\C: r<|z|<R\}$ denote the annulus around $z=0$ with interior and exterior radii $r$ and $R$, $R_{ij,\pm}:=\big(\limsup\limits_{n\to\infty}\sqrt[n]{|c_{ij,\pm n}|}\big)^{\mp 1}$ and $R_+=\min\limits_{i,j=1,\dots, m}R_{ij,+}$ and $R_{ij,-}=\max\limits_{i,j=1,\dots, m} R_{ij,-}$. Then, according to the Cauchy--Hadamard theorem the series $F_{\mu}(z)$ converges uniformly in any compact set $K$, $K\subset D(0;R_-,R_+)$.

The space $\Lambda_{m,[p,q]}:=\mathbb{M}_m\{\mathbb{I}z^{-p},\mathbb{I}z^{-p+1},\dots,\mathbb{I}z^q\}$ (where $\mathbb{I} \in \mathbb{M}_m$ is the identity matrix) of complex
 Laurent polynomials with $m\times m$ matrix coefficients and the corresponding restrictions on their degrees is a $\mathbb M_m$ free module of rank $p+q+1$. We denote by  $L_{\text{pol}}\mathbb{M}_m$   the infinite set of Laurent matrix polynomials or polynomial loops in $\mathbb{M}_m$.

Given a matrix measure $\mu$  we introduce
the following  left and right  matrix valued sesquilinear forms in the loop space $L\mathbb M_m$ considered as left and right modules for the ring $\mathbb M_m$, respectively,
\begin{align}\label{prodL}
\langle\!\langle f ,g\rangle\!\rangle_{L}:=\oint_\mathbb{T}g(z)\frac{\d\mu(z)}{\operatorname{i}z}f(z)^\dagger\in\mathbb M_m,
\\\label{prodR}
\langle\!\langle f ,g\rangle\!\rangle_{R}:=\oint_\mathbb{T}f(z)^\dagger\frac{\d\mu(z)}{\operatorname{i}z} g(z)\in\mathbb M_m.
\end{align}
The sesquilinearity of these forms means that the following two properties hold
\begin{enumerate}
\item $\Langle f_1+f_2, g\Rangle_H=\Langle f_1, g\Rangle_H+\Langle f_2, g\Rangle_H$ and $\Langle f, g_1+g_2\Rangle_H=\Langle f, g_1\Rangle_H+\Langle f,  g_2\Rangle_H$ for all $f,f_1,f_2,g,g_1,g_2\in L\mathbb M_m$ and $H=L,R$.
\item $\Langle mf,g\Rangle_L=\Langle f,g \Rangle_L m^\dagger$,  $\Langle f,m g\Rangle_L = m \Langle f,g\Rangle_L$,  $\Langle fm,g\Rangle_R=m^\dagger\Langle f,g \Rangle_R$ and $\Langle f,gm\Rangle_R =  \Langle f,g\Rangle_R m$, for all $f,g\in L\mathbb M_m$ and $m\in\mathbb M_m$.
\end{enumerate}
Moreover, if the matrix measure is Hermitian then so are  these  forms; i.e.,
\begin{align*}
 \langle\!\langle f ,g\rangle\!\rangle_{H}^{\dagger}&=\langle\!\langle g ,f\rangle\!\rangle_{H}, & H&=L,R.
\end{align*}
Actually, from these sesquilinear forms, for a positive definite Hermitian measure we can derive the corresponding scalar products
\begin{align*}
\langle f ,g\rangle_H^{\dagger}=\langle f ,g\rangle_H&:= \operatorname{Tr}[\langle\!\langle f ,g\rangle\!\rangle_H],&  \Arrowvert f \Arrowvert^2_H&=\langle f ,f \rangle_H,&H&=\,\,L,R,
\end{align*}
and corresponding  Hilbert spaces $\mathcal{H}^H$ with a norm --of Frobenius type-- given by
\begin{align*}
  \Arrowvert f \Arrowvert_H&=+\sqrt{\langle f ,f \rangle_H},&H&=L,R.
\end{align*}
A set $\{p_l\}_{l=0}^{N}\subset\mathcal{H}^H$, $H=L,R$, is an orthogonal set if and only if
\begin{align*}
\langle\!\langle p_l^H ,p_j^H\rangle\!\rangle_{K}&=\delta_{ij}h_j,  &  &h_j\in \mathbb{M}_m.
\end{align*}

\subsection{On the content of the paper}

In  previous papers we have approached the study of the link between orthogonality and integrability within an algebraic/group theoretical point of view. Our keystone relies on the fact that a number of facets of orthogonality and integrability can be described with the aid of the Gauss--Borel factorization of an infinite matrix. This approach was applied in \cite{afm-2} for the analysis of multiple orthogonal polynomials of mixed
type, allowing for an algebraic proof of the Christoffel--Darboux formula, alternative to the analytic one, based on the Riemann--Hilbert problem (and constrained therefore by convenient analytic conditions) given in \cite{kuijlaars}. This approach was also used successfully in \cite{carlos} in where a CMV ordering of the Fourier basis induced, for a given measure on the unitary circle, a moment matrix whose Gauss--Borel factorization leads to OLPUC. Recursion relations and Christoffel--Darboux formula appeared also in a straightforward manner.   Also  continuous and discrete deformations, and $\tau$-function theory was  extended to the circular case under the suitable choice of moment matrices and shift operators. In this last paper we only requested to the measure to be quasi-definite, condition that implies the existence of the Gauss--Borel factorization. Let us mention that we have applied this method in the finding of Christoffel--Darboux type formulae in other situations, see \cite{araznibarreta,
carlos2}.

In this paper we consider  two semi-infinite block  matrices, whose coefficients (matrices in $\C^{m\times m}$)  are left and right matrix moments, ordered in a   Cantero--Morales--Vel\'{a}zquez style, of a matrix measure on the circle. The corresponding  block Gauss-Borel factorization of these CMV block moment matrices leads to MOLPUC. To be more precise,  we get  the right and left versions of two biorthogonal families of matrix Laurent polynomials and corresponding Szeg\H{o} polynomials. When the matrix measure is Hermitian  these two families happen to be proportional resulting in two families of MOLPUC. Following \cite{miranian, Cafasso} we express them as Schur complements of bordered truncated moment matrices. We also prove,  in an algebraic manner using the Gauss--Borel factorization, the five term recursion relations and the Christoffel--Darboux formula. Let us stress that in this paper we introduce an intertwining operator $\eta$ not used in \cite{carlos}  that clarifies the appearance of
reciprocal polynomials and simplifies the algebraic proofs. The recursion relations indicate which deformations of the quasi-definite matrix measure lead to integrable systems of Toda type. Thus, we discuss  the following elements: wave and adjoint wave functions, Lax and Zakharov--Shabat equations, bilinear equations and discrete flows --connected with Darboux transformations--. In this context we find a generalization of the  matrix  Cafasso's extension of the Toeplitz lattice for the Verblunsky coefficients of Szeg\H{o} polynomials. The Cafasso flows correspond to what we call total flows, which are only a part of the integrable flows associated with MOLPUC. We unsuccessfully tried  to get a matrix $\tau$ theory, but despite this failure we get  interesting byproducts. We analyze the role of Miwa shifts in this context and, as a collateral effect, nicely connect them with the Christoffel--Darboux kernels. These formulae suggest a link of these kernels with the Cauchy propagators that in the Grassmannian $\bar\partial$ approach to multicomponent KP hierarchy was used in \cite{mio1,mio2}. This identification allows us to give in Theorem \ref{elteorema} expressions of the MOLPUC in terms of products of their Miwa shifted and non shifted  quasi-norms. Despite that these expressions lead to the $\tau$-function representation  in the scalar case this is not the case within the matrix context.

Let us mention that the  submodules of matrix Laurent polynomials considered in this paper have the higher and lower powers constrained  to be of some particular form, implied by the chosen CMV ordering. In \cite{carlos} this limitation was overcome by the introduction of extended CMV orderings of the Fourier basis, which allowed for general subpaces of Laurent polynomials. A similar procedure can be performed in this matrix situation; but, as its development follows very closely the ideas of \cite{carlos} we prefer to avoid its inclusion here.

The layout of this paper is as follows.  \S \ref{Sec-matrix orthogonality} is devoted to orthogonality theory, in particular in \S\ref{Sec-moment} we consider the left and right block CMV moment matrices and perform corresponding block Gaussian factorizations in \S\ref{Sec-MOLPUC},  getting the associated families of right and left MOLPUC and matrix Szeg\H{o} polynomials and their biorthogonality relations. We also get the recursion relations and Schur complement expressions of them in terms of bordered truncations of the moment matrices. Then, in \S \ref{Sec-second} we introduce the matrix second kind functions that are connected with the Fourier series of the measure and that will be relevant later on for the adjoint Baker functions.
The reconstruction of the recursion relations from the Gauss--Borel factorizations is performed in \S \ref{Sec-recursions}; the Christoffel--Darboux formulae for this non Abelian scenario are given in  \S \ref{Sec-CD}. Observe that in this case the projection operators are projectors in a module over the ring $\C^{m\times m}$, that in the Hermitian definite positive situation lead to orthogonal projections in the standard geometrical sense. The integrability aspects are treated in \S \ref{Sec-integrability}. Given adequate deformations of the moment matrices we find wave functions, Lax equations and Zakharov--Shabat equations in \S\ref{Sec-Toda-0}; here we also consider a generalization of the Cafasso's Toeplitz lattice and the bilinear equations formulation of the hierarchy.
Finally, we extend to this matrix context the discrete flows for the Toeplitz lattice, intimately related to Darboux transformations in \S \ref{Sec-discrete} and also derive the bilinear equations fulfilled by the MOLPUC in \S \ref{Sec-bilinear}. Finally, in \S \ref{Sec-Miwa} we consider the action of Miwa transformations and get the previously mentioned results. We conclude the paper with a series of appendices that serve as support of certain sections.

Finally, let us stress that this paper is not just an extension of the results of \cite{carlos} to the matrix realm but we also  have introduced important elements not discussed there, which also hold in that scalar case, as the $\eta$ operator, a different proof of the Chirstoffel--Darboux formula with no need of associated polynomials and new relations between Christoffel--Darboux kernels and Miwa shifted MOLPUC.

\section{Matrix orthogonality  and block Gauss--Borel factorization}\label{Sec-matrix orthogonality}
In this section, inspired by the CMV construction \cite{CMV} and the previous work \cite{carlos}, for a given matrix measure we introduce an appropriate block moment matrix, that when factorized as a product of lower and upper block  matrices, gives a set of biorthogonal matrix Laurent polynomials on the unit circle. This Borel--Gauss factorization problem also allows us to derive the recursion relations and the Christoffel--Darboux theory.
\subsection{The CMV right and left moment matrices for quasi-definite matrix measures}\label{Sec-moment}
The following $m\times m$ matrix valued vectors will be relevant in the construction of biorthogonal families of MOLPUC
\begin{definition} The CMV vectors are given by
\begin{align*}
\chi_1(z):=&(\mathbb{I},0,\mathbb{I}z,0,\mathbb{I}z^{2},\ldots)^{\top},\\
\chi_2(z):=&(0,\mathbb{I} ,0,\mathbb{I}z,0,\mathbb{I}z^2,\ldots)^{\top},\\
\chi_a^{*}(z):=&z^{-1}\chi_a(z^{-1}), &  a =1,2,\\
\chi(z):=&\chi_1(z)+\chi_2^{*}(z)=(\mathbb{I},\mathbb{I}z^{-1},\mathbb{I}z,\mathbb{I}z^{-2},\mathbb{I}z^{2},\ldots)^{\top}.
\end{align*}
\end{definition}
In the sequel the matrix $\chi^{(l)}$ will denote the $l$-th component of the matrix vector $\chi$
\begin{align*}
\chi  =(\chi^{(0)},\chi^{(1)},\chi^{(2)},\dots)^\top.
\end{align*}

.
\begin{definition} The CMV left and right moment matrices of the measure $\mu$ are
\begin{align}\label{defgL}
g^L&:=\oint_\mathbb{T}\chi(z)\,\frac{\d\mu(z)}{\operatorname{i}z}\,\left(\chi(z)\right)^\dagger=2\pi\begin{pmatrix}
   c_0 & c_{-1} & c_1 & c_{-2} & \dots\\
   c_{1} & c_0 & c_2 & c_{-1} &\dots\\
   c_{-1} & c_{-2} & c_0 & c_{-3}&\dots\\
      c_{2} & c_{1} & c_{3} & c_{0}&\dots\\
  \vdots&\vdots& \vdots& \vdots & \ddots
\end{pmatrix},\\
\label{defgR}
g^R&:=\oint_\mathbb{T}(\chi(z)^\top)^\dagger\,\frac{\d\mu(z)}{\operatorname{i}z}\,\chi(z)^\top=2\pi\begin{pmatrix}
   c_0 & c_1 & c_{-1} & c_2 & \dots\\
   c_{-1} & c_0 & c_{-2} & c_1 &\dots\\
   c_1 & c_2 & c_0 & c_3&\\
      c_{-2} & c_{-1} & c_{-3} & c_{0}&\dots\\
 \vdots&\vdots& \vdots& \vdots & \ddots
\end{pmatrix}.
\end{align}
\end{definition}
Notice that when $\d\mu(\theta)$ is Hermitian so are the moment matrices $g ^L$ and $g^R$.\\

In the scalar case \cite{carlos} the only requirement that the moment matrix needs to meet is to be  Gaussian factorable; i.e., all the principal minors of the matrix are request to be not degenerated. The measure from which this moment matrix
is constructed receives the name of quasi-definite measure. This condition is related to the existence of \emph{biorthogonal} polynomials
of all degrees --also called non-triviality of the measure--.
 In the matrix case the requirement is a bit more relaxed.
\begin{definition}
 The matrix measure $\mu$ is said to be  \emph{quasi-definite} if its truncated moment matrices satisfy
\begin{align*}
 \det\big((g^H)^{[l]}\big)\neq 0 \text{ for } H=R,L \text{ and } l=1,2,3,\dots
\end{align*}
\end{definition}

Notice that $(g^H)^{[l]} \in \mathbb{M}_{ml}$; a quite different situation from the scalar case in which all the principal minors had to be
non degenerate, while in the matrix case only the $ml$-order principal minors should meet this requirement. Actually, this is the
only restriction --besides having compact support on $\T$-- that from hereon the matrix measures must satisfy, since
when this condition holds
\begin{pro}\label{Pro-existence LU}
 The moment matrices $g^H$, $H=L,R$,  of a matrix quasi-definite measure $\mu$ admit a block Gauss--Borel factorization.
\end{pro}
\begin{proof}
See Appendix \ref{proofs}.
\end{proof}

\subsubsection{The generalized matrix Szeg\H{o} polynomials}
\begin{definition}\label{GSzg}Given a matrix quasi-definite measure $\mu$   the set of  monic matrix polynomials
$\{P_{i,l}^L\}_{l=0}^{\infty}$, $\{P_{i,l}^R\}_{l=0}^\infty$, $i=1,2$, with $\deg P_{i,l}^H=l$, $H=L;R$, satisfying
\begin{align*}
  \langle\!\langle z^j \mathbb I ,P_{1,l}^L(z)\rangle\!\rangle_{L}&=\oint_\T P_{1,l}^L(z)\frac{d\mu(z)}{\operatorname{i}z}z^{-j}=0, & j=0,\dots,l-1,\\
  \langle\!\langle P_{2,l}^L(z),z^j \mathbb I\rangle\!\rangle_{L}&=\oint_\T z^{j} \frac{d\mu(z)}{\operatorname{i}z}[P_{2,l}^L(z)]^\dagger=0, & j=0,\dots,l-1,\\
  \langle\!\langle P_{2,l}^R(z),z^{j} \mathbb I\rangle\!\rangle_{R}&=\oint_\T [P_{2,l}^R(z)]^\dagger \frac{d\mu(z)}{\operatorname{i}z}z^{j}=0, & j=0,\dots,l-1,\\
  \langle\!\langle z^{j} \mathbb I, P_{1,l}^R(z)\rangle\!\rangle_{R}&=\oint_\T z^{-j} \frac{d\mu(z)}{\operatorname{i}z}P_{1,l}^R(z)=0, & j=0,\dots,l-1,
\end{align*}
are said to be   Szeg\H{o} polynomials.
\end{definition}

\begin{pro}
 The matrix Szeg\H{o} polynomials introduced in Definition \ref{GSzg} for the quasi-definite situation  exist and are unique. Moreover, there exist matrices $h^H_r\in\mathbb M_m$, $H=L,R$, such that  the biorthogonality conditions are fulfilled
\begin{align*}
 \delta_{r,j}h^H_r&:=\langle\!\langle P_{2,r}^H,P_{1,j}^H \rangle\!\rangle_H ,& H&=R,L.
\end{align*}
\end{pro}

Now we introduce the matrix extension of the  Verblunsky coefficients.
\begin{definition}
The Verblunsky matrices of a matrix quasi-definite measure are
\begin{align*}
\alpha_{i,l}^H&:= P_{i,l}^H(0), & i&=1,2, & l&=1,2,3,... & H&=L,R,
\end{align*}
 and the reciprocal or reversed Szeg\H{o} matrix polynomials are given by
\begin{align*}
 ( P_l^H)^*(z)&:=z^{l}\big(P^H_l\big(\bar z^{-1}\big)\big)^\dagger, & H&=L,R.
\end{align*}
\end{definition}

Notice that in the Hermitian positive definite case the matrices $h^H_l$, $H=L,R$, $l=0,1,2,\dots$,  can be interpreted as a kind of ``matrix-valued norms" for the matrix Szeg\H{o} polynomials, as the square-root of their traces is a norm  indeed.

\subsection{The CMV matrix Laurent polynomials}\label{Sec-MOLPUC}
We consider now the $m\times m$ block $LU$ factorization of the  moment matrices \eqref{defgL} and \eqref{defgR}; in fact, there are two block  Gauss--Borel factorizations, for both the right and left moment matrices, to consider
\begin{align}\label{eq:fac1}
g^L&:=S_1^{-1} D^L \widehat{S_2}=S_1^{-1}S_2,   &S_1&\in \mathscr L,&S_2&\in\mathscr U, & \widehat{S_2}&\in \widehat{\mathscr{U}},& D_L& \in \mathscr{D},\\
g^R&:=Z_2 D^R \widehat{Z_1}^{-1}=Z_2Z_1^{-1}, &Z_2&\in \mathscr L,&Z_1&\in\mathscr U, & \widehat{Z_1}&\in \widehat{\mathscr{U}},& D_R& \in \mathscr{D}.\label{eq:fac2}
\end{align}
For the entries of the block diagonal matrices we use the  notation
\begin{align}\label{intro-tau}
 D^H&=\diag(D^H_0,D^H_1,\dots), & H&=L,R.
\end{align}

The reader should notice that in the Hermitian case the two normalized matrices of the factorization are related
\begin{align}\label{hermitian}
 S_1^\dagger&=\widehat{S_2}^{-1},  &   Z_2^\dagger&=\widehat{Z_1}^{-1},
\end{align}
and the block diagonal matrices are Hermitian; $\big(D^H\big)^\dagger=D^H$, $H=L,R$.

\begin{definition}We introduce the following partial CMV matrix Laurent polynomials
\begin{align*}
 \phi_{1,1}^L&:=S_1\chi_1(z), & \phi_{1,2}^L&:=S_1\chi_2^*(z),& \phi_{2,1}^L&:=\big(S_2^{-1}\big)^\dagger\chi_1(z), & \phi_{2,2}^L&:=\big(S_2^{-1}\big)^\dagger\chi_2^*(z),\\
\phi_{1,1}^R&:=\chi_1^\top(z)Z_1, &  \phi_{1,2}^R&:=[\chi_2^*]^\top(z)Z_1,& \phi_{2,1}^R&:=\chi_1^\top(z)\big(Z_2^{-1}\big)^\dagger, & \phi_{2,2}^R&:=[\chi_2^*]^\top(z)\big(Z_2^{-1}\big)^\dagger,
\end{align*}
and CMV matrix Laurent polynomials
\begin{align}
\phi_1^L&:=\phi_{1,1}^L+\phi_{1,2}^L=S_1\chi(z), &  \phi_2^L&:=\phi_{2,1}^L+\phi_{2,2}^L=\big(S_2^{-1}\big)^\dagger\chi(z),\label{laurent1}\\
\phi_1^R&:=\phi_{1,1}^R+\phi_{1,2}^R=\chi^\top(z)Z_1, &  \phi_2^R&:=\phi_{2,1}^R+\phi_{2,2}^R=\chi^\top(z)\big(Z_2^{-1}\big)^\dagger.\label{laurent2}
\end{align}

\end{definition}
Notice that these semi-infinite vectors with matrix coefficients $(\varphi^H_{j})^{(l)}(z)$, $l=0,1,\dots$, can be written as
\begin{align*}
\phi_j^L&=:
\begin{pmatrix}
(\varphi^L_{j})^{(0)}(z)\\
(\varphi^L_{j})^{(1)}(z) \\
\vdots
\end{pmatrix},&
\phi_j^R&=:((\varphi^R_{j})^{(0)}(z),
(\varphi^R_{j})^{(1)}(z),\dots),& j&=1,2.
\end{align*}
For the Hermitian case we have
\begin{align}
  \label{laurent-hermitian}
  (\varphi^L_{2})^{(l)}(z)&=(D_l^L)^{-1}(\varphi^L_{1})^{(l)}(z), &
   (\varphi^R_{2})^{(l)}(z)&=(\varphi^R_{1})^{(l)}(z)D_l^R, & l&=0,1,\dots.
\end{align}

\subsubsection{Biorthogonality}

From the Gaussian factorization, whose existence is ensured for quasi-definite matrix measures, we infer  that these
matrix Laurent polynomials satisfy  biorthogonal type relations.
\begin{theorem}\label{bio}
The matrix Laurent polynomials  $\big\{(\varphi^H)_{1}^{(l)}\big\}_{l=0}^\infty$ and $\big\{(\varphi^H)_{2}^{(l)}\big\}_{l=0}^\infty$,  $H=L,R$, introduced in \eqref{laurent1} and \eqref{laurent2} are biorthogonal on the unit circle
\begin{align}\label{ortCMV}
\langle\!\langle (\varphi^H_2)^{(j)},(\varphi^H_1)^{(k)}\rangle\!\rangle_{H}&= \mathbb{I}\delta_{j,k}, & H&=L,R,&j,k&=0,1,\dots
\end{align}
\end{theorem}
\begin{proof}It is straightforward to check that
\begin{align*}
\oint_\mathbb{T}\phi^L_1 (z)\frac{\d\mu(z)}{\operatorname{i}z} (\phi^L_2(z))^\dagger &=S_1\left(\oint_\mathbb{T}\chi(z) \frac{\d\mu(z)}{\operatorname{i}z} \chi(z)^\dagger\right)S_2^{-1}=S_1g ^LS_2^{-1}=\mathbb{I},\\
\oint_\mathbb{T}(\phi^R_2(z))^\dagger \frac{\d\mu(z)}{\operatorname{i}z} \phi^R_1(z) &=Z_2^{-1}\left(\oint_\mathbb{T}\overline{\chi(z)} \frac{\d\mu(z)}{\operatorname{i}z} \chi(z)^\top \right)Z_1=Z_2^{-1}g^RZ_1=\mathbb{I}.
\end{align*}
\end{proof}

In order to relate the CMV matrix Laurent polynomials to the Szeg\H{o} polynomials,
we rewrite the \emph{quasi-orthogonality }conditions from Theorem \ref{bio}
\begin{align}\label{ort}
\begin{aligned}
\oint_{\T} (\varphi^L_{1})^{(2l)} (z) \frac{\d\mu(z)}{\operatorname{i}z} z^{-k}&=0, &k=-l,\dots,l-1, \\
\oint_{\T} (\varphi^L_{1})^{(2l+1)}(z) \frac{\d\mu(z)}{\operatorname{i}z} z^{-k}&=0, &k=-l,\dots,l,\\
\oint_{\T}  z^{k} \frac{\d\mu(z)}{\operatorname{i}z} [(\varphi^L_{2})^{(2l)}(z)]^\dagger&=0, &k=-l,\dots,l-1,\\
\oint_{\T}  z^{k} \frac{\d\mu(z)}{\operatorname{i}z} [(\varphi^L_{2})^{(2l+1)}(z)]^\dagger&=0, &k=-l,\dots,l,
\end{aligned}\\\notag\\\label {ort1}
\begin{aligned}
\oint_{\T} [(\varphi^R_{2})^{(2l)}(z)]^\dagger \frac{\d\mu(z)}{\operatorname{i}z} z^{k}&=0, &k=-l,\dots,l-1,\\
\oint_{\T} [(\varphi^R_{2})^{(2l+1)}(z)]^\dagger \frac{\d\mu(z)}{\operatorname{i}z} z^{k}&=0, &k=-l,\dots,l,\\
\oint_{\T}  z^{-k} \frac{\d\mu(z)}{\operatorname{i}z} (\varphi^R_{1})^{(2l)}(z)&=0, &k=-l,\dots,l-1,\\
\oint_{\T}  z^{-k} \frac{\d\mu(z)}{\operatorname{i}z} (\varphi^R_{1})^{(2l+1)}(z)&=0, &k=-l,\dots,l.
\end{aligned}
\end{align}

\begin{pro}
  For a quasi-definite matrix measure $\mu$ the matrix Szeg\H{o} polynomials and the CMV matrix Laurent polynomials are related in the following way for the left case
  \begin{align}\label{CMVSi}\begin{aligned}
z^l(\varphi^L_{1})^{(2l)}(z)&= P_{1,2l}^{L}(z),\\
z^{l+1}(\varphi^L_{1})^{(2l+1)}(z)&= (P_{2,2l+1}^{R})^*(z),\\
z^l(D^L_{2l})^\dagger(\varphi^L_{2})^{(2l)}(z)&= P_{2,2l}^{L}(z),\\
z^{l+1}(D^L_{2l+1})^\dagger(\varphi^L_{2})^{(2l+1)}(z)&= (P_{1,2l+1}^{R})^*(z),\end{aligned}
  \end{align}
  and
\begin{align}\begin{aligned}
z^l(\varphi^R_{2})^{(2l)}(z)&= P_{2,2l}^{R}(z),\\
z^{l+1}(\varphi^R_{2})^{(2l+1)}(z)&= (P_{1,2l+1}^{L})^*(z),\\
z^l(\varphi^R_{1})^{(2l)}(z)D^R_{2l}&= P_{1,2l}^{R}(z),\\
\label{CMVSf}z^{l+1}(\varphi^R_{1})^{(2l+1)}(z)D^R_{2l+1}&= (P_{2,2l+1}^{L})^*(z)\end{aligned}
  \end{align}
  for the right case.
\end{pro}
\begin{proof}
  Taking the differences between the RHS and LHS of the equalities we get matrix polynomials, of degree $d=2l-1,2l$, that when paired via $\langle\!\langle\cdot,\cdot\rangle\!\rangle_H$, $H=L,R$, to all the powers $z^j$, $j=0,\dots,q$ cancels. Therefore, as we have a quasi-definite matrix measure, with moment matrices having non-null principal block minors, the only possibility for the difference is to be $0$.
\end{proof}
The last identifications together with \eqref{GSzg} define some of the entries
of the Gaussian factorization matrices.
\begin{pro} \label{pro:quasinorms}
The \emph{matrix quasi-norms} $h^H_k$ introduced in  Definition \ref{GSzg}  and the coefficients $D^H_k$ given in \eqref{intro-tau} satisfy
\begin{align*}
h^L_{2l}&=D^L_{2l},  &  h^L_{2l+1}&=D^R_{2l+1},\\
h^R_{2l}&=D^R_{2l}    &   h^R_{2l+1}&=D^L_{2l+1}.
\end{align*}
\end{pro}
For the first non-trivial block diagonal of the factors in the Gaussian--Borel factorization  we get
\begin{pro}\label{faclosalpha}The matrices of the block $LU$ factorization can be written more explicitly in terms of the Verblunsky coefficients as follows
\begin{align*}
S_1&=\begin{pmatrix}
  \mathbb{I}& 0 & 0 &0&0&\hdots\\
  [\alpha_{2,1}^R]^\dagger&\mathbb{I}& 0 &0&0&\hdots\\
  *&\alpha_{1,2}^L&\mathbb{I}&0&0&\hdots\\
  *&*& [\alpha_{2,3}^R]^\dagger &\mathbb I& 0&\hdots\\
  * & *  &*&\alpha_{1,4}^L &\mathbb{I}&& \\
  \vdots&\vdots&\vdots&&\ddots&\ddots
\end{pmatrix},&
\widehat{S_2}^{-1}&=\begin{pmatrix}
  \mathbb{I}& \alpha_{1,1}^R & * &*&*&\dots\\
  0 &\mathbb{I}& [\alpha_{2,2}^L]^\dagger&*&*&\dots\\
0&0&   \mathbb{I}& \alpha_{1,3}^R & * &\dots\\
0&0&0& \mathbb{I}& [\alpha_{2,4}^L]^\dagger &\\
0&0&0&0&\mathbb{I}&\ddots\\
\vdots&\vdots&\vdots&\vdots&&\ddots&
\end{pmatrix},\\
Z_2^{-1}&=\begin{pmatrix}
  \mathbb{I}& 0 & 0 &0&0&\dots\\
  \alpha_{1,1}^L&\mathbb{I}& 0 &0&0&\dots\\
  *&[\alpha_{2,2}^R]^\dagger&\mathbb{I}&0&0&\dots\\
    *&*&\alpha_{1,3}^L&\mathbb{I}& 0 &\dots\\
     *&*& *&[\alpha_{2,4}^R]^\dagger&\mathbb{I}&\\
  \vdots&\vdots& \vdots&& \ddots&\ddots\\
\end{pmatrix}, &
\widehat{Z_1}&=\begin{pmatrix}
  \mathbb{I}& [\alpha_{2,1}^L]^\dagger & * &*&*&\dots\\
  0 &\mathbb{I}& \alpha_{1,2}^R&*&*&\dots\\
  0&0& \mathbb{I}& [\alpha_{2,3}^L]^\dagger & * &\dots\\
0&0&    0 &\mathbb{I}& \alpha_{1,4}^R&\\
\vdots&\vdots&\vdots&\vdots&\ddots&\ddots
\end{pmatrix}.
\end{align*}
\end{pro}
This gives the following stucture for the MOLPUC
\begin{pro}
\begin{enumerate}
  \item  The MOLPUC are of the form
  \begin{align*}
(\phi^L_1)^{(2l)}&=\alpha_{1,2l}^Lz^{-l}+\cdots+z^l, & (\phi^L_1)^{(2l+1)}&=z^{-l-1}+\cdots +(\alpha^R_{2,2l+1})^\dagger z^l,\\
(\phi^L_2)^{(2l)}&=\big((h^L_{2l})^\dagger\big)^{-1}\big(\alpha^L_{2,2l}z^{-l}+\cdots+z^l\big), &
(\phi^L_2)^{(2l+1)}&=\big((h^R_{2l+1})^\dagger\big)^{-1}\big(z^{-l-1}+\cdots+(\alpha^R_{1,2l+1})^\dagger z^{l}\big),\\
(\phi^R_1)^{(2l)}&=\big(\alpha^R_{1,2l}z^{-l}+\cdots+z^l\big)(h^R_{2l})^{-1}, &
(\phi^R_1)^{(2l+1)}&=\big(z^{-l-1}+\cdots+(\alpha^L_{2,2l+1})^\dagger z^{l}\big)(h^L_{2l+1})^{-1},\\
(\phi^R_2)^{(2l)}&=\alpha_{2,2l}^Rz^{-l}+\cdots+z^l, & (\phi^R_2)^{(2l+1)}&=z^{-l-1}+\cdots +\alpha^L_{1,2l+1} z^l.
  \end{align*}
  \item The ``quasi-norms" and the MOLPUC fulfill
  \begin{align}\label{norms1}
    h^R_{2l+1}&=\oint_\T (\phi_1^L)^{(2l+1)}(z)\frac{\d\mu (z)}{\operatorname{i}z}z^{l+1}, &
    h^L_{2l}&=\oint_\T (\phi_1^L)^{(2l)}(z)\frac{\d\mu (z)}{\operatorname{i}z}z^{-l},\\\notag
    h^R_{2l}&=\oint_\T   \big((\phi_2^R)^{(2l)}(z)\big)^\dagger\frac{\d\mu (z)}{\operatorname{i}z}z^{l}, &
    h^L_{2l+1}&=\oint_\T   \big((\phi_2^R)^{(2l+1)}(z)\big)^\dagger\frac{\d\mu (z)}{\operatorname{i}z}z^{-l-1}.
  \end{align}
\end{enumerate}
\end{pro}
\begin{proof}
\begin{enumerate}
  \item Use \eqref{eq:fac1}, \eqref{eq:fac2} and Propositions \ref{pro:quasinorms} and \ref{faclosalpha}.
  \item Consider the biorthogonality \eqref{ortCMV} together with the explicit expressions of the first item in this Proposition and orthogonality relations \eqref{ort} and \eqref{ort1}.
\end{enumerate}
\end{proof}

Recalling \eqref{hermitian} we conclude from Proposition \ref{faclosalpha} that in the Hermitian context we have
\begin{align*}
  \alpha^H_{1,l}&=\alpha_{2,l}^H, & H&=L,R, & l&=0,1,\dots,\\
( D_l^H)^\dagger&=D^H_l, & H&=L,R, & l&=0,1,\dots.
\end{align*}
It is not difficult to see comparing the previous result with the proof of the Gaussian factorization \eqref{Gss} that in terms of Schur complements we have
\begin{pro}
\begin{enumerate}
\item The matrices $D^H_l\in\C^{m\times m}$, $H=L,R$, $l=0,1,\dots$, from the diagonal block of the bock LU factorization can be written as the following Schur complements
\begin{align}\label{com}
 D^H_l &=(g^H)^{[l+1]}\diagup (g^H)^{[l]},  & H&=L,R & l&=0,1,\dots.
\end{align}
\item The  Verblunsky matrices can be expressed as
\begin{align*}
\alpha_{1,2k}^L&=-\sum_{i=0}^{2k-1}(g^L)_{2k,i}(((g^{L})^{[2k]})^{-1})_{i,2k-1}, &
[\alpha_{2,2k+1}^R]^\dagger&=-\sum_{i=0}^{2k}(g^L)_{2k+1,i}(((g^{L})^{[2k+1]})^{-1})_{i,2k},\\
[\alpha_{2,2k}^R]^\dagger&=-\sum_{i=0}^{2k-1}(g^R)_{2k,i}(((g^R)^{[2k]})^{-1})_{i,2k-1},&
\alpha_{1,2k+1}^L&=-\sum_{i=0}^{2k}(g^R)_{2k+1,i}(((g^{R})^{[2k+1]})^{-1})_{i,2k},\\
[\alpha_{2,2k}^L]^\dagger&=-\sum_{i=0}^{2k-1}(((g^{L})^{[2k]})^{-1})_{2k-1,i}(g^L)_{i,2k},&
\alpha_{1,2k+1}^R&=-\sum_{i=0}^{2k}([(g^{L})^{[2k+1]})^{-1})_{2k,i}(g^L)_{i,2k+1},\\
\alpha_{1,2k}^R&=-\sum_{i=0}^{2k-1}(([(g^R)^{[2k]})^{-1})_{2k-1,i}(g^R)_{i,2k}, &
[\alpha_{2,2k+1}^L]^\dagger&=-\sum_{i=0}^{2k}(((g^R)^{[2k+1]})^{-1})_{2k,i}(g^R)_{i,2k+1}.
\end{align*}
\end{enumerate}
\end{pro}

\subsubsection{Alternative ways to express the CMV matrix Laurent polynomials}\label{alternative}
For later use we now present  some alternative expressions for the MOLPUC  $(\varphi^H_{i})^{(l)}(z)$, $H=L,R$, $l=0,1,\dots$  in terms of Schur complements of bordered truncated matrices
\begin{lemma}\label{lemma-alternative}
The next expressions hold true
 \begin{align}\label{sc1}
& \begin{aligned}
   (\varphi^L_{1})^{(l)}(z)&=(S_2)_{ll}\begin{pmatrix}
    0 &0 &\dots &0& \mathbb{I}
  \end{pmatrix}((g^L)^{[l+1]})^{-1}\chi^{[l+1]}
 =\chi^{(l)}-\begin{pmatrix}
    (g^L)_{l,0} & (g^L)_{l,1} & \cdots &(g^L)_{l,l-1}
  \end{pmatrix}((g^L)^{[l]})^{-1}\chi^{[l]}\\
&=\operatorname{SC}\left(\begin{BMAT}{cccc|c}{cccc|c}
       (g^L)_{0,0} & (g^L)_{0,1} & \dots & (g^L)_{0,l-1} & \chi(z)^{(0)}\\
       (g^L)_{1,0} & (g^L)_{1,1} & \dots & (g^L)_{1,l-1} & \chi(z)^{(1)}\\
                   & \vdots      &       &               &  \vdots\\
       (g^L)_{l-1,0} & (g^L)_{l-1,1} & \dots & (g^L)_{l-1,l-1} & \chi(z)^{(l-1)}\\
      (g^L)_{l,0} & (g^L)_{l,1} & \dots & (g^L)_{l,l-1} & \chi(z)^{(l)}\\
      \end{BMAT}\right),
 \end{aligned}\\
&\begin{aligned}
\,[(\varphi^L_{2})^{(l)}(z)]^\dagger&=(\chi^{[l+1]})^\dag
  ((g^L)^{[l+1]})^{-1}\begin{pmatrix}
    0 \\0\\\vdots \\0\\ \mathbb{I}
  \end{pmatrix}=
  \Big((\chi^{(l)})^{\dag}-(\chi^{[l]})^\dag ((g_{L})^{[l]})^{-1}\begin{pmatrix}
    (g^L)_{0,l}\\(g^L)_{1,l}\\\vdots\\(g^L)_{l-1,l}
  \end{pmatrix}\Big)(D^L)_{l}\\
  &=\operatorname{SC}\left(\begin{BMAT}{cccc|c}{cccc|c}
       (g^L)_{0,0} & (g^L)_{0,1} & \dots & (g^L)_{0,l-1} & (g^L)_{0,l}\\
       (g^L)_{1,0} & (g^L)_{1,1} & \dots & (g^L)_{1,l-1} & (g^L)_{1,l}\\
                   & \vdots      &       &               &  \vdots\\
       (g^L)_{l-1,0} & (g^L)_{l-1,1} & \dots & (g^L)_{l-1,l-1} & (g^L)_{l-1,l}\\
      [\chi(z)^\dagger]^{(0)} & [\chi(z)^\dagger]^{(1)} & \dots & [\chi(z)^\dagger]^{(l-1)} & [\chi(z)^\dagger]^{(l)}\\
      \end{BMAT}\right)(D^L)_{l}
\end{aligned}
\end{align}
and
\begin{align*}
&
  \begin{aligned}
  (\varphi^R_{1})^{(l)}(z)&=[\chi^{[l+1]}]^\top
  (((g^R)^{[l+1]})^{-1}\begin{pmatrix}
    0 \\0\\\vdots \\0\\ \mathbb{I}
  \end{pmatrix}=[\chi^{(l)}]^\top-[\chi^{[l]}]^\top ((g^R)^{[l]})^{-1}\begin{pmatrix}
    (g^R)_{0,l}\\(g^R)_{1,l}\\\vdots\\(g^R)_{l-1,l}
  \end{pmatrix}((D^R)^{-1})_l\\
&=\operatorname{SC}\left(\begin{BMAT}{cccc|c}{cccc|c}
       (g^R)_{0,0} & (g^R)_{0,1} & \cdots & (g^R)_{0,l-1} & (g^R)_{0,l}\\
       (g^R)_{1,0} & (g^R)_{1,1} & \cdots & (g^R)_{1,l-1} & (g^R)_{1,l}\\
        \vdots           & \vdots      &       &            \vdots   &  \vdots\\
       (g^R)_{l-1,0} & (g^R)_{l-1,1} & \dots & (g^R)_{l-1,l-1} & (g^R)_{l-1,l}\\
      [\chi(z)^\top]^{(0)} & [\chi(z)^\top]^{(1)} & \dots & [\chi(z)^\top]^{(l-1)} & [\chi(z)^\top]^{(l)}\\
      \end{BMAT}\right)((D^R)^{-1})_{l},
  \end{aligned}\\
&\begin{aligned}
\,  [(\varphi^R_{2})^{(l)}(z)]^\dagger&=(D^R)_l \begin{pmatrix}
    0  &\dots &0& \mathbb{I}
  \end{pmatrix}  (((g^R)^{[l+1]})^{-1}[(\chi^{[l+1]})^{\top}]^\dagger=
[(\chi^{(l)})^{\top}]^\dagger-\begin{pmatrix}
    (g^R)_{l,0}  & \cdots & (g^R)_{l,l-1}
  \end{pmatrix}((g^{R})^{[l]})^{-1}[(\chi^{[l]})^{\top}]^\dagger\\
&=\operatorname{SC}\left(\begin{BMAT}{cccc|c}{cccc|c}
       (g^R)_{0,0} & (g^R)_{0,1} & \dots & (g^R)_{0,l-1} &  [(\chi(z)^\top)^\dagger]^{(0)}\\
       (g^R)_{1,0} & (g^R)_{1,1} & \dots & (g^R)_{1,l-1} & [(\chi(z)^\top)^\dagger]^{(1)}\\
                   & \vdots      &       &               &  \vdots\\
       (g^R)_{l-1,0} & (g^R)_{l-1,1} & \dots & (g^R)_{l-1,l-1} & [(\chi(z)^\top)^\dagger]^{(l-1)}\\
      (g^R)_{l,0} & (g^R)_{l,1} & \dots & (g^R)_{l,l-1} & [(\chi(z)^\top)^\dagger]^{(l)}\\
      \end{BMAT}\right).
\end{aligned}
\end{align*}
\end{lemma}
\begin{proof}
See Appendix \ref{proofs}.
\end{proof}

Following \cite{Cafasso} we give expressions in terms of Schur complements for the matrix Szeg\H{o} polynomials, in terms of bordered truncated matrices of the right and left block CMV moment matrices. Extending though similar expressions given in \cite{Cafasso}   in terms of standard block moment matrices.
\begin{cor}
The left matrix Szeg\H{o} polynomials can be rewritten as the following Schur complements of bordered truncated CMV moment matrices
\begin{align*}
P_{1,2l}^L(z)&=z^l \operatorname{SC}\left(\begin{BMAT}{ccc|c}{c|c}
                   &(g^L)^{[2l]} &               &  \chi(z)^{[2l]}\\
      (g^L)_{2l,0}  & \dots       & (g^L)_{2l,2l-1} & \chi(z)^{(2l)}
      \end{BMAT}\right),  \\
P_{1,2l+1}^L(z)&=z^{l+1} \operatorname{SC}\left(\begin{BMAT}{ccc|c}{c|c}
                   &(g^R)^{[2l+1]} &               &  \chi^*(z)^{[2l+1]}\\
      (g^R)_{2l+1,0}  & \dots       & (g^R)_{2l+1,2l-1} & \chi^*(z)^{(2l+1)}\\
      \end{BMAT}\right),\\
[P_{2,2l}^L(z)]^\dagger&=\bar{z}^l \operatorname{SC}\left(\begin{BMAT}{c|c}{ccc|c}
                                  &  (g^L)_{0,2l}\\
              (g^L)^{[2l]}       &  \vdots\\
                                  &  (g^L)_{2l-1,2l}\\
         (\chi(z)^\dagger)^{[2l]} & (\chi(z)^\dagger)^{(2l)}
      \end{BMAT}\right),\quad
[P_{2,2l+1}^L(z)]^\dagger=\bar{z}^{l+1} \operatorname{SC}\left(\begin{BMAT}{c|c}{ccc|c}
                                  &  (g^R)_{0,2l+1}\\
              (g^R)^{[2l+1]}       &  \vdots\\
                                  &  (g^R)_{2l,2l+1}\\
         (\chi^*(z)^\dagger)^{[2l+1]} & (\chi^*(z)^\dagger)^{(2l+1)}
      \end{BMAT}\right),
\end{align*}
while for the right polynomials we have
\begin{align*}
P_{1,2l}^R(z)&=z^l \operatorname{SC}\left(\begin{BMAT}{c|c}{ccc|c}
                                  &  (g^R)_{0,2l}\\
               (g^R)^{[2l]}       &  \vdots\\
                                  &  (g^R)_{2l-1,2l}\\
         (\chi(z)^\top)^{[2l]} & (\chi(z)^\top)^{(2l)}\\
      \end{BMAT}\right),\quad
P_{1,2l+1}^R(z)=z^{l+1} \operatorname{SC}\left(\begin{BMAT}{c|c}{ccc|c}
                                  &  (g^L)_{0,2l+1}\\
               (g^L)^{[2l+1]}       &  \vdots\\
                                  &  (g^L)_{2l,2l+1}\\
         (\chi^*(z)^\top)^{[2l+1]} & (\chi^*(z)^\top)^{(2l+1)}
      \end{BMAT}\right),\\
[P_{2,2l}^R(z)]^\dagger&=\bar{z}^l \operatorname{SC}\left(\begin{BMAT}{ccc|c}{c|c}
                   & (g^R)^{[2l]} &               &  (\chi(z)^\dagger)^{[2l]}\\
      (g^R)_{2l,0}  & \dots       & (g^R)_{2l,2l-1} & (\chi(z)^\dagger)^{(2l)}
      \end{BMAT}\right), \\
[P_{2,2l+1}^R(z)]^\dagger&=\bar{z}^{l+1} \operatorname{SC}\left(\begin{BMAT}{ccc|c}{c|c}
                   & (g^L)^{[2l+1]} &               &  (\chi^*(z)^\dagger)^{[2l+1]}\\
      (g^R)_{2l+1,0}  & \dots       & (g^R)_{2l+1,2l-1} & (\chi^*(z)^\dagger)^{(2l+1)}
      \end{BMAT}\right).      \end{align*}
\end{cor}
\begin{proof} These relations appear when one introduces in \eqref{CMVSi} and \eqref{CMVSf} the expressions of the CMV polynomials in terms of Schur complements.
\end{proof}

\subsection{Matrix second kind functions}\label{Sec-second}
The following matrix fashion of rewritting previous left objects
\begin{align*}
 \begin{pmatrix}
  \phi_{1,1}^L&\phi_{1,2}^L\\
  \phi_{2,1}^L&\phi_{2,2}^L
 \end{pmatrix}&=
\left[\begin{array}{cc}
                      S_1&0\\
                      0&[S_2^{-1}]^\dagger
                     \end{array} \right]
\begin{pmatrix}
 \chi_1 & \chi_2^*\\
 \chi_1 & \chi_2^*
\end{pmatrix}=
\begin{pmatrix}
S_1\chi_1&S_1\chi_2^*\\
[S_2^{-1}]^\dagger\chi_1& [S_2^{-1}]^\dagger\chi_2^*
\end{pmatrix}, \\
 \begin{pmatrix}
  \phi_1^L\\
  \phi_2^L
 \end{pmatrix}&=
\begin{pmatrix}
  \phi_{1,1}^L&\phi_{1,2}^L\\
  \phi_{2,1}^L&\phi_{2,2}^L
 \end{pmatrix}
\begin{pmatrix}
  1\\
  1\end{pmatrix},\\
 \left[\begin{array}{cc}
                      [g^L]^\dagger&0\\
                      0&g^L
                     \end{array} \right]&=
\left[\begin{array}{cc}
                      [S_2]^{\dagger}[S_1^{-1}]^\dagger&0\\
                      0&S_1^{-1}S_2
                     \end{array} \right]=
\left[\begin{array}{cc}
                      [S_2]^\dagger&0\\
                      0&S_1^{-1}
                     \end{array} \right]
\left[\begin{array}{cc}
                      [S_1^{-1}]^\dagger&0\\
                      0&S_2
                     \end{array} \right]
\end{align*}
 and the right ones
\begin{align*}
 \begin{pmatrix}
  \phi_{1,1}^R&\phi_{2,1}^R\\
  \phi_{1,2}^R&\phi_{2,2}^R
 \end{pmatrix}&=
\begin{pmatrix}
 \chi_1^\top & \chi_1^\top\\
 [\chi_2^*]^\top & [\chi_2^*]^\top
\end{pmatrix}
\left[\begin{array}{cc}
                      Z_1&0\\
                      0&[Z_2^{-1}]^\dagger
                     \end{array} \right],\\
 \begin{pmatrix}
  \phi_1^R & \phi_2^R
 \end{pmatrix}&=
\begin{pmatrix}
  1&1\end{pmatrix}
\begin{pmatrix}
  \phi_{1,1}^R&\phi_{2,1}^R\\
  \phi_{1,2}^R&\phi_{2,2}^R
 \end{pmatrix},\\
 \left[\begin{array}{cc}
                      [g^R]^\dagger&0\\
                      0&g^R
                     \end{array} \right]&=
\left[\begin{array}{cc}
                      [Z_1^{-1}]^\dagger&0\\
                      0&Z_2
                     \end{array} \right]
\left[\begin{array}{cc}
                      [Z_2]^\dagger&0\\
                      0&Z_1^{-1}
                     \end{array} \right]
\end{align*}
inspires the next
\begin{definition}
 The partial matrix CMV second kind sequences are given by
\begin{align*}
 \begin{pmatrix}
  C_{1,1}^L& C_{1,2}^L\\
  C_{2,1}^L&C_{2,2}^L
 \end{pmatrix}&:=
\left[\begin{array}{cc}
                      [S_1^{-1}]^\dagger&0\\
                      0&S_2
                     \end{array} \right]
\begin{pmatrix}
 \chi_1^* & \chi_2\\
 \chi_1^* & \chi_2
\end{pmatrix}=
\begin{pmatrix}
[S_1^{-1}]^\dagger\chi_1^*&[S_1^{-1}]^\dagger\chi_2\\
S_2\chi_1^*& S_2\chi_2
\end{pmatrix},\\
 \begin{pmatrix}
  C_{1,1}^R&C_{2,1}^R\\
  C_{1,2}^R&C_{2,2}^R
 \end{pmatrix}&:=
\begin{pmatrix}
 [\chi_1^*]^\top & [\chi_1^*]^\top\\
 \chi_2^\top & \chi_2^\top
\end{pmatrix}
\left[\begin{array}{cc}
                      [Z_1^{-1}]^\dagger&0\\
                      0&Z_2
                     \end{array} \right]=
\begin{pmatrix}
[\chi_1^*]^\top [Z_1^{-1}]^\dagger & [\chi_1^*]^\top Z_2\\
\chi_2^\top [Z_1^{-1}]^\dagger& \chi_2^\top Z_2
\end{pmatrix},
\end{align*}
and the corresponding matrix CMV second kind sequences are
\begin{align*}
 \begin{pmatrix}
  C_1^L\\
  C_2^L
 \end{pmatrix}&=
\begin{pmatrix}
  C_{1,1}^L&C_{1,2}^L\\
  C_{2,1}^L&C_{2,2}^L
 \end{pmatrix}
\begin{pmatrix}
  1\\
  1\end{pmatrix},\\
 \begin{pmatrix}
  C_1^R & C_2^R
 \end{pmatrix}&=
\begin{pmatrix}
  1&1\end{pmatrix}
\begin{pmatrix}
  C_{1,1}^R&C_{2,1}^R\\
  C_{1,2}^R&C_{2,2}^R
 \end{pmatrix}.
\end{align*}
\end{definition}
Complementary to the above definition
\begin{definition}
The associated CMV Fourier series are
 \begin{align*}
 \begin{pmatrix}
  \Gamma_{2,1}^L&\Gamma_{2,2}^L\\
  \Gamma_{1,1}^L&\Gamma_{1,2}^L
 \end{pmatrix}&:=
\left[\begin{array}{cc}
                      [g^L]^\dagger&0\\
                      0&g^L
                     \end{array} \right]
\begin{pmatrix}
 \chi_1^* & \chi_2\\
 \chi_1^* & \chi_2
\end{pmatrix}=
\begin{pmatrix}
[g^L]^\dagger\chi_1^*&[g^L]^\dagger\chi_2\\
g^L\chi_1^*& g^L\chi_2
\end{pmatrix},\\
 \begin{pmatrix}
  \Gamma_{2,1}^R&\Gamma_{1,1}^R\\
  \Gamma_{2,2}^R&\Gamma_{1,2}^R
 \end{pmatrix}&:=
\begin{pmatrix}
 [\chi_1^*]^\top & [\chi_1^*]^\top\\
 \chi_2^\top & \chi_2^\top
\end{pmatrix}
\left[\begin{array}{cc}
                      [g^R]^\dagger&0\\
                      0&g^R
                     \end{array} \right]=
\begin{pmatrix}
[\chi_1^*]^\top [g^R]^\dagger & [\chi_1^*]^\top g^R\\
\chi_2^\top [g^R]^\dagger& \chi_2^\top g^R
\end{pmatrix},
\end{align*}
\end{definition}
for which we have:
\begin{pro}\label{previous}
\begin{enumerate}
\item The elements  $\Gamma^H$ and $C^H$, $H=L,R$,  are related in the following way
\begin{align*}
\begin{pmatrix}
  \Gamma_2^L\\
  \Gamma_1^L
 \end{pmatrix}&=
\left[\begin{array}{cc}
                      [S_2]^\dagger&0\\
                      0&S_1^{-1}
                     \end{array} \right]
\begin{pmatrix}
  C_1^L\\
  C_2^L
 \end{pmatrix},&
 \begin{pmatrix}
  \Gamma_2^R & \Gamma_1^R
 \end{pmatrix}&=
\begin{pmatrix}
  C_1^R & C_2^R
 \end{pmatrix}
\left[\begin{array}{cc}
                      [Z_2]^\dagger&0\\
                      0&Z_1^{-1}
                     \end{array} \right].
\end{align*}
 \item The second kind functions can be expressed as Schur complements as follows
\begin{align*}
[(C_1^L)^{(l)}(z)]^\dagger&=\operatorname{SC}\left(\begin{array}{c|c}
                                  &  (g^L)_{0,l}\\
               (g^L)^{[l]}       &  \vdots\\
                                  &  (g^L)_{l-1,l}\\
\hline
         (\Gamma_2^L(z)^\dagger)^{[l]} & (\Gamma_2^L(z)^\dagger)^{(l)}\\
      \end{array}\right)(D_L^{-1})_l,\\
(C_2^L)^{(l)}(z))&=\operatorname{SC}\left(\begin{array}{ccc|c}
                   & (g^L)^{[l]} &               &  (\Gamma_1(z))^{[l]}\\
\hline
      (g^L)_{l,0}  & \dots       & (g^L)_{l,l-1} & (\Gamma_1(z))^{(l)}\\
      \end{array}\right), \\
[(C_1^R)^{(l)}(z)]^\dagger&=\operatorname{SC}\left(\begin{array}{ccc|c}
                   & (g^R)^{[l]} &               &  (\Gamma_2^R(z)^\dagger)^{[l]}\\
\hline
      (g^R)_{l,0}  & \dots       & (g^R)_{l,l-1} & (\Gamma_2^R(z)^\dagger)^{(l)}\\
      \end{array}\right), \\
(C_2^R)^{(l)}(z)&=\operatorname{SC}\left(\begin{array}{c|c}
                                  &  (g^R)_{0,l}\\
               (g^R)^{[l]}       &  \vdots\\
                                  &  (g^R)_{l-1,l}\\
\hline
         (\Gamma_1^R(z))^{[l]} & (\Gamma_1^R(z))^{(l)}\\
      \end{array}\right)(D_R^{-1})_l.
\end{align*}
\item In terms of the matrix Laurent orthogonal polynomials and the Fourier series of the matrix measure we have
\begin{align}\label{CF}
\begin{aligned}
 (C_1^L)^{(l)}(z)&=2\pi z^{-1}(\varphi_2^L)^{(l)}(z^{-1})F_{\mu}^\dagger(z),
  & (C_2^L)^{(l)}(z)&=2\pi z^{-1}(\varphi_1^L)^{(l)}(z^{-1})F_{\mu}(z^{-1}),\\
 (C_1^R)^{(l)}(z)&=2\pi z^{-1}F_{\mu}^\dagger(z)(\varphi_2^R)^{(l)}(z^{-1}),
  & (C_2^R)^{(l)}(z)&=2\pi z^{-1}F_{\mu}(z^{-1})(\varphi_1^R)^{(l)}(z^{-1}).
\end{aligned}
\end{align}
\end{enumerate}
\end{pro}
\begin{proof}
The first part of the proposition follows directly from comparison of the structure of the relations from the
previous lemma with the definitions of the CMV matrix polynomials. For example
\begin{align*}
 \Gamma_1^L&=S_1^{-1} C_2^L\Rightarrow C_2^L=S_1 \Gamma_1^L  &
\mbox{ same structure as } &
 &\phi_1^L=S_1 \chi & & \mbox{ replacing } \Gamma_1^L\longleftrightarrow \chi.
\end{align*}
For the second part of the Proposition, we shall only prove one of the cases since the rest of them can be
proven following the same procedure. First, from the definition of the second kind functions, we have
\begin{align*}
 C_1^R(z)&=(\chi^*(z))^\top [Z_1^{-1}]^\dagger\\&=(\chi^*(z))^\top [g^R]^\dagger (Z_2^{-1})^\dagger\\
 &=[\chi^*(z)]^{\top}\oint_\T[\chi^\top(u)]^\dagger \left[\frac{\d\mu(u)}{\operatorname{i}u}\right]^\dagger [\chi(u)]^\top Z_2^{-1}\\&=
[\chi^*(z)]^{\top}\oint_\T[\chi^\top(u)]^\dagger \left[\frac{\d\mu(u)}{\operatorname{i}u}\right]^\dagger \phi_2^R(u).
\end{align*}
Taking the $l^{\text{th}}$ component of this vector of matrices we get
\begin{align*}
 (C_1^R)^{(l)}(z)&=\int_0^{2\pi}\sum_{n=-\infty}z^{n-1}\Exp{\operatorname{i}n\theta}[\d\mu(\theta)]^\dagger(\varphi_2^R)^{(l)}
 \big(\Exp{\operatorname{i}\theta}\big)\\&=
\sum_{k,n=-\infty}z^{n-1}\int_0^{2\pi} \Exp{\operatorname{i}(n+k)\theta}[\d\mu(\theta)]^\dagger(\varphi_{2,k}^R)^{(l)}\\
&=2\pi z^{-1} \Big(\sum_{k=-\infty}(\varphi_{2,k}^R)^{(l)}z^{-k}\Big)\Big(\sum_{n=-\infty}c_{n+k}^\dagger z^{n+k}\Big)\\&=
2\pi z^{-1}F_{\mu}^\dagger(z)(\varphi_2^R)^{(l)}(z^{-1}).
\end{align*}
\end{proof}
Recalling the previously stated relation between the $\Gamma^H$ and the $C^H$ it follows from  Proposition \ref{previous} that
\begin{pro} The associated CMV Fourier series satisfy
\begin{align*}
 \Gamma_{1,j}^L&=2\pi z^{-1}F_\mu(z^{-1})\chi^{(j)}(z^{-1}), &
 \Gamma_{2,j}^L&=2\pi z^{-1}F_\mu^{\dagger}(z)\chi^{(j)}(z^{-1}),\\
 \Gamma_{1,j}^R&=2\pi z^{-1}F_\mu(z^{-1})\chi^{(j)}(z^{-1}), &
 \Gamma_{2,j}^R&=2\pi z^{-1}F_\mu^\dagger(z)\chi^{(j)}(z^{-1}).
\end{align*}
\end{pro}
Another interesting representation of these functions is
\begin{pro} The second kind functions have the following Cauchy integral type formulae
\begin{align*}
  [C_{1,1}^L(z)]^\dagger&=\oint_\T [z^{-1}\frac{u}{u-z^{-1}}]^\dagger \frac{\d\mu(u)}{\operatorname{i}u} [\phi_2^L(u)]^\dagger, &
  C_{2,1}^L(z)&=\oint_\T \phi_1^L(u) \frac{\d\mu(u)}{\operatorname{i}u} [z^{-1}\frac{u}{u-z^{-1}}], & |z|&>1,\\
  [C_{1,2}^L(z)]^\dagger&=\oint_\T [-z^{-1}\frac{u}{u-z^{-1}}]^\dagger \frac{\d\mu(u)}{\operatorname{i}u} [\phi_2^L(u)]^\dagger, &
  C_{2,2}^L(z)&=\oint_\T \phi_1^L(u) \frac{\d\mu(u)}{\operatorname{i}u} [-z^{-1}\frac{u}{u-z^{-1}}] ,& |z|&<1,
\\
  [C_{1,1}^R(z)]^\dagger&=\oint_\T [\phi_2^R(u)]^\dagger \frac{\d\mu(u)}{\operatorname{i}u} [z^{-1}\frac{u}{u-z^{-1}}]^\dagger,  &
  [C_{2,1}^R(z)]&=\oint_\T [z^{-1}\frac{u}{u-z^{-1}}] \frac{\d\mu(u)}{\operatorname{i}u} \phi_1^R(u), & |z|&>,1\\
  [C_{1,2}^R(z)]^\dagger&=\oint_\T [\phi_2^R(u)]^\dagger \frac{\d\mu(u)}{\operatorname{i}u} [-z^{-1}\frac{u}{u-z^{-1}}]^\dagger,  &
  [C_{2,1}^R(z)]&=\oint_\T [-z^{-1}\frac{u}{u-z^{-1}}] \frac{\d\mu(u)}{\operatorname{i}u} \phi_1^R(u), & |z|&<1.
\end{align*}
\end{pro}
\begin{proof}
Direct substitution leads to
 \begin{align*}
  [C_{1,1}^L](z)^\dagger&=\oint_\T \Big[\sum_{n=0}z^{-1}(uz)^{-n}\Big]^\dagger \frac{\d\mu(u)}{\operatorname{i}u} [\phi_2^L(u)]^\dagger, &
  C_{2,1}^L(z)&=\oint_\T \phi_1^L(u) \frac{\d\mu(u)}{\operatorname{i}u} \Big[\sum_{n=0}z^{-1}(uz)^{-n}\Big],\\
  [C_{1,2}^L(z)]^\dagger&=\oint_\T \Big[\sum_{n=0}u(uz)^{n}\Big]^\dagger \frac{\d\mu(u)}{\operatorname{i}u} [\phi_2^L(u)]^\dagger, &
  C_{2,2}^L(z)&=\oint_\T \phi_1^L(u) \frac{\d\mu(u)}{\operatorname{i}u} \Big[\sum_{n=0}u(uz)^{n}\Big],\\
  [C_{1,1}^R(z)]^\dagger&=\oint_\T [\phi_2^R(u)]^\dagger \frac{\d\mu(u)}{\operatorname{i}u} \Big[\sum_{n=0}z^{-1}(uz)^{-n}\Big]^\dagger,  &
  C_{2,1}^R(z)&=\oint_\T \Big[\sum_{n=0}z^{-1}(uz)^{-n}\Big] \frac{\d\mu(u)}{\operatorname{i}u} \phi_1^R(u),\\
  [C_{1,2}^R(z)]^\dagger&=\oint_\T [\phi_2^R(u)]^\dagger \frac{\d\mu(u)}{\operatorname{i}u} \Big[\sum_{n=0}u(uz)^{n}\Big]^\dagger,  &
  C_{2,1}^R(z)&=\oint_\T \Big[\sum_{n=0}u(uz)^{-n}\Big] \frac{\d\mu(u)}{\operatorname{i}u} \phi_1^R(u).\\
\end{align*}
But  these are the series expansions of the functions of the proposition. We will not deal here with convergence problems since their discussion follows the ideas in \cite {carlos}.
\end{proof}

\subsection{Recursion relations}\label{Sec-recursions}
In order to get the recursion relations we introduce the following
\begin{definition} For each pair $ i,j\in \Z_+$ we consider the block semi-infinite matrix $E_{i,j}$ whose only non zero $m\times m$ block is the $i,j$-th block where the identity of $\mathbb M_m$ appears.  Then, we define the projectors
 \begin{align*}
   \Pi_{1}&:=\sum_{j=0}^\infty E_{2j,2j},&  \Pi_{2}&:=\sum_{j=0}^\infty E_{2j+1,2j+1},
 \end{align*}
 and the following matrices
\begin{gather*}
\begin{aligned}
\Lambda_{1}&:=\sum_{j=0}^{\infty}E_{2j,2+2j},&
\Lambda_{2}&:=\sum_{j=0}^{\infty} E_{1+2j,3+2j},&
\Lambda&:=\sum_{j=0}^{\infty}E_{j,j+1},
\end{aligned}\\
 \Upsilon:=\Lambda_1+\Lambda_2^{\top}+E_{1,1}\Lambda^{\top}.
\end{gather*}
\end{definition}

The matrix $\Upsilon$, which can be written more explicitly as follows
\begin{align*}\Upsilon=\left(
\begin{BMAT}{cc:cc:cc:cc:ccc}{cc:cc:cc:cc:ccc}
0 & 0 & \mathbb{I} & 0 & 0 & 0 & 0 & 0 & 0 & 0 &\cdots  \\
\mathbb{I} & 0 & 0 & 0 & 0 & 0 & 0 & 0 & 0 & 0 &\cdots  \\
0 & 0 & 0 & 0 & \mathbb{I} & 0 & 0 & 0 & 0 & 0 &\cdots  \\
0 & \mathbb{I} & 0 & 0 & 0 & 0 & 0 & 0 & 0 & 0 &\cdots  \\
0 & 0 & 0 & 0 & 0 & 0 & \mathbb{I} & 0 & 0 & 0 &\cdots  \\
0 & 0 & 0 & \mathbb{I} & 0 & 0 & 0 & 0 & 0 & 0 &\cdots  \\
0 & 0 & 0 & 0 & 0 & 0 & 0 & 0 & \mathbb{I} & 0 &\cdots  \\
0 & 0 & 0 & 0 & 0 & \mathbb{I} & 0 & 0 & 0 & 0 &\cdots  \\
0 & 0 & 0 & 0 & 0 & 0 & 0 & 0 & 0 & 0 &\cdots  \\
0 & 0 & 0 & 0 & 0 & 0 & 0 & \mathbb{I} & 0 & 0 &\cdots  \\
\vdots & \vdots & \vdots & \vdots & \vdots & \vdots &
\vdots & \vdots & \vdots & \vdots &\ddots
\end{BMAT}\right),
\end{align*}
satisfies
\begin{align*}
  \Upsilon^{\dagger}=\Upsilon^{-1}=\Upsilon^\top
\end{align*}
and has the following properties
\begin{pro}
The next \emph{eigen-value} type relations hold true
\begin{align}
\Upsilon \chi(z)&=z\chi(z),&
\Upsilon^{-1} \chi(z)&=z^{-1}\chi(z), \\
\chi(z)^{\top} \Upsilon^{-1} &= z \chi(z)^{\top},&
\chi(z)^{\top} \Upsilon &= z^{-1}\chi(z)^{\top}.
\end{align}
\end{pro}
\begin{proof}
It  follows from the relations
 \begin{align*}
\Lambda_{1} \chi(z)&=z\Pi_{1}  \chi(z),&
\Lambda_{2}  \chi(z)&=z^{-1}\Pi_{2}  \chi(z), \\
\Lambda_{1}^{\top} \chi(z)&=(z^{-1}\Pi_1-E_{0,0}\Lambda) \chi(z),&
\Lambda_{2}^{\top} \chi(z)&=(z\Pi_{2} -E_{1,1}\Lambda^{\top}) \chi(z).
\end{align*}
\end{proof}

From which the following \emph{symmetry} relations are obtained
\begin{pro}
  The  moment matrices commute with $\Upsilon$; i.e.
  \begin{align}\label{conmut}
\Upsilon g ^H&=g^H\Upsilon,& H&=L,R.
\end{align}
\end{pro}
\begin{proof}
Is a consequence of
\begin{align*}
\Upsilon g^L&=\oint_{\T}z\chi(z) \frac{\d\mu(z)}{\operatorname{i}z} \chi(z)^{\dagger}=\oint_{\T}\chi(z) \frac{\d\mu(z)}{\operatorname{i}z} (z^{-1}\chi(z))^{\dagger}=g^L \Upsilon,\\
\Upsilon g^R&=\oint_{\T}\overline{z\chi(z)}\, \frac{\d\mu(z)}{\operatorname{i}z} \chi(z)^{\top}=\oint_{\T}\overline{\chi(z)}\,\frac{\d\mu(z)}{\operatorname{i}z} z^{-1}\chi(z)^\top=g^R \Upsilon.
\end{align*}
\end{proof}
We now introduce another important matrix in the CMV theory
\begin{definition} The intertwining matrix $\eta$ is
\begin{align*}\eta&:=\left(
\begin{BMAT}{c:cc:cc:ccc}{c:cc:cc:ccc}
 \mathbb{I} & 0 & 0 & 0 & 0 & 0 & 0 & \cdots\\
 0 & 0 & \mathbb{I} & 0 & 0 & 0 & 0 & \cdots\\
 0 & \mathbb{I} & 0 & 0 & 0 & 0 & 0 & \cdots\\
 0 & 0 & 0 & 0 & \mathbb{I} & 0 & 0 & \cdots\\
 0 & 0 & 0 & \mathbb{I} & 0 & 0 & 0 & \cdots\\
 0 & 0 & 0 & 0 & 0 & 0 & \mathbb{I} & \cdots\\
 0 & 0 & 0 & 0 & 0 & \mathbb{I} & 0 & \cdots\\
 \vdots & \vdots & \vdots & \vdots & \vdots & \vdots & \vdots & \ddots
\end{BMAT}\right).
\end{align*}
\end{definition}
Which, as the reader can easily check, has the following properties
\begin{align*}
 \eta^{-1} & =\eta, & \eta \chi(z)&=\chi\big(z^{-1}\big), & \chi(z)^{\top}\eta= \chi\big(z^{-1}\big)^{\top}.
\end{align*}
When $z \in \T$ we have that  $\eta \chi=\bar\chi$ and $\chi^\top\eta=\chi^\dagger$ which lead to the intertwining property
\begin{pro}
 The left and right moment matrices satisfy the intertwining type property
 \begin{align*}
  \eta g^R&= g^L \eta.
 \end{align*}
\end{pro}
\begin{proof}
 It is straightforward to realize that
 \begin{align*}
  \eta g^L \eta &= \oint_{\T} \eta \chi(z) \frac{\d\mu(z)}{\operatorname{i}z} \chi(z)^{\dagger} \eta=
  \oint_{\T} \chi(\bar{z}) \frac{\d\mu(z)}{\operatorname{i}z} \chi\big(\bar{z}^{-1}\big)^{\top}=
  \oint_{\T} \overline{\chi(z)}\,\frac{\d\mu(z)}{\operatorname{i}z} \chi(z)^{\top}= g^R.
 \end{align*}
\end{proof}
\begin{pro}
 Both the matrices $\Upsilon$ and $\eta$ are related by
 \begin{align*}
  \eta \Upsilon= \Upsilon^{-1} \eta.
 \end{align*}
\end{pro}
Now we proceed to the \emph{dressing} of  $\Upsilon$ and $\eta$. We first notice that
\begin{pro}
  The following equations hold
\begin{align*}
 S_1 \Upsilon S_1^{-1}&= S_2 \Upsilon S_2^{-1}, \\ Z_1^{-1}\Upsilon Z_1&=Z_2^{-1}\Upsilon Z_2,\\
Z_2^{-1}\eta \Upsilon^{p} S_1^{-1}&= Z_1^{-1}\eta \Upsilon^{p}S_2^{-1} &p&\in\Z.
\end{align*}
\end{pro}
Which allow us to define
\begin{definition}
Let us define
\begin{align} \label{jota}
J^L&:= S_1 \Upsilon S_1^{-1}= S_2 \Upsilon S_2^{-1}, &
J^R&:= Z_1^{-1}\Upsilon Z_1=Z_2^{-1}\Upsilon Z_2,
\end{align}
and for any $p\in\Z$  introduce
\begin{align} \label{Ccas}
 C_{[p]}=Z_2^{-1}\eta \Upsilon^{p} S_1^{-1}= Z_1^{-1}\eta \Upsilon^{p}S_2^{-1}.
\end{align}
\end{definition}
\paragraph{Observations}
\begin{enumerate}
\item
\begin{align*}
 C_{[-|p|]}=Z_2^{-1}\eta \Upsilon^{-|p|} S_1^{-1}=Z_2^{-1}\Upsilon^{|p|}\eta  S_1^{-1}
\end{align*}
\item In the Hermitian case
\begin{align*}
 C_{[0]}=Z_2^{-1}\eta S_1^{-1}= Z_1^{-1} \eta S_2^{-1}=D^R \widehat{Z_1}^{-1}\eta\,\, \widehat{S_2}^{-1} D^L=
D^RZ_2^{\dagger}\eta S_1^\dagger D^L \Longrightarrow C_{[0]}^\dagger=D^LC_{[0]}^{-1}D^R.
\end{align*}
\end{enumerate}
\begin{pro}
Powers  of $J_H$ can be expressed as follows
 \begin{align*}
(J^R)^{l-p}&=C_{[p]}[C_{[l]}]^{-1}, & (J^L)^{p-l}&=[C_{[l]}]^{-1}C_{[p]}.
 \end{align*}
\end{pro}
Now we give the schematic shape of some of these matrices
\begin{align*}J ^H&=\left(
\begin{BMAT}{cc:cc:cc:cc:ccc}{cc:cc:cc:cc:ccc}
* & * & * & 0 & 0 & 0 & 0 & 0 & 0 & 0 &\cdots  \\
* & * & * & 0 & 0 & 0 & 0 & 0 & 0 & 0 &\cdots  \\
0 & * & * & * & * & 0 & 0 & 0 & 0 & 0 &\cdots  \\
0 & * & * & * & * & 0 & 0 & 0 & 0 & 0 &\cdots  \\
0 & 0 & 0 & * & * & * & * & 0 & 0 & 0 &\cdots  \\
0 & 0 & 0 & * & * & * & * & 0 & 0 & 0 &\cdots  \\
0 & 0 & 0 & 0 & 0 & * & * & * & * & 0 &\cdots  \\
0 & 0 & 0 & 0 & 0 & * & * & * & * & 0 &\cdots  \\
0 & 0 & 0 & 0 & 0 & 0 & 0 & * & * & * &\cdots  \\
0 & 0 & 0 & 0 & 0 & 0 & 0 & * & * & * &\cdots  \\
\vdots & \vdots & \vdots & \vdots & \vdots & \vdots &
\vdots & \vdots & \vdots & \vdots &\ddots
\end{BMAT}\right), &
(J^{H})^{-1}&=\left(
\begin{BMAT}{cc:cc:cc:cc:ccc}{cc:cc:cc:cc:ccc}
* & * & 0 & 0 & 0 & 0 & 0 & 0 & 0 & 0 &\cdots  \\
* & * & * & * & 0 & 0 & 0 & 0 & 0 & 0 &\cdots  \\
* & * & * & * & 0 & 0 & 0 & 0 & 0 & 0 &\cdots  \\
0 & 0 & * & * & * & * & 0 & 0 & 0 & 0 &\cdots  \\
0 & 0 & * & * & * & * & 0 & 0 & 0 & 0 &\cdots  \\
0 & 0 & 0 & 0 & * & * & * & * & 0 & 0 &\cdots  \\
0 & 0 & 0 & 0 & * & * & * & * & 0 & 0 &\cdots  \\
0 & 0 & 0 & 0 & 0 & 0 & * & * & * & * &\cdots  \\
0 & 0 & 0 & 0 & 0 & 0 & * & * & * & * &\cdots  \\
0 & 0 & 0 & 0 & 0 & 0 & 0 & 0 & * & * &\cdots  \\
\vdots & \vdots & \vdots & \vdots & \vdots & \vdots &
\vdots & \vdots & \vdots & \vdots &\ddots
\end{BMAT}\right),\\
C_{[0]}&=\left(
\begin{BMAT}{c:cc:cc:ccc}{c:cc:cc:ccc}
 * & * & 0 & 0 & 0 & 0 & 0 & \cdots\\
 * & * & * & 0 & 0 & 0 & 0 & \cdots\\
 0 & * & * & * & 0 & 0 & 0 & \cdots\\
 0 & 0 & * & * & * & 0 & 0 & \cdots\\
 0 & 0 & 0 & * & * & * & 0 & \cdots\\
 0 & 0 & 0 & 0 & * & * & * & \cdots\\
 0 & 0 & 0 & 0 & 0 & * & * & \cdots\\
 \vdots & \vdots & \vdots & \vdots & \vdots & \vdots & \vdots & \ddots
\end{BMAT}\right), &
C_{[-1]}&=\left(
\begin{BMAT}{c:cc:cc:ccc}{c:cc:cc:ccc}
  * & * & 0 & 0 & 0 & 0 & 0 & \cdots\\
 * & * & * & 0 & 0 & 0 & 0 & \cdots\\
 0 & * & * & * & 0 & 0 & 0 & \cdots\\
 0 & 0 & * & * & * & 0 & 0 & \cdots\\
 0 & 0 & 0 & * & * & * & 0 & \cdots\\
 0 & 0 & 0 & 0 & * & * & * & \cdots\\
 0 & 0 & 0 & 0 & 0 & * & * & \cdots\\
 \vdots & \vdots & \vdots & \vdots & \vdots & \vdots & \vdots & \ddots
\end{BMAT}\right).
\end{align*}

Where the $*$ are non-zero $m\times m$ blocks that thanks to the factorization problem can be written in terms of the Verblunsky coefficients
as we will see later. The shape of each matrix is a consequence of  the two possible definitions (in terms of upper or lower block-triangular matrices).
 For the explicit form of these matrices see Appendix \ref{explicit}.

 A first consequence is the following relations among Verblunsky coefficients and the \emph{matrix quasi-norms} of the Szeg\H{o} polynomials
\begin{pro}\label{rrr}
The following   relations are fulfilled
 \begin{align*}
h_{m}^R [\alpha_{2,m}^L]^\dagger&= [\alpha_{2,m}^R]^\dagger h_{m}^L,  &   h_{m}^L \alpha_{1,m+1}^R&= \alpha_{1,m+1}^L h_{m}^R,\\
\ \alpha_{1,m}^L h_{m}^R&=h_{m}^L \alpha_{1,m}^R,  &   [\alpha_{2,m+1}^R]^\dagger h_{m}^L&=h_{m}^R [\alpha_{2,m+1}^L]^\dagger,\\
 h_k^L&=\Big(\mathbb{I}-\alpha_{1,k}^L[\alpha_{2,k}^R]^\dagger \Big)h_{k-1}^{L},  &   h_k^R&=\Big(\mathbb{I}-[\alpha_{2,k}^R]^\dagger \alpha_{1,k}^L\Big) h_{k-1}^R,\\
h_k^L&= \ h_{k-1}^{L}\Big(\mathbb{I}-\alpha_{1,k}^R[\alpha_{2,k}^L]^\dagger\Big),   & h_k^R&= h_{k-1}^{R}\Big(\mathbb{I}-[\alpha_{2,k}^L]^\dagger \alpha_{1,k}^R\Big),
\end{align*}
\end{pro}
\begin{proof}
Just compare the two possible definitions of $C_{[0]}^{\pm{1}}$ and $C_{[-1]}^{\pm{1}}$.
\end{proof}
Notice that the two relations in each column  coincide in the Hermitian case.

\begin{pro}\label{rec}
 The next
 \emph{eigen-value} properties hold
 \begin{align*}
 J^L \Phi_1^L &= z \Phi_1^L,                             &       (J^L)^{-1} \Phi_1^L &= z^{-1} \Phi_1^L, &
 [J^L]^\dagger \Phi_2^L &= z^{-1} \Phi_2^L,              &       ([J^L]^\dagger)^{-1} \Phi_2^L &= z \Phi_2^L,\\
 \Phi_2^R [J^R]^\dagger&= z \Phi_2^R,                     &       \Phi_2^R ([J^R]^\dagger)^{-1}&= z^{-1} \Phi_2^R, &
 \Phi_1^R J^R&= z^{-1} \Phi_1^R,                         &       \Phi_1^R (J^R)^{-1}&= z \Phi_1^R,
 \end{align*}
 and the following properties are fulfilled
 \begin{align*}
 C_{[p]}\Phi_1^L(z)&=z^p\Big( \Phi_2^R\big(\bar{z}^{-1}\big)\Big)^\dagger, &
 \Phi_1^R(z)C_{[p]}&=z^p \Big( \Phi_2^L\big(\bar{z}^{-1}\big)\Big)^\dagger.
\end{align*}
\end{pro}
\begin{proof}The results follow directly from the action of $\Upsilon^{\pm 1}$ and $\eta$ on $\chi$ and the definitions
of $J^H,C_{[p]}$ and $\Phi^H_i$.
For example
 \begin{align*}
J^L \Phi^L_{1} =& S_1 \Upsilon S_1^{-1}S_1 \chi(z)=S_1 \Upsilon \chi(z)= z S_1 \chi(z)= z \Phi^L_{1} \\
 C_{[p]}\Phi^L_{1}=& Z_2^{-1} \eta \Upsilon^{p} S_1^{-1} S_1 \chi(z)= Z_2^{-1} \eta z^p \chi(z) = z^p Z_2^{-1} \chi\big(z^{-1}\big)=
 z^p \Big( \Phi_2^R\big(\bar{z}^{-1}\big)\Big)^\dagger.
 \end{align*}
For the remaining relations one proceeds in a similar way.
\end{proof}

This last proposition implies
\begin{pro} \label{ss}The following recursion relations for the left Laurent polynomials hold
\begin{align*}
 z(\varphi_1^L)^{(2k)}=&-\alpha_{1,2k+1}^L(\mathbb{I}-[\alpha_{2,2k}^R]^\dagger \alpha_{1,2k}^L)(\varphi_1^L)^{(2k-1)}
-\alpha_{1,2k+1}^L[\alpha_{2,2k}^R]^\dagger (\varphi_1^L)^{(2k)}-\alpha_{1,2k+2}^L (\varphi_1^L)^{(2k+1)}+(\varphi_1^L)^{(2k+2)},\\
z(\varphi_1^L)^{(2k+1)}=&(\mathbb{I}-[\alpha_{2,2k+1}^R]^\dagger \alpha_{1,2k+1}^L)(\mathbb{I}-[\alpha_{2,2k}^R]^\dagger \alpha_{1,2k}^L)(\varphi_1^L)^{(2k-1)}\\&+(\mathbb{I}-[\alpha_{2,2k+1}^R]^\dagger \alpha_{1,2k+1}^L)[\alpha_{2,2k}^R]^\dagger(\varphi_1^L)^{(2k)}-[\alpha_{2,2k+1}^R]^\dagger\alpha_{1,2k+2}^L(\varphi_1^L)^{(2k+1)}+[\alpha_{2,2k+1}^R]^\dagger(\varphi_1^L)^{(2k+2)},\\
z(\varphi_2^L)^{(2k)}(z)= &- \alpha_{2,2k+1}^R(\varphi_2^L)^{(2k-1)}(z)- \alpha_{2,2k+1}^R[\alpha_{1,2k}^L]^\dagger(\varphi_2^L)^{(2k)}- (\mathbb{I}-\alpha_{2,2k+1}^R [\alpha_{1,2k+1}^L ]^\dagger)(\varphi_2^L)^{(2k+1)}(z)\\
&+ (\mathbb{I}-\alpha_{2,2k+1}^R [\alpha_{1,2k+1}^L]^\dagger)(\mathbb{I}-\alpha_{2,2k+2}^R [\alpha_{1,2k+2}^L]^\dagger)(\varphi_2^L)^{(2k+2)}(z),\\
z(\varphi_2^L)^{(2k+1)}(z)=& (\varphi_2^L)^{(2k-1)}(z) + [\alpha_{1,2k}^L]^\dagger (\varphi_2^L)^{(2k)}(z)-[\alpha_{1,2k+1}^L]^\dagger \alpha_{2,2k+2}^R (\varphi_2^L)^{(2k+1)}(z)\\&+[\alpha_{1,2k+1}^L]^\dagger (\mathbb{I}- \alpha_{2,2k+2}^R[\alpha_{1,2k+2}^L]^\dagger)(\varphi_2^L)^{(2k+2)}(z).
\end{align*}
While for the right polynomials are
\begin{align*}
 z(\varphi_1^R)^{(2k)}=&-(\varphi_1^R)^{(2k-1)}\alpha_{1,2k+1}^L-(\varphi_1^R)^{(2k)} [\alpha_{2,2k}^R]^\dagger \alpha_{1,2k+1}^L-(\varphi_1^R)^{(2k+1)}\alpha_{1,2k+2}^L(\mathbb{I}-[\alpha_{2,2k+1}^R]^\dagger\alpha_{1,2k+1}^L )\\&+(\varphi_1^R)^{(2k+2)}(\mathbb{I}-[\alpha_{2,2k+2}^R]^\dagger\alpha_{1,2k+2}^L )(\mathbb{I}-[\alpha_{2,2k+1}^R]^\dagger\alpha_{1,2k+1}^L ),\\
z(\varphi_1^R)^{(2k+1)}=&(\varphi_1^R)^{(2k-1)}+(\varphi_1^R)^{(2k)} [\alpha_{2,2k}^R]^\dagger-(\varphi_1^R)^{(2k+1)}\alpha_{1,2k+2}^L [\alpha_{2,2k+1}^R]^\dagger\\&+(\varphi_1^R)^{(2k+2)}(\mathbb{I}-[\alpha_{2,2k+2}^R]^\dagger\alpha_{1,2k+2}^L )[\alpha_{2,2k+1}^R]^\dagger,\\
 z(\varphi_2^R)^{(2k)}=&-(\varphi_2^R)^{(2k-1)}(\mathbb{I}- \alpha_{2,2k}^R[\alpha_{1,2k}^L]^\dagger)\alpha_{2,2k+1}^R-(\varphi_2^R)^{(2k)}[ \alpha_{1,2k}^L]^\dagger\alpha_{2,2k+1}^R\\&-(\varphi_2^R)^{(2k+1)}\alpha_{2,2k+2}^R+(\varphi_2^R)^{(2k+2)}\\
z(\varphi_2^R)^{(2k+1)}=&(\varphi_2^R)^{(2k-1)}(\mathbb{I}-\alpha_{2,2k}^R[\alpha_{1,2k}^L ]^\dagger)(\mathbb{I}- \alpha_{2,2k+1}^R[\alpha_{1,2k+1}^L]^\dagger)+(\varphi_2^R)^{(2k)} [\alpha_{1,2k}^L]^\dagger(\mathbb{I}- \alpha_{2,2k+1}^R[\alpha_{1,2k+1}^L]^\dagger)\\
&-(\varphi_2^R)^{(2k+1)}\alpha_{2,2k+1}^R[\alpha_{1,2k+1}^L ]^\dagger+(\varphi_2^R)^{(2k+2) }[\alpha_{1,2k+1}^L]^\dagger.
\end{align*}
\end{pro}
We have written down just the recursion relations for $z$ and not those for $z^{-1}$,  which can be derived similarly to these ones. For the complete recursion expressions see
Appendix \ref{zeta-1}.

\begin{pro}\label{sss}The following relations hold true
\begin{align*}
 \left((\varphi_2^R)^{(2k)}\big(\bar z^{-1}\big)\right)^\dagger&=(\mathbb{I}-[\alpha_{2,2k}^R]^\dagger \alpha_{1,2k}^L)(\varphi_1^L)^{(2k-1)}(z)+[\alpha_{2,2k}^R]^\dagger(\varphi_1^L)^{(2k)}(z),\\
 \left((\varphi_2^R)^{(2k+1)}\big(\bar z^{-1}\big)\right)^\dagger&=-\alpha_{1,2k+2}^L(\varphi_1^L)^{(2k+1)}(z)+(\varphi_1^L)^{(2k+2)}(z),\\
 \left((\varphi_2^L)^{(2k)}\big(\bar z^{-1}\big)\right)^\dagger&=(\varphi_1^R)^{(2k-1)}(z)+(\varphi_1^R)^{(2k)}(z)[\alpha_{2,2k}^R]^\dagger,\\
 \left((\varphi_2^L)^{(2k+1)}\big(\bar z^{-1}\big)\right)^\dagger&=-(\varphi_1^R)^{(2k+1)}(z)\alpha_{1,2k+2}^L+(\varphi_1^R)^{(2k+2)}(z)(\mathbb{I}-[\alpha_{2,2k+2}^R]^\dagger \alpha_{1,2k+2}^L),\\
 \frac{1}{z}\left((\varphi_2^R)^{(2k)}\big(\bar z^{-1}\big) \right)^\dagger &= -[\alpha_{2,2k+1}^R]^\dagger(\varphi_1^L)^{(2k)}(z)+(\varphi_1^L)^{(2k+1)}(z),\\
 \frac{1}{z}\left((\varphi_2^R)^{(2k+1)}\big(\bar z^{-1}\big) \right)^\dagger &=(\mathbb{I}-\alpha_{1,2k+1}^L[\alpha_{2,2k+1}^R]^\dagger )(\varphi_1^L)^{(2k)}(z)+\alpha_{1,2k+1}^L(\varphi_1^L)^{(2k+1)}(z),\\
 \frac{1}{z}\left((\varphi_2^L)^{(2k+1)}\big(\bar z^{-1}\big) \right)^\dagger &=(\varphi_1^R)^{(2k)}(z)+(\varphi_1^R)^{(2k)}(z)\alpha_{1,2k+1}^L,\\
 \frac{1}{z}\left((\varphi_2^L)^{(2k)}\big(\bar z^{-1}\big) \right)^\dagger &=-(\varphi_1^R)^{(2k)}(z)[\alpha_{2,2k+1}^R]^\dagger+(\varphi_1^R)^{(2k+1)}(z)(\mathbb{I}-\alpha_{1,2k+1}^L[\alpha_{2,2k+1}^R]^\dagger ).
 \end{align*}
\end{pro}

\begin{proof}
 These relations appear just by substituting into \eqref{rec} the expressions of the blocks of $(J_H)^{\pm 1},\,C_{[0]},\,C_{[-1]}$.
\end{proof}
Using Proposition \ref{ss} and the matrix CMV recursion relations in Proposition \ref{sss} one derives the recursion relations for the matrix Szeg\H{o} polynomials.

\begin{align*}
 zP_{1,2l+1}^L(z)-P_{1,2l+2}^L(z)&=-\alpha_{1,2l+2}^L(P_{2,2l+1}^R(z))^*\\
 (P_{2,2l}^R(z))^*-(\mathbb{I}-(\alpha_{2,2l}^R)^\dagger \alpha_{1,2l}^L)(P_{2,2l-1}^R(z))^*&=(\alpha_{2,2l}^R)^\dagger P_{1,2l}^L(z)\\
 (P_{2,2l}^L(z))^*-(P_{2,2l-1}^L(z))^*(\mathbb{I}- \alpha_{1,2l}^R(\alpha_{2,2l}^L)^\dagger)&=P_{1,2l}^R(z)(\alpha_{2,2l}^L)^\dagger\\
 zP_{1,2l+1}^R(z)-P_{1,2l+2}^R(z)&=-(P_{2,2l+1}^L(z))^*\alpha_{1,2l+2}^L\\
 (P_{2,2l+1}^R(z))^*-(P_{2,2l}^R(z))^*&=(\alpha_{2,2l+1}^R)^\dagger zP_{1,2l}^L(z)\\
 P_{1,2l+1}^L(z)-(\mathbb{I}-\alpha_{1,2l+1}^L (\alpha_{2,2l+1}^R)^\dagger)zP_{1,2l}^L(z)&=\alpha_{1,2l+1}^L (P_{2,2l+1}^R(z))^*\\
 P_{1,2l+1}^R(z)-zP_{1,2l}^R(z)(\mathbb{I}-(\alpha_{2,2l+1}^L)^\dagger\alpha_{1,2l+1}^R)&=(P_{2,2l+1}^L(z))^*\alpha_{1,2l+1}^R \\
 (P_{2,2l}^L(z))^*-(P_{2,2l+1}^L(z))^*&=-zP_{1,2l}^R(z)(\alpha_{2,2l+1}^L)^\dagger
\end{align*}
Which after the the prescription
\begin{align*}
 x_N^l&:=\alpha_{1,N}^L & x_N^r&:=\alpha_{1,N}^R \\
 y_N^l&:=(\alpha_{2,N}^L)^\dagger & y_N^r&:=(\alpha_{2,N}^R)^\dagger
\end{align*}
coincide with the formulas in \cite{Cafasso}.

\subsection{Christoffel--Darboux Theory}\label{Sec-CD}
To conclude this section we show how the Gaussian factorization leads to the Christoffel--Darboux theorem for the matrix Laurent polynomials on the unit circle context.
In this particular situation we must consider two different cases. As we are working in a non Abelian situation we first have projections in the corresponding modules, ``orthogonal" in the ring (our blocks) context. Secondly, when the matrix measures are Hermitian and positive definite, we will have a scalar product, and the projections to consider are orthogonal indeed.
\subsubsection{Projections in modules}
Given  a right or left $\mathbb M_m$ module $M$ any idempotent endomorphism $\pi\in\operatorname{End}_{\mathbb M_m}(M)$, $\pi^2=\pi$, is called a projection. For any given projection $\pi$  we have
$\operatorname{Ker}\pi=\operatorname{Im}(1-\pi)$,   $\operatorname{Ker}(1-\pi)=\operatorname{Im}\pi$, and the following direct decomposition holds: $M=\operatorname{Im}\pi\oplus \operatorname{Im}(1-\pi)$. Two projections  $\pi$ and $\pi'$ are said to be orthogonal if $\pi\pi'=0$; observe that $(1-\pi)$ is  idempotent and moreover  orthogonal to $\pi$.
Orthogonality is not related here to any inner product so far, is just a construction in the module. In particular,  in our discussion of  matrix Laurent polynomials we introduce the following free modules
\begin{align*}
  \Lambda_{[l]}:=\mathbb{M}_m\Big\{\chi^{(j)}\Big\}_{j=0}^l=\begin{cases}
\Lambda_{m,[-k,k]}, &l=2k,\\
\Lambda_{m,[-k-1,k]}, &l=2k+1.
  \end{cases}
\end{align*}
That we can consider as a left free module, when multiplied by the left, and denoted by $V_{[l+1]}$ or as a right free module (when multiplication by matrices is performed by the right) and denoted by by $W_{[l+1]}$. We will denote by $V=\lim\limits_{\rightarrow}V_{[l]}$ and $W=\lim\limits_{\rightarrow}W_{[l]}$ the corresponding direct limits, the left and right modules of matrix Laurent polynomials.
The bilinear form
\begin{align*} 
 G(f,g)&=\langle\!\langle g^{\dagger}, f\rangle\!\rangle_{L}=\langle\!\langle f^{\dagger},g\rangle\!\rangle_{R}=\oint_\mathbb{T}f(z)\frac{\d\mu(z)}{\operatorname{i}z}g(z), &
G_{i,j}&:=\oint_{\mathbb{T}}\chi^{(i)}(z) \frac{\d\mu(z)}{\operatorname{i}z} (\chi^{(j)}(z))^{\top},
  \end{align*}
fulfills
  \begin{align*}
G=      \eta g^R&= g^L \eta.
  \end{align*}
  This can be understood as a change of basis in the left and right modules $W_{[l]}$ and $V_{[l]}$; the left moment matrix can be understood as the matrix of the bilinear form $G$
  when on the left module $W_{[l]}$ we apply the isomorfism or change o basis represented by the $\eta$ matrix. Similarly,  the right moment matrix can be understood as the matrix of the bilinear form $G$
  when on the right module $V_{[l]}$ we apply the isomorfism represented by the $\eta$ matrix.  Observe that the $G$ dual vectors introduced in  Appendix \ref{modules} are of the form
\begin{align*}
\left((\varphi_1^L)_{j}\right)^*&=(\varphi_2^L)_{j}^{\dagger}, &
\left((\varphi_2^L)_{j}^{\dagger} \right)^*&=(\varphi_1^L)_{j},\\
\left((\varphi_1^R)_{j}\right)^*&=(\varphi_2^R)_{j}^{\dagger},  &
\left((\varphi_2^R)_{j}^{\dagger} \right)^*&=(\varphi_1^R)_{j}.
\end{align*}
Thus, following  Appendix \ref{modules} we consider the ring of $G$ projections in these left and right modules
\begin{definition}\label{kernels}
\begin{enumerate}
  \item
The Christoffel--Darboux projectors
\begin{align*}
\pi_L^{[l]}&: V  \longrightarrow  V_{[l]} , &
\pi_R^{[l]}&: W\longrightarrow W_{[l]},
\end{align*}
are the ring left and right projections associated to the bilinear form $G$.
  \item The matrix Christoffel-Darboux kernels are
\begin{align}\label{CDkernel}
  K^{L,[l]}(z,z')&:=\sum_{k=0}^{l-1}[(\varphi^L_2)^{(k)}(z)]^\dagger (\varphi^L_{1})^{(k)}(z'), &
K^{R,[l]}(z,z')&:=\sum_{k=0}^{l-1}(\varphi^R_{1})^{(k)}(z)[(\varphi^R_{2})^{(k)}(z')]^\dagger.
\end{align}
\end{enumerate}
\end{definition}
\begin{pro}
  For the projections and matrix Christoffel-Darboux kernels introduced in Definition \ref{kernels} we have the following relations
  \begin{align*}
   ( \pi_L^{[l]}f)(z)&=\int_{\mathbb T}f(z') \frac{\d\mu(z')}{\operatorname{i}z'} K^{L,[l]}(z',z)=\sum_{k=0}^{l-1}\langle\!\langle (\varphi_2^L)^{(k)},f \rangle\! \rangle_L (\varphi_1^L)^{(k)}(z),& \forall f\in V,\\
   [ ( \pi_L^{[l]}f)(z)]^\dagger&= \oint_\T K^{L,[l]}(z,z') \frac{\d\mu(z')}{\operatorname{i}z'} [f(z')]^\dagger =\sum_{k=0}^{l-1}[(\varphi_2^L)^{(k)}(z)]^\dagger\langle\!\langle f,(\varphi_1^L)^{(k)} \rangle \!\rangle_L, & \forall f\in V,\\
   ( \pi_R^{[l]}f)(z)&=\oint_\T K^{R,[l]}(z,z') \frac{\d\mu(z')}{\operatorname{i}z'} f(z')= \sum_{k=0}^{l-1}(\varphi_1^R)^{(k)}(z)\langle\!\langle (\varphi_2^R)^{(k)},f \rangle\! \rangle_R,&\forall\in W,\\
   [ ( \pi_R^{[l]}f)(z)]^\dagger&= \oint_\T [f(z')]^\dagger \frac{\d\mu(z')}{\operatorname{i}z}K^{R,[l]}(z',z)=\sum_{k=0}^{l-1}\langle\!\langle f,(\varphi_1^R)^{(k)} \rangle \!\rangle_R [(\varphi_2^R)^{(k)}(z)]^\dagger,&\forall\in W.
  \end{align*}
\end{pro}
\begin{pro}
The Christoffel--Darboux kernels have the reproducing property
\begin{align}
K^{H,[l]}(z,y)&=\oint_\T K^{H,[l]}(z,z')\frac{\d\mu(z')}{\operatorname{i} z'} K^{H,[l]}(z',y), & H&=L,R.
\end{align}
\end{pro}
\begin{proof}
This follows from the idempotency  property of the $\pi$'s.
\end{proof}
Moreover,
\begin{pro}If the matrix measure $\mu$ is Hermitian  then
\begin{enumerate}
\item The followings expansions are satisfied
  \begin{align*}
   ( \pi_L^{[l]}f)(z)&=\sum_{k=0}^{l-1}\langle\!\langle (\varphi_1^L)^{(k)}, f \rangle \!\rangle_L (h^{L}_k)^{-1}(\varphi_1^L)^{(k)}(z),& \forall f\in V,\\
   ( \pi_R^{[l]}f)(z)&= \sum_{k=0}^{l-1}(\varphi_1^R)^{(k)}(z)(h^{R}_k)^{-1}\langle\!\langle (\varphi_1^R)^{(k)},f \rangle \!\rangle_R,&\forall f\in W.
  \end{align*}
  \item The following Hermitian type property holds for the projectors
  \begin{align*}
    \Langle \pi^{[l]}_Hf,g\Rangle_H&=  \Langle f,\pi^{[l]}_Hg\Rangle_H, & H&=R,L.
  \end{align*}
 \end{enumerate}
\end{pro}

When the matrix measure is Hermitian and positive definite we have a standard scalar product and a complex Hilbert space, and the projections $\pi^{[l]}_H$ are orthogonal projections --not only in the module  but in the geometrical sense as well-- to the subspaces of truncated matrix Laurent polynomials; notice that there are two different, however equivalent, scalar products and distances involved. In this situation, as is well known, these projections
give the best approximation within the truncated Laurent polynomials and the corresponding left and right distances.
\subsubsection{The Christoffel-Darboux type formulae}
\begin{theorem} \label{th:CD}For $\bar z z'\neq 1$ the  matrix Christoffel--Darboux kernels fulfill
\begin{align}\label{quer1}\begin{aligned}
  K^{L,[2l]}(z,z')(1-\bar{z}z')=& (\varphi_1^R)^{(2l)}(\bar z^{-1})h_{2l}^R (h^R_{2l-1})^{-1}(\varphi_1^L)^{(2l-1)}(z')-(\varphi_1^R)^{(2l-1)}(\bar z^{-1})(\varphi_1^L)^{(2l)}(z')\\
 =&\left[[z(\varphi_2^L)^{(2l+1)}(z)]^\dagger h_{2l+1}^R- [z(\varphi_2^L)^{(2l)}(z)]^\dagger h_{2l}^L \alpha_{1,2l+1}^R\right] (h^R_{2l-1})^{-1}(\varphi_1^L)^{(2l-1)}(z')- \\
&-\left[[z(\varphi_2^L)^{(2l-2)}(z)]^\dagger + [z(\varphi_2^L)^{(2l-1)}(z)]^\dagger [\alpha_{2,2l-1}^R]^\dagger\right](\varphi_1^L)^{(2l)}(z'),\\
 K^{L,[2l+1]}(z,z')(1-\bar{z}z')=&[z(\varphi_2^L)^{(2l+1)}(z)]^\dagger h_{2l+1}^R (h_{2l}^{R})^{-1}[(\varphi_2^R)^{(2l)}]^\dagger(z'^{-1})-[z(\varphi_2^L)^{(2l)}(z)]^\dagger[(\varphi_2^R)^{(2l+1)}]^\dagger(z'^{-1})\\
 =&[z(\varphi_2^L)^{(2l+1)}(z)]^\dagger h_{2l+1}^R \left[ (h_{2l-1}^{R})^{-1}(\varphi_1^L)^{(2l-1)}(z')+ [\alpha_{2,2l}^L]^\dagger (h_{2l}^{L})^{-1}(\varphi_1^L)^{(2l)}(z')\right]\\
 &-z(\varphi_2^L)^{(2l)}(z)]^\dagger \left[(\varphi_1^L)^{(2l+2)}(z')-\alpha_{1,2l+2}^L(\varphi_1^L)^{(2l+1)}(z')  \right],
\end{aligned}\\
\label{quer2}
\begin{aligned}
K^{R,[2l]}(z,z')(1-\bar{z'}z)=&[(\varphi_2^L)^{(2l)}]^\dagger(z^{-1}) h_{2l}^L (h_{2l-1}^{L})^{-1}[(\varphi_2^R)^{(2l-1)}(z')]^\dagger-[(\varphi_2^L)^{(2l-1)}]^\dagger(z^{-1})[(\varphi_2^R)^{(2l)}(z')]^\dagger\\
=&z\left[(\varphi_1^R)^{(2l+1)}(z)h_{2l+1}^L-(\varphi_1^R)^{(2l)}(z)h_{2l}^R[\alpha_{2,2l+1}^L]^\dagger  \right] (h_{2l+1}^{L})^{-1}[(\varphi_2^R)^{(2l-1)}(z')]^\dagger\\
&-z\left[(\varphi_1^R)^{(2l-2)}(z)-(\varphi_1^R)^{(2l-1)}(z)\alpha_{1,2l-1}^L  \right][(\varphi_2^R)^{(2l)}(z')]^\dagger,\\
 K^{R,[2l+1]}(z,z')(1-\bar{z'}z)=&z(\varphi_1^R)^{(2l+1)}(z)h_{2l+1}^L (h_{2l}^{L})^{-1}(\varphi_1^L)^{(2l-1)}(\bar z'^{-1})-z(\varphi_1^R)^{(2l)}(z)(\varphi_1^L)^{(2l+1)}(\bar z'^{-1})\\
=&z(\varphi_1^R)^{(2l+1)}(z)h_{2l+1}^L \left[ (h_{2l-1}^{L})^{-1}(\varphi_2^R)^{(2l-1)}(z')]^\dagger+\alpha_{1,2l}^R (h_{2l}^{R})^{-1}[(\varphi_2^R)^{(2l)}(z')]^\dagger\right]\\
 &-z(\varphi_1^R)^{(2l)}(z)\left[(\varphi_2^R)^{(2l+2)}(z')]^\dagger - [\alpha_{2,2l+2}^R]^\dagger [(\varphi_2^R)^{(2l+1)}(z')]^\dagger\right].
\end{aligned}
\end{align}
\end{theorem}
\begin{proof}
  See Appendix \ref{proofs}.
\end{proof}

In terms of the matrix Sezg\H{o} polynomials we have
\begin{cor}The matrix Christoffel--Darboux kernels can be expressed in terms of the matrix Sezg\H{o} polynomials as follows
\begin{align*}
 K^{L,[2l]}(z,z')&=\frac{(\bar{z}z'^{-1})^l}{1-\bar{z}z'}\left[ P_{1,2l}^R(\bar{z}^{-1}) (h^R_{2l-1})^{-1} (P_{2,2l-1}^R)^*(z')-
 (P_{2,2l-1}^L)^*(\bar{z}^{-1}) (h^L_{2l-1})^{-1} P_{1,2l}^L(z') \right],\\
 K^{L,[2l+1]}(z,z')&=\frac{\bar{z}^{l+1}z'^{-l}}{1-\bar{z}z'}\left[ P_{1,2l+1}^R(\bar{z}^{-1}) (h^R_{2l})^{-1} (P_{2,2l}^R)^*(z')-
 (P_{2,2l}^L)^*(z^{-1}) (h^L_{2l})^{-1} P_{1,2l+1}^L(z') \right],\\
 K^{R,[2l]}(z,z')&=\frac{(z\bar{z'}^{-1})^l}{1-\bar{z'}z}\left[ [P_{2,2l}^L(\bar{z}^{-1})]^\dagger (h^L_{2l-1})^{-1} [(P_{1,2l-1}^L)^*(z')]^\dagger-
 [(P_{1,2l-1}^R)^*(\bar{z}^{-1})]^\dagger (h^R_{2l-1})^{-1}[ P_{2,2l}^R(z')]^\dagger \right],\\
 K^{R,[2l+1]}(z,z')&=\frac{z^{l+1}\bar{z'}^{-l}}{1-\bar{z'}z}\left[ [P_{2,2l+1}^L(\bar{z}^{-1})]^\dagger (h^L_{2l})^{-1} [(P_{1,2l}^L)^*(z')]^\dagger-
 [(P_{1,2l}^R)^*(\bar{z}^{-1})]^\dagger (h^R_{2l})^{-1} [P_{1,2l+1}^R(z')]^\dagger \right],
\end{align*}
where we assume that $\bar zz'\neq 1$.
\end{cor}

As we have just seen, letting an operator act to the left or to the
right and comparing the two results has been very successful with
$J_{K}$. Actually we still have the operators $C_{0}, C_{-1}$ to which we can also apply the same procedure
to get some other interesting relations for the CD kernels.
\begin{pro}
 The next relations between $K^{L}$ and $K^{R}$ hold
\begin{align*}
K^{R,[2l+1]}(z,\frac{1}{\bar{z}'})&=K^{L,[2l+1]}(\frac{1}{\bar{z}},z'),\\
K^{R,[2l+2]}(z,\frac{1}{\bar{z}'})-K^{L,[2l+2]}(\frac{1}{\bar{z}},z')&=
(\varphi^R_{1})^{(2k+1)}(z)(\varphi^L_{1})^{(2k+2)}(z')-
(\varphi^R_{1})^{(2k+2)}(z)\left(\mathbb{I}-[\alpha_{2,2k+2}^R]^\dagger\alpha_{1,2k+2}^L \right)(\varphi^L_{1})^{(2k+1)}(z'),\\
\frac{1}{z'}K^{R,[2l+1]}(z,\frac{1}{\bar{z}'})-\frac{1}{z}K^{L,[2l+1]}(\frac{1}{\bar{z}},z')&=
(\varphi^R_{1})^{(2k)}(z)(\varphi^L_{1})^{(2k+1)}(z')-
(\varphi^R_{1})^{(2k+1)}(z)\left(\mathbb{I}-\alpha_{1,2k+1}^L[\alpha_{2,2k+1}^R]^\dagger \right)(\varphi^L_{1})^{(2k)}(z'),\\
\frac{1}{z'}K^{R,[2l+2]}(z,\frac{1}{\bar{z}'})&=\frac{1}{z}K^{L,[2l+2]}(\frac{1}{\bar{z}},z').
\end{align*}
\end{pro}
\begin{proof}
 The first two relations arise when comparing the action of $C_{0}$ to the left or to the right in
 \begin{align*}
  \phi_{1}^{R}(z)C_{0}\phi_{1}^{L}(z')
 \end{align*}
and truncating the expressions up to $[2k+1]$ (first relation) or $[2k+2]$ (second one).
The other two relations are obtained proceeding in the same way but using $C_{-1}$ instead.
\end{proof}

\section{MOLPUC and two dimensional Toda type hierarchies}\label{Sec-integrability}

Once we have explored how the Gauss--Borel factorization of block CMV moment matrices  leads to the algebraic theory of MOLPUC, we are ready to show how this approach also connects these polynomials to integrable hierarchies of Toda type. We first introduce convenient deformations of the moment matrices, that as we will show correspond to deformations of the matrix measure.
With these we will construct wave functions, Lax equations, Zakharov--Shabat equations, discrete flows and Darboux transformations and Miwa transformations. These last transformations will lead to interesting relations between the matrix Christoffel--Darboux kernels, Miwa shifted MOLPUC and their ``norms". The integrable equations that we derive are a non-Abelian version of the Toeplitz lattice or non-Abelian ALL equations that extend, in the partial flows case,  those of \cite{Cafasso} --appearing these last ones in what we denominate total flows--.
\subsection{2D Toda continuous flows}\label{Sec-Toda-0}
In order to construct  deformation matrices which  will  act on the moment matrices
(resulting in a deformation of the matrix measure) we first introduce  some definitions.
\begin{definition}
\begin{enumerate}
\item Given the diagonal matrices $t^H_j=\diag(t^H_{j,1},\dots,t^H_{j,m})\in\diag _m$, $j=0,1,2,\dots$, $H=L,R$ and $t^H_{j,a}\in\C$ we introduce
 \begin{align*}
  t^L&:= \begin{pmatrix}
  t^L_0, t^L_1, t^L_2, \dots
  \end{pmatrix}, &
  t^H&:= \begin{pmatrix}
  t^R_0,
  t^R_1,
  t^R_2,
  \dots
  \end{pmatrix}^\top,
 \end{align*}
we also impose $t^H_0=0$.
 \item  We also consider the CMV ordered Fourier monomial  vector but evaluated in $\Upsilon$
  \begin{align*}
 [\chi(\Upsilon)]^\top&=\begin{pmatrix}\mathbb{I},\Upsilon,\Upsilon^{-1},\Upsilon^{2},\Upsilon^{-2},\dots \end{pmatrix}.
 \end{align*}
\item With this we introduce
 \begin{align*}
 t^L *\chi(\Upsilon)&:= \sum_{j=0}^\infty t^L_j \chi(\Upsilon)^{(j)}, &
 [\chi(\Upsilon)]^\top * t^R&:=\sum_{j=0}^\infty [\chi(\Upsilon)^{(j)}]^\top t^R_j.
 \end{align*}
 The products in the above expressions are by blocks; i.e. the factors in  $\mathbb{M}_m$  multiply $\mathbb{M}_m$ block of the $\mathbb{M}_{\infty}$ block matrix.
\item  The deformation matrices are
 \begin{align*}
  W_0(t^L)&:=\exp\Big(t^L *\chi(\Upsilon)\Big), &
  V_0(t^R)&:=\exp\Big([\chi(\Upsilon)]^\top * t^R\Big).
 \end{align*}
\item The $t$-dependent deformation of  moment matrices, $g^{H}(t)$, $H=L,R$, and their Gauss--Borel factorization are considered
\begin{align}
g^{L}(t)&:=W_0(t^L) g^L [V_0(-\eta t^R)]^{-1}, & g^{L}(t)&=(S_1(t))^{-1} S_2(t),\\
g^{R}(t)&:=[W_0(-t^L \eta)]^{-1} g^R V_0(t^R), & g^{R}(t)&=Z_2(t)(Z_1(t))^{-1}.
\end{align}
 \end{enumerate}
\end{definition}
\begin{pro}
\begin{enumerate}
  \item The deformed moment matrices can be understood as the moment matrices for a deformed ($t$-dependent) measure given by
\begin{align*}
\d\mu(z,t):= \exp \Big( t_L \chi(z) \Big)\d\mu(z) \exp \Big(\chi(z)^{\top} t_R \Big),
\end{align*}
 with  the deformed Fourier series of the evolved matrix measure given by
\begin{equation}\label{Fmut}
 F_{\mu(t)}(z):= \exp \Big( t_L \chi(z) \Big)F(z) \exp \Big(\chi(z)^{\top} t_R \Big).
\end{equation}
\item The  Hermitian character of the matrix measure is preserved by the deformation whenever  $t^L=(t^R)^\dagger \eta$.
\end{enumerate}
\end{pro}
Observe that in this paper we introduce a slightly different set  of flows or deformations of the measure than those in the scalar case \cite{carlos}. Despite that in that scalar situation both definitions give the very same flows that is not the case in this non-Abelian scenario, as in this case we have deformation matrices multiplying at the left and right of the initial matrix measure, and the order is relevant.

\subsubsection{The Gauss--Borel approach to integrability}
We consider the elements that enable us to construct the integrable hierarchy
\begin{definition}
   \begin{enumerate}
\item  Left and right  wave matrices
\begin{align}\label{wave}
  W_1^{L}(t)&:=S_1(t)W_{0}(t ^L), & W_2^{L}(t)&:=S_2(t)V_{0}(-\eta t^R),&
  W_1^{R}(t)&:=V_{0}(t^R)Z_1(t), & W_2^{R}(t)&:=W_{0}(-t^L \eta)Z_2(t).
\end{align}
\item Left and right wave and adjoint wave functions
\begin{align*}
 \Psi_1^L(z,t)&=W_1^L(t)\chi(z), & (\Psi_1^L)^*(z,t)&=[(W_1^L)^{-1}(t)]^\dagger\chi^*(z),\\
 \Psi_2^L(z,t)&=W_2^L(t)\chi^*(z), & (\Psi_2^L)^*(z,t)&=[(W_2^{L})^{-1}(t)]^\dagger\chi(z),\\
 \Psi_1^R(z,t)&=\chi(z)^\top W_1^R(t), & (\Psi_1^R)^*(z,t)&=\chi^*(z)^\top [(W_1^{R})^{-1}(t)]^\dagger, \\
 \Psi_2^R(z,t)&=\chi^*(z)^\top W_2^R(t), & (\Psi_2^R)^*(z,t)&=\chi(z)^\top [(W_2^{R})^{-1}(t)]^\dagger.
\end{align*}
\item Left and right Jacobi vector of matrices (using our previous notation)
\begin{align*}
 \chi\Big(J_H(t)\Big)& :=\begin{pmatrix}
             \mathbb{I}\\
             \big(J^H(t)\big)^{-1}\\
             J^H(t)\\
             \big(J^H(t)\big)^{-2}\\
            \big( J^H(t)\big)^{2}\\
             \vdots
            \end{pmatrix}, &
            H&= L,R.
\end{align*}
\item Projection operators, $a=1,\dots,m$
 \begin{align}\label{projections}
P^{(H,H')}_a&:=\begin{cases}
S_1 E_{aa}S_1^{-1},& H=L, H'=L,\\
S_2 E_{aa}S_2^{-1},& H=R, H'=L,\\
Z_2^{-1} E_{aa}Z_2,& H=L, H'=R,\\
Z_1^{-1} E_{aa}Z_1,& H=R, H'=R.
\end{cases}
\end{align}
\item Left and right Lax matrices
\begin{align}\label{Lax}
  L_1(t)&:=S_1(t)\Upsilon S_1(t)^{-1}=S_2(t)\Upsilon S_2(t)^{-1}=J^L(t), &
  R_1(t)&:=Z_1(t)^{-1}\Upsilon^{-1} Z_1(t)=Z_2(t)^{-1}\Upsilon^{-1} Z_2(t)=J^R(t), \\
  L_2(t)&:=S_2(t)\Upsilon ^{-1}  S_2(t)^{-1}=S_1(t)\Upsilon ^{-1}  S_1(t)^{-1}=\big(J^L(t)\big)^{-1}&
  R_2(t)&:=Z_2(t)^{-1}\Upsilon Z_2(t)=Z_1(t)^{-1}\Upsilon Z_1(t)=\big(J^R(t)\big)^{-1}.
\end{align}

\item Zakharov--Shabat matrices
\begin{align}\label{def:ZS}
B^{(H,H')}_{j,a}&:=\begin{cases}
  \big(S_1 E_{aa}(\chi(\Upsilon))^{(j)}S_1^{-1}\big)_+,& H=L, H'=L,\\
    -\big(S_2 E_{aa}(\chi(\Upsilon))^{(j)}S_2^{-1}\big)_-,& H=R, H'=L,\\
      -\big(Z_2^{-1} E_{aa}\big(\chi\big(\Upsilon^{-1}\big)\big)^{(j)}Z_2\big)_+,& H=L, H'=R,\\
        \big(Z_1^{-1} E_{aa}\big(\chi\big(\Upsilon^{-1}\big)\big)^{(j)}Z_1\big)_-,& H=R, H'=R,
\end{cases}&
B^{(H,H')}_{j}&:=\begin{cases}
  \big((\chi(J^L))^{(j)}\big)_+,& H=L, H'=L,\\
    -\big((\chi(J^L))^{(j)}\big)_-,& H=R, H'=L,\\
      -\big( (\chi((J^R)^{-1})^{(j)}\big)_+,& H=L, H'=R,\\
        \big((\chi((J^R)^{-1})^{(j)}\big)_-,& H=R, H'=R.
\end{cases}
\end{align}

\item A time dependent intertwining operator
\begin{align}\label{Cpt}
 C_{[p]}(t)&=Z_2(t)^{-1}\eta \Upsilon^p S_1(t)^{-1}= Z_1(t)^{-1}\eta \Upsilon^p S_2(t)^{-1}.
\end{align}

     \end{enumerate}
\end{definition}
Observe that
\begin{align}\label{ye}
g^{L}&= (W_1^L(t))^{-1}W_2^{L}(t), & g^{R}&=W_2^{R}(t) (W_1^R)^{-1}(t).
\end{align}

\begin{definition}
For $H=R,L$ we introduce the total derivatives
\begin{align*}
{\partial_{H,j}}&:=\sum_{a=1}^m {\frac{\partial}{\partial t^H_{j_{aa}}}}.
\end{align*}
\end{definition}

We now present the linear systems, Lax equations and Zakharov--Shabat equations that characterize integrability
\begin{pro}\label{integrable}
 The following  equations hold
\begin{enumerate}
 \item Linear systems for the wave matrices
 \begin{align*}
    {\frac{\partial W_i^{L}}{\partial t^H_{j,a}}} &=B^{H,L}_{j,a}W_i^{L}, &  {\partial_{H,j}}W_{i}^{L}&=B^{H,H'}_{j}W_{i}^{L},\\
     {\frac{\partial W_i^{R}}{\partial t^H_{j,a}}} &=W_i^RB^{H,R}_{j,a} &  {\partial_{H,j}}W_{i}^{R}&=W_{i}^{H'}B^{H,R}_{j},
 \end{align*}
for $ i=1,2$, $H=L,R$, $a=1,\dots, m$,  $j=0,1,\dots$.

\item Lax equations:
\begin{align*}
    {\frac{\partial J^{H'}}{\partial t^H_{j,a}}} &=[B^{H,H'}_{j,a},J^{H'}],                 &   \partial_{H,j} J^{H'}&=[B^{H,H'}_j,J^{H'}], \\
    {\frac{\partial P^{H',H''}_b}{\partial t^{H}_{j,a}}} &=[B^{H,H''}_{j,a},P^{H',H''}_b],  &   \partial_{H,j} P^{H',H''}_b&=[B^{H,H''}_{j},P^{H',H''}_b]
\end{align*}
with $H,H',H''=L,R$, $a,b=1,\dots, m$ and $j=0,1,\dots$.
\item Evolution of the dressed intertwining operator
\begin{align*}
\frac{  \partial C_{[p]}}{\partial t^H_{j,a}}&=-B_{j,a}^{H,R}C_{[p]}-C_{[p]}B^{H,L}_{j,a}, &
\frac{  \partial C_{[p]}}{\partial t^H_{j}}=-B_{j}^{H,R}C_{[p]}-C_{[p]}B^{H,L}_{j}.
\end{align*}
\item Zakharov-Shabat equations
\begin{align*}
     {\frac{\partial B^{H_1,H'}_{j_1,b_1}}{\partial t^{H_2}_{j_2,a_2}}}-   \frac{\partial B^{H_2,H'}_{j_2,b_2}}{\partial t^{H_1}_{j_1,a_1}}+[B^{H_1,H'}_{j_1,b_1},B^{H_2,H'}_{j_2,b_2} ]=0.
\end{align*}
\end{enumerate}
\end{pro}
\begin{proof}
  See Appendix \ref{proofs}
\end{proof}

From the definitions of the wave functions, the action of $\Upsilon$ on $\chi$, the expression \eqref{Fmut},
and the relations \eqref{CF} it follows that
\begin{pro}
 The wave functions are linked to the CMV polynomials and the Fourier series of the measure as  follows
\begin{align}\label{WF}
\begin{aligned}
 \Psi_1^L(z,t)&=\phi_1^L(z,t)\exp \big( t^L \chi(z) \big), &
(\Psi_1^L)^*(z,t)&
=2\pi z^{-1} \phi_2^L(z^{-1},t)F_{\mu}^\dagger(z) \exp \big( -\overline{t^L}\chi(z)\big),\\
\Psi_2^L(z,t)&
2\pi z^{-1} \phi_1^L(z^{-1},t)F_{\mu}(z^{-1}) \exp \big( -\chi(z^{-1})^{\top}t^R\big), &
(\Psi_2^L)^*(z,t)&=\phi_2^L(z,t)\exp \big( \chi(z^{-1})^{\top}\overline{t^R}\big),\\
\Psi_1^R(z,t)&=\exp \big( \chi(z)^{\top}t^R\big)\phi_1^R(z,t), &
(\Psi_1^R)^*(z,t)&
2\pi z^{-1} \exp \big(-\chi(z)^{\top}\overline{t^R} \big) F_{\mu}^\dagger(z)(\phi_2^R)(z^{-1},t),\\
\Psi_2^R(z,t)&
2\pi z^{-1} \exp \big(-t^L\chi(z^{-1}) \big)F_{\mu}(z^{-1})\phi_1^R(z^{-1},t),&
(\Psi_2^R)^*(z,t)&=\exp \big( \overline{t^L}\chi(z^{-1})\big)\phi_2^R(z,t).
\end{aligned}
\end{align}
These wave functions are also  eigen-functions of the Lax  matrices \eqref{Lax} $L_i,R_i$ for $i=1,2$,
\begin{align*}
 L_i \Psi_i^L&=z\Psi_i^L &  \Psi_i^R R_i&=z\Psi_i^R,\\
 L_i^\dagger (\Psi_i^L)^*&=z(\Psi_i^L)^* & (\Psi_i^R)^* R_i^\dagger&=z(\Psi_i^R)^*.
\end{align*}
\end{pro}

\subsubsection{CMV matrices and matrix Toeplitz lattice}\label{Sec-ALL}
For the CMV ordering of the Laurent basis, the Lax equations acquire a dynamical non-linear system form that is the matrix
version, in the CMV context,  of the Toeplitz lattice developed in \cite{Adler-Van-Moerbecke-Toeplitz}. In \cite{Cafasso} Mattia Cafasso presented a non Abelian extension of the TL which corresponds to our total flows.
The partial flows presented here are, to our knowledge, new in the literature.
\begin{pro}
The Lax equations result in the following non-linear dynamical system for the matrix Verblunsky coefficients $H=L,R$,
\begin{itemize}
\item Partial flows
\begin{align*}
{\frac{\partial}{\partial t^L_{1,a}}}\alpha_{1,k}^R&=-(h_{k-1}^L)^{-1}\alpha_{1,k-1}^L E_{a,a}h_k^R,      &
{\frac{\partial}{\partial t^L_{1,a}}}[\alpha_{2,k}^L]^{\dagger}&=(h_{k-1}^R)^{-1} E_{a,a} [\alpha_{2,k+1}^R]^{\dagger}  h_k^L,  \\
{\frac{\partial}{\partial t^R_{1,a}}}[\alpha_{2,k}^R]^{\dagger}&=h_{k}^R  [\alpha_{2,k+1}^L]^{\dagger} E_{a,a} (h_{k-1}^L)^{-1}, &
{\frac{\partial}{\partial t^R_{1,a}}}\alpha_{1,k}^L&=-(h_{k}^L) E_{a,a} \alpha_{1,k-1}^R (h_{k-1}^R)^{-1}, \\
 {\frac{\partial}{\partial t^L_{2,a}}}\alpha_{1,k}^R&=(h_{k-1}^L)^{-1} E_{a,a} \alpha_{1,k+1}^L h_k^R,  &
{\frac{\partial}{\partial t^L_{2,a}}}[\alpha_{2,k}^L]^{\dagger}&=-(h_{k-1}^R)^{-1}  [\alpha_{2,k-1}^R]^{\dagger} E_{a,a} h_{k}^L, \\
{\frac{\partial}{\partial t^R_{2,a}}}[\alpha_{2,k}^R]^{\dagger}&=-h_{k}^R E_{a,a} [\alpha_{2,k-1}^L]^{\dagger}  (h_{k-1}^L)^{-1}, &
 {\frac{\partial}{\partial t^R_{2,a}}}\alpha_{1,k}^L&=h_{k}^L \alpha_{1,k+1}^R E_{a,a} (h_{k-1}^R)^{-1},
\end{align*}
\begin{align*}
{\frac{\partial}{\partial t^L_{1,a}}}h_k^L&=-\alpha_{1,k}^L E_{a,a} (\alpha_{2,k+1}^R)^{\dagger}h_k^L, &
{\frac{\partial}{\partial t^R_{1,a}}}h_k^R&=-h_k^R (\alpha_{2,k-1}^L)^{\dagger} E_{a,a} \alpha_{1,k}^R, \\
{\frac{\partial}{\partial t^L_{2,a}}}h_k^R&=-(\alpha_{2,k}^R)^{\dagger} E_{a,a} \alpha_{1,k+1}^L h_k^R, &
{\frac{\partial}{\partial t^R_{2,a}}}h_k^L&=-h_k^L \alpha_{1,k+1}^R E_{a,a} (\alpha_{2,k}^L)^{\dagger}.
\end{align*}
\item Total flows \cite{Cafasso}
\begin{align*}
{\partial_{H,1}}[\alpha_{2,k}^R]^\dagger  &= [\alpha_{2,k+1}^R]^\dagger (\mathbb{I}-\alpha_{1,k}[\alpha_{2,k}^R]^\dagger), &
{\partial_{H,1}} \alpha_{1,k}^L  &= -(\mathbb{I}-\alpha_{1,k}^L[\alpha_{2,k}^R]^\dagger)\alpha_{1,k-1}^L,           \\
{\partial_{H,1}} [\alpha_{2,k}^L]^\dagger &= (\mathbb{I}-[\alpha_{2,k}^L]^\dagger \alpha_{1,k}^R)[\alpha_{2,k+1}^L]^\dagger,      &
{\partial_{H,1}} \alpha_{1,k}^R  &= -\alpha_{1,k-1}^R(\mathbb{I}-[\alpha_{2,k}^L]^\dagger \alpha_{1,k}^R)           \\
{\partial_{H,2}}[\alpha_{2,k}^R]^\dagger &= -(\mathbb{I}-[\alpha_{2,k}^R]^{\dagger} \alpha_{1,k}^L)[\alpha_{2,k-1}^R]^{\dagger},    &
{\partial_{H,2}} \alpha_{1,k}^L  &=\alpha_{1,k+1}^L(\mathbb{I}-[\alpha_{2,k}^R]^\dagger \alpha_{1,k}^L),           \\
{\partial_{H,2}} [\alpha_{2,k}^L]^\dagger &=-[\alpha_{2,k-1}^L]^{\dagger}(\mathbb{I}-\alpha_{1,k}^R[\alpha_{2,k}^L]^{\dagger}),       &
{\partial_{H,2}} \alpha_{1,k}^R  &=(\mathbb{I}-\alpha_{1,k}^R[\alpha_{2,k}^L]^{\dagger})\alpha_{1,k+1}^R,            \\
\end{align*}
\begin{align*}
{\partial_{H,1}}h_k^L&=-\alpha_{1,k}^L (\alpha_{2,k+1}^R)^{\dagger}h_k^L, &
{\partial_{H,1}}h_k^R&=-h_k^R (\alpha_{2,k-1}^L)^{\dagger} \alpha_{1,k}^R,\\
{\partial_{H,2}}h_k^R&=-(\alpha_{2,k}^R)^{\dagger}  \alpha_{1,k+1}^L h_k^R, &
{\partial_{H,2}}h_k^L&=-h_k^L \alpha_{1,k+1}^R  (\alpha_{2,k}^L)^{\dagger}.
\end{align*}
\end{itemize}
\end{pro}
\begin{proof}
To obtain the partial flows it is enough to use the Lax equations for $j,p=1,2$ and operate. In order to obtain the
total flows we go back to the partial flows, and sum in $a$. From the Lax equations we know that in this total case
we no longer need to distinguish between $R,L$. This procedure leads to the result that is finally rewritten using the relations in
Proposition \ref{rrr}.
\end{proof}

\subsubsection{Bilinear equations}\label{Sec-bilinear}
Bilinear equations are an alternative way of expressing an integrable hierachy developed by the Japanese school, see \cite{date1}-\cite{date3}. We are going to show that these MOLPUC also fulfill a particular type
of bilinear equations.  These results are the matrix extensions of the  scalar situation described in \cite{carlos}. Let us start by considering
the wave semi-infinite matrices $W^H_i(t)$ \ref{wave} associated with the moment matrix $g^H$, $H=L,R$. Since the last one
is time independent the reader can easily check that
\begin{pro}
\begin{enumerate}
  \item The wave matrices associated with different times satisfy
 \begin{align}
 W_1^L(t) \big(W_1^{L}(t')\big)^{-1} &=W_2^L(t) \big(W_2^{L}(t')\big)^{-1},  &  \big(W_1^{R}(t)\big)^{-1} W_1^R(t')&=\big(W_2^{R}(t)\big)^{-1} W_2^R(t').
 \end{align}
\item The vectors $\chi,\chi^*$ fulfill
\begin{align*}
\operatorname{Res}_{z=0}[\chi(z)(\chi^*(\bar{z}))^\dagger ]&= \Res_{z=0} [\chi^*(z) (\chi(\bar{z}))^\dagger]=\mathbb{I}.
\end{align*}
\item One has that the product of two matrices can be expressed as
\begin{align*}
 UV&=\Res_{z=0}[U\chi(z)(V^\dagger\chi^*(\bar{z}))^\dagger ]= \Res_{z=0} [U\chi^*(z) (V^\dagger\chi(\bar{z}))^\dagger]=\\
   &= \Res_{z=0} [(\chi^*(\bar{z})^\top U^\dagger)^\dagger \chi(z)^\top V]=\Res_{z=0}[(\chi(\bar{z})^\top U^\dagger)^\dagger \chi^*(z)^\top V]
\end{align*}
\end{enumerate}
\end{pro}
From where we derive
\begin{theorem}
 For two different set of times $t,\tilde t$ the wave functions satisfy
\begin{align}
 &\Res_{z=0}\left[\Psi_1^L(z,t) [(\Psi_1^L)^*(\bar{z},\tilde t)]^\dagger\right]=\Res_{z=0}\left[\Psi_2^L(z,t) [(\Psi_2^L)^*(\bar{z},\tilde t)]^\dagger\right],\\
 &\Res_{z=0}\left[[(\Psi_1^R)^*(\bar{z},t)]^\dagger\Psi_1^R(z,\tilde t)\right]=\Res_{z=0}\left[[(\Psi_2^R)^*(\bar{z},t)]^\dagger\Psi_2^R(z,\tilde t)\right].
\end{align}
\end{theorem}
From the identities in \eqref{WF} the previous theorem can be rewritten in terms of CMV polynomials as
\begin{multline*}
\Res_{z=0}\Big[(\varphi_1^L)^{(l)}(z,t) \left ( \exp \Big( (t^L-\tilde t^L)\chi(z) \Big)z^{-1}F_\mu (z)\right)  [(\varphi_2^L)^{(m)}(\bar{z}^{-1},\tilde t)]^\dagger  \Big]\\=-
\Res_{z=\infty}\Big[(\varphi_1^L)^{(l)}(z,t) \left (z^{-1}F_\mu (z)\exp \Big( \chi(z)^{\top}(\tilde t^R-t^R)  \Big)\right )  [(\varphi_2^L)^{(m)}(\bar{z}^{-1},\tilde t)]^\dagger  \Big],
\end{multline*}
\begin{multline*}
\Res_{z=0}\Big[[(\varphi_2^R)^{(l)}(\bar{z}^{-1},t)]^\dagger \left (z^{-1}F_\mu (z) \exp \Big(\chi(z)^{\top}(t_R'-t_R)  \Big) \right )  (\varphi_1^R)^{(m)}(z,\tilde t) \Big]\\=-
\Res_{z=\infty}\Big[[(\varphi_2^R)^{(l)}(\bar{z}^{-1},t)]^\dagger \left (\exp \Big((t_L-t_L')\chi(z)  \Big) z^{-1}F_\mu (z) \right ) (\varphi_1^R)^{(m)}(z,\tilde t) \Big].
\end{multline*}
Where we have used that $\Res_{z=\infty}F(z)=-\Res_{z=0}z^{-2}F(z^{-1})$.
Alternatively,   we can write all the previous expressions using integrals instead of using residues. To do this, let us denote
$\gamma_0$ and $\gamma_{\infty}$ two positively oriented circles  around $z=0$ and $z=\infty$, respectively, included in the annulus of convergence of the Fourier series of the matrix measure, such that they do not include different simple poles that $z=0,\infty$, respectively. Then,
\begin{align}
 \ointclockwise_{\gamma_0}\Psi_1^L(z,t) [(\Psi_1^L)^*(\bar{z},\tilde t)]^\dagger\d z&=  \ointclockwise_{\gamma_0}\Psi_2^L(z,t) [(\Psi_2^L)^*(\bar{z},\tilde t)]^\dagger\d z,\\
 \ointclockwise_{\gamma_0}[(\Psi_1^R)^*(\bar{z},t)]^\dagger\Psi_1^L(z,\tilde t)\d z&= \ointclockwise_{\gamma_0}[(\Psi_2^R)^*(\bar{z},t)]^\dagger\Psi_2^R(z,\tilde t)\d z
\end{align}
Or, in terms of matrix Laurent orthogonal polynomials and Fourier series of the matrix measure:
\begin{pro}
  The evolved MOLPUC satisfy
  \begin{multline*}
  \ointclockwise_{\gamma_0}(\varphi_1^L)^{(l)}(z,t) \left (\exp \Big( (t^L-\tilde t^L)\chi(z) \Big)z^{-1}F_\mu (z)\right )  [(\varphi_2^L)^{(m)}(\bar{z}^{-1},\tilde t)]^\dagger \d z\\
=\ointclockwise_{\gamma_{\infty}}(\varphi_1^L)^{(l)}(z,t) \left ( z^{-1}F_\mu (z)\exp \Big( \chi(z)^{\top}(\tilde t^R-t^R)  \Big)\right )  [(\varphi_2^L)^{(m)}(\bar{z}^{-1},\tilde t)]^\dagger \d z,
\end{multline*}
\begin{multline*}
\ointclockwise_{\gamma_0}[(\varphi_2^R)^{(l)}(\bar{z}^{-1},t)]^\dagger \left (z^{-1}F_\mu (z) \exp \Big(\chi(z)^{\top}(\tilde t^R-t^R)  \Big) \right )  (\varphi_1^R)^{(m)}(z,\tilde t)\d z\\=
\ointclockwise_{\gamma_{\infty}}[(\varphi_2^R)^{(l)}(\bar{z}^{-1},t)]^\dagger \left ( \exp \Big((t^L-\tilde t^L)\chi(z)  \Big) z^{-1}F_\mu (z) \right)  (\varphi_1^R)^{(m)}(z,\tilde t)\d z.
\end{multline*}
\end{pro}
\subsection{2D Toda discrete flows}\label{Sec-discrete}
Given a couple of sequences of diagonal matrices
\begin{align*}
  d&=\{d_+,d_-\}, & d_\pm&=\{d_{\pm,0}=0,d_{\pm,1},d_{\pm,2},\dots\}, & d_{\pm,j}\in\diag_m,
\end{align*}
and a pair of non-negative integers  $n=\{n_+,n_-\}\in\Z_+^2$ we consider the next semi-infinite block matrices
\begin{align*}
  \Delta^L_d(n)&=\big(\mathbb{I}-d_{-,0}\Upsilon^{-1} \big)\cdots\big(\mathbb{I}-d_{-,n_-}\Upsilon^{-1} \big)\big(\mathbb{I}-d_{+,0}\Upsilon \big)\cdots \big(\mathbb{I}-d_{+,n_+}\Upsilon\big),\\
  \Delta^R_d(n)&=\big(\mathbb{I}-d_{-,0}\Upsilon \big)\cdots\big(\mathbb{I}-d_{-,n_-}\Upsilon\big)\big(\mathbb{I}-d_{+,0}\Upsilon^{-1}  \big)\cdots \big(\mathbb{I}-d_{+,n_+}\Upsilon^{-1} \big),
\end{align*}
Observe that the order of the factors does not alter the product as each of them commutes  with the others.
\begin{definition}
 Given two couples of sequences of diagonal matrices, say $d^H=\{d_{+}^H,d^H_{-}\}$, $H=L,R$, we introduce the discrete flows for the right and left moment matrices
 \begin{align*}
  g^H(n^L,n^R) &=\Delta^H_{d^L}(n^L)g^H\Delta^H_{d^R}(n^R),
  & n^H=\{n_+^{H},n_-^{H}\}&\in\Z_+^2, &
   g^H(0,0) &=g^H, & H&=L,R.
 \end{align*}
\end{definition}
The property $\eta g^L(n^L,n^R)=g^R(n^L,n^R)\eta$ is easily checked and it follows that we have an associated measure of which these are the corresponding left and right moment matrices given by
\begin{align*}
 \d \mu(n^L,n^R)&=\Big[\prod_{i=0}^{n^{L}_-}\left(\mathbb{I}-d^L_{-,i}z^{-1} \right) \prod_{j=0}^{n^{L}_+}\left(\mathbb{I}-d_{+,j}^Lz \right)\Big]\d \mu  \Big[\prod_{i=0}^{n^{R}_-}\left(\mathbb{I}-d^R_{-,i} z^{-1} \right) \prod_{j=0}^{n_+^{R}}\left(\mathbb{I}-d_{+,j}^Rz \right)\Big].
\end{align*}
The measure is Hermitian if the following conditions are fulfilled
\begin{align*}
\big[d_{\mp,j}^R\big]^\dagger&=d_{\pm,j}^L=d_{\pm,j}, &n_\pm^L&=n_\mp^R=n_\pm,
\end{align*}
 being the evolved measure
\begin{align*}
   \d \mu(n_+,n_-)&=\Big[\prod_{i=0}^{n_-}\left(\mathbb{I}-d_{-,i}z^{-1} \right) \prod_{j=0}^{n_+}\left(\mathbb{I}-d_{+,j}z \right)\Big]\d \mu  \Big[\prod_{i=0}^{n_-}
   \left(\mathbb{I}-d_{-,i}z^{-1} \right) \prod_{j=0}^{n_+}\left(\mathbb{I}-d_{+,j}z \right)\Big]^\dagger.
\end{align*}
Positive definiteness for the Hermitian situation  can be ensured if we request $d_i:=d_{+,i}=[d_{-,i}]^\dagger$ and $n:=n_+=n_-$ so that
\begin{align*}
   \d \mu(n)&=\Big[\sum_{a=1}^m\Big(\prod_{j=0}^{n}|1-d_{j,a}z|^2\Big)E_{a,a} \Big]\d \mu
\Big[\sum_{a=1}^m\Big(\prod_{j=0}^{n}|1-d_{j,a}z|^2\Big)E_{a,a} \Big].
\end{align*}

As in the continuous case we introduce
\begin{definition}The wave matrices, depending on discrete variables $n^L,n^R\in\Z_+^2$, are defined as
\begin{align*}
W_1^L(n^L,n^R)&= S_1(n^L,n^R)\Delta^L_{d^L}(n^L), & W_2^L(n^L,n^R)&= S_2(n^L,n^R)\big(\Delta^L_{d^R}(n^R)\big)^{-1},\\
W_1^R(n^L,n^R)&=\Delta^R_{d^R}(n^R) Z_1(n^L,n^R),&
 W_2^R(n^L,n^R)&=\big(\Delta^R_{d^L}(n^L) \big)^{-1}Z_2(n^L,n^R).
\end{align*}
\end{definition}
Hence
\begin{align*}
 g^L&=\left[W_1^L(n^L,n^R) \right]^{-1}W_2^{L}(n^L,n^R) & g^R&=W_2^{R}(n^L,n^R) \left[W_1^R(n^L,n^R) \right]^{-1}.
\end{align*}
We also need to introduce the following objects
\begin{definition}\label{delta}
\begin{enumerate}
  \item Given a diagonal matrix $d\in\diag_m$  we define the semi-infinite block matrices
\begin{align*}
  \delta^{H,H'}_\pm(d)&=\begin{cases}
S_1\left(\mathbb{I}-d\Upsilon^{\pm 1} \right)S_1^{-1}, & H=L, H'=L,\\
S_2\left(\mathbb{I}-d\Upsilon^{\pm 1} \right)S_2^{-1}, &H=R, H'=L,\\
Z_2^{-1}\left(\mathbb{I}-d\Upsilon^{\mp 1} \right)Z_2, & H=L, H'=R,\\
Z_1^{-1}\left(\mathbb{I}-d\Upsilon^{\mp 1} \right)Z_1, & H=R, H'=R.
  \end{cases}
\end{align*}
 \item The shifts are
\begin{align*}
 T^{L}_+&:\begin{pmatrix} (n^{L}_+,n^{L}_-)\longrightarrow (n^{L}_++1,n^{R}_-)\\ n^R \longrightarrow n^R \end{pmatrix}, &
 T^{L}_-&:\begin{pmatrix} (n^{L}_+,n^{R}_-)\longrightarrow(n^{L}_+,n^{L}_-+1) \\ n^R \longrightarrow n^R \end{pmatrix} ,\\
 T^{R}_+&:\begin{pmatrix} n_L\longrightarrow n_L \\  (n^{R}_+,n^{R}_-)\longrightarrow   (n^{R}_++1,n^{R}_-)\end{pmatrix}, &
 T^{R}_-&:\begin{pmatrix} n_L\longrightarrow n_L \\  (n^{R}_+,n^{R}_-)\longrightarrow  (n^{R}_+,n^{R}_-+1) \end{pmatrix}.
\end{align*}
\end{enumerate}
\end{definition}
For any diagonal matrix $d=\sum_{a=1}^m d_aE_{a,a}\in\diag_m$, $d_a\in\C$, we introduce the semi-infinite matrices
\begin{align*}
  d^{H,H'}&=\sum_{a=1}^md_aP^{H,H'}_{a},
\end{align*}
where $P^{H,H'}_a$ was defined in \eqref{projections}; observe that when  $d= c \mathbb I_m$, $c\in\C$, we have $d^{H,H'}=c\mathbb I$.

 Notice  that the $  \delta^{H,H'}_\pm$ are just particular combinations of the block  Jacobi matrices $J^H$
\begin{align*}
  \delta^{H,H'}_\pm(d)=\mathbb{I}-d^{H,H'}\big(J^{H'}\big)^{\pm 1}.
\end{align*}
\begin{pro}
  If $g^H(n^L,n^R)$, $(T^{H}_{\pm }g^{H'})(n^L,n^R)$, $H,H'=L,R$ admit a block \emph{LU} factorization then the $\delta$ matrices introduced in Definition \ref{delta}.1 also admit a block \emph{LU} factorization.
\end{pro}
\begin{proof}
We have
\begin{align*}
 T^{L}_\pm g^L&=\left(\mathbb{I}-d^L_{\pm,n^{L}_\pm+1}\Upsilon^{\pm 1} \right)g^L  & &\Longrightarrow   &
 S_1(T^{L}_{\pm} S_1)^{-1}(T^{L}_{\pm} S_2)S_2^{-1}&=S_1\left(\mathbb{I}-d^L_{\pm,n^{L}_{\pm}+1}\Upsilon^{\pm 1} \right)S_1^{-1}=\delta^{L,L}_{\pm}(d^L_{\pm,n^{L}_{\pm}+1}),\\
  T^{L}_{\pm} g^R&=\left(\mathbb{I}-d^L_{\pm,n^{L}_{\pm}+1}\Upsilon^{\mp 1} \right)g^R  & &\Longrightarrow   &
 Z_2^{-1}(  T^{L}_{\pm}  Z_2)(  T^{L}_{\pm}  Z_1)^{-1}Z_1&=Z_2^{-1}\left(\mathbb{I}-d^L_{\pm,n^{L}_{\pm}+1}\Upsilon^{\mp 1} \right)Z_2=\delta^{L,R}_{\pm}(d^L_{\pm,n^{L}_{\pm}+1}),\\
 T^{R}_{\pm} g^L&=g^L\left(\mathbb{I}-d^R_{\pm,n^{R}_\pm+1}\Upsilon^{\pm 1} \right) & &\Longrightarrow   &
 S_1(T^{R}_{\pm} S_1)^{-1}(T^{R}_{\pm} S_2)S_2^{-1}&=S_2\left(\mathbb{I}-d^R_{\pm,n^{R,\pm}+1}\Upsilon^{\pm 1} \right)S_2^{-1}= \delta^{R,L}_{\pm}(d^R_{\pm,n^{R}_{\pm}+1}),\\
T^{R}_{\pm} g^R&=g^R\left(\mathbb{I}-d_{\pm,n^{R}_{\pm}+1}^R\Upsilon^{\mp 1} \right) & &\Longrightarrow   &
 Z_2^{-1}(T^{R}_{\pm} Z_2)(T^{R}_{\pm} Z_1)^{-1}Z_1&=Z_1^{-1}\left(\mathbb{I}-d_{\pm,n^{R}_{\pm}+1}^R\Upsilon^{\mp 1} \right)Z_1
 = \delta^{R,R}_{\pm}(d_{\pm,n^{R}_{\pm}+1}^R). 
\end{align*}
Therefore, for $H=L,R$
\begin{align*}
  \delta^{H,L}_{\pm}&=(\delta^{H,L}_{\pm})_-^{-1}\cdot (\delta^{H,L}_{\pm})_+, &
  (\delta^{H,L}_{\pm})_-&=(T^{H}_{\pm} S_1)S_1^{-1}\in\mathscr L, &
  (\delta^{H,L}_{\pm})_+&=(T^{H}_{\pm} S_2)S_2^{-1}\in\mathscr U,\\
    \delta^{H,R}_{\pm}&=(\delta^{H,R}_{\pm})_-\cdot (\delta^{H,R}_{\pm})_+^{-1}, &
  (\delta^{H,R}_{\pm})_-&=Z_2^{-1}(T^{H}_{\pm} Z_2)\in\mathscr L, &
  (\delta^{H,R}_{\pm})_+&=Z_1^{-1}(T^{H}_{\pm} Z_1)\in\mathscr U.
\end{align*}
\end{proof}
\begin{definition}
  We define
\begin{align*}
  \omega^{H,H'}_{\pm}&:=
  \begin{cases}
  (\delta^{L,L}_{\pm})_+=(T^{L}_{\pm} S_2)S_2^{-1}, &  H=L, H'=L,\\
  (\delta^{R,L}_{\pm})_-=(T^{R}_{\pm} S_1)S_1^{-1}, & H=R, H'=L, \\
  (\delta^{L,R}_{\pm})_-=Z_1^{-1}(T^{L}_{\pm} Z_1), & H=L, H'=R, \\
  (\delta^{R,R}_{\pm})_+=Z_2^{-1}(T^{R}_{\pm} Z_2), & H=R, H'=R.
\end{cases}
\end{align*}
\end{definition}
We ready to derive discrete  integrability.
\begin{theorem}
 \begin{itemize}
  \item The discrete linear systems
 \begin{align*}
  T_\pm^H W_i^L&=\omega^{H,L}_{\pm}W_i^L, &   T_\pm^H W_i^R&= W_i^R\omega^{H,R}_{\pm}, & i&=1,2, \quad H=L,R.
 \end{align*}
\item Discrete  Lax equations hold
\begin{align*}
T_\pm^H J^{L}&=\omega^{H,L}_{\pm}J^L\big(\omega^{H,L}_{\pm}\big)^{-1}, &
T_\pm^H J^R&=\big(\omega^{H,R}_\pm\big)^{-1}J^R\omega^{H,R}_{\pm}, & H&=L,R.
\end{align*}
\item Intertwining matrix
\begin{align*}
  T^H_\pm C_{[p]}&=\big(\omega^{H,R}_\pm\big)^{-1}C_{[p]}\big(\omega^{H,L}_\pm\big)^{-1}, & H&=L,R.
\end{align*}
 \item Zakharov--Shabat equations
 \begin{align*}
\big( T_a^H\omega^{H',L}_{b}\big)\omega^{H,L}_{a}&=\big( T_b^{H'}\omega^{H,L}_{a}\big)\omega^{H',L}_{b}, &
\omega^{H,R}_{a}\big( T_a^H\omega^{H',R}_{b}\big)&=\omega^{H',R}_{b}\big( T_b^{H'}\omega^{H,R}_{a}\big), & a,b=\pm,\quad  H,H'=L,R.
 \end{align*}
\item Continuous-discrete equations
\begin{align*}
  \frac{\partial\omega^{H',L}_\pm}{\partial t^H_{j,a}}+\omega^{H',L}_\pm B^{H,L}_{j,a}=&\big(T^{H'}_\pm B^{H,L}_{j,a}\big)\omega^{H',L}_\pm,
  &
    \frac{\partial\omega^{H',R}_\pm}{\partial t^H_{j,a}} +B^{H,R}_{j,a}\omega^{H',R}\pm=&\omega^{H',R}_\pm\big(T^{H'}_\pm B^{H,R}_{j,a}\big),
\end{align*}
with $H,H'=L,R$, $a=1,\dots,m$ and $j=0,1,\dots$.
\end{itemize}
\end{theorem}
From these results one may derive discrete matrix equations for the Verblunsky coefficients.

 It also follows that these flows are extensions of Darboux transformations, see \cite{carlos} for the scalar case.
 Each of these discrete shifts are generalizations of the typical Darboux transformation corresponding to the flip of the upper  and lower triangular factors of the operators $\delta^{H,H'}_\pm$. These flips occur in some specific cases as follows. Let us assume that the diagonal matrices $d^H_{\pm,j}$ do not depend on $j$; then,
\begin{align*}
  \delta_\pm^{H,H'}&=\begin{cases}
    \big(\delta^{H,L}\big)_-^{-1}\big(\delta^{H,L}_\pm\big)_+, & H'=L,\\
    \big(\delta^{H,R}\big)_-\big(\delta^{H,R}_\pm\big)_+^{-1}, & H'=L,
  \end{cases}& &\stackrel{T^H_\pm}{ \xrightarrow{\hspace*{1cm}}}&
    T^H_\pm\delta_\pm^{H,H'}&=\begin{cases}
\big(\delta^{H,L}_\pm\big)_+ \big(\delta^{H,L}\big)_-^{-1}, & H'=L,\\
   \big(\delta^{H,R}_\pm\big)_+^{-1} \big(\delta^{H,R}\big)_-, & H'=L.
  \end{cases}
\end{align*}
It is clear that the shift corresponds to the flip of the factors in the Gaussian factorization of the $\delta^{H,H'}_\pm$ matrices, just as in the Darboux transformations. When the constant sequences $d^H_{\pm,j}=c^H_\pm\mathbb I_m$, with $c^H_\pm\in\C$ scalars we have that $\delta^{H,H'}_\pm$ are pentadiagonal block matrices (main diagonal and the two next diagonals above and below it), and therefore the Gauss-Borel factorizations give upper or lower block tridiagonal matrices, $(\delta^{H,H'}_\pm)_+$ and $(\delta^{H,H'}_\pm)_-$, respectively. This is quite close to some results in the talk \cite{CMV-Paco}.

\subsection{Miwa shifts}\label{Sec-Miwa}
In our unsuccessful search for a neat  $\tau$-function theory in this matrix scenario we have studied the action of Miwa shifts. Despite we did not find appropriate $\tau$-functions we found interesting relations among Christoffel--Darboux kernels and the Miwa transformations of the  MOLPUC.
These relations do in fact lead in the scalar case to the $\tau$-function representation of MOLPUC. Unfortunately, apparently that is not the case in the matrix scenario.

Miwa shifts are coherent time translations that lead to discrete type flows. Given a diagonal matrix $w=\diag(w_1,\dots,w_m)\in\C^{m\times m}$ we introduce four different $\mathfrak M^{H,\pm}_w$, $H=L,R$, coherent shifts
\begin{align*}
{ \mathfrak M}^{L,+}_{w}&:t^L_{2k} \mapsto t^L_{2k}-\frac{w^k}{k}, &
{\mathfrak M}^{L,-}_{w}&: t^L_{2k-1} \mapsto t^L_{2k-1}-\frac{w^k}{k}, &
 {\mathfrak M}^{R,+}_{w}&:t^R_{2k} \mapsto t^R_{2k}-\frac{w^k}{k},&
 {\mathfrak M}^{R,-}_{w}&:t^R_{2k-1} \mapsto t^R_{2k-1}-\frac{w^k}{k}.
 \end{align*}
 For each Miwa shift we only write down those times with a non trivial transformation. When these shifts act on the deformed matrix measure we get new matrix measures
  \begin{align}\label{eq:miwa.measure}
\d {\mathfrak M}^{L,\pm}_{w}\left[\mu \right]&=(1-wz^{\pm 1})\d \mu,  & \d {\mathfrak M}^{R,\pm}_{w}\left[\mu \right]&=\d \mu(1-wz^{\pm 1}),
 \end{align}
with corresponding left and right moment matrices given by
\begin{align}\label{eq:miwa.moment}
{\mathfrak M}^{L,\pm}_w\left[g^L\right]&=(\mathbb{I}-w\Upsilon^{\pm 1})g^L, &
 {\mathfrak M}^{L,\pm}_w\left[g^R\right]&=(\mathbb{I}-w\Upsilon^{\mp 1})g^R, &
 {\mathfrak M}^{R,\pm}_w\left[g^L\right]&=g^L(\mathbb{I}-w\Upsilon^{\pm 1}), &
{\mathfrak M}^{R,\pm}_w\left[g^R\right]&=g^R(\mathbb{I}-w\Upsilon^{\mp 1}).
\end{align}

From these we can deduce the next
\begin{theorem}\label{th: Miwa.diag.CD}
 For every diagonal matrix $w\in\diag_m$ the following relations between Miwa shifted and non shifted Christoffel-Darboux kernels and MOLPUC hold
 \begin{align*}
  K^{L,[2l+1]}(z,u)&={\mathfrak M}^{L,+}_w\left[K^{L,[2l]}\right](z,u)(\mathbb{I}-wu)+
  \left({\mathfrak M}^{L,+}_{w}\left[(\varphi_2^L)^{(2l)}\right](z)\right)^\dagger
  {\mathfrak M}^{L,+}_w\left[h_{2l}^L\right](h_{2l}^L)^{-1}(\varphi_1^L)^{(2l)}(u),\\
  K^{R,[2l]}(z,u)&={\mathfrak M}^{L,+}_w\left[K^{R,[2l-1]}\right](z,u)(\mathbb{I}-w\bar{u}^{-1})+
{\mathfrak M}^{L,+}_w\left[(\varphi_1^R)^{(2l-1)}\right](z)
{\mathfrak M}^{L,+}_w\left[h_{2l-1}^L\right](h_{2l-1}^L)^{-1}\left((\varphi_2^R)^{(2l-1)}(u)\right)^\dagger,\\
   K^{L,[2l]}(z,u)&={\mathfrak M}^{L,-}_w\left[K^{L,[2l-1]}\right](z,u)(\mathbb{I}-wu^{-1})+
  \left({\mathfrak M}^{L,-}_w\left[(\varphi_2^L)^{(2l-1)}\right](z)\right)^\dagger
  {\mathfrak M}^{L,-}_w\left[h_{2l-1}^R\right](h_{2l-1}^R)^{-1}(\varphi_1^L)^{(2l-1)}(u),\\
   K^{R,[2l+1]}(z,u)&={\mathfrak M}^{L,-}_w\left[K^{R,[2l]}\right](z,u)(\mathbb{I}-w\bar{u})+
{\mathfrak M}^{L,-}_w\left[(\varphi_1^R)^{(2l)}\right](z)
{\mathfrak M}^{L,-}_w\left[h_{2l}^L\right](h_{2l}^L)^{-1}\left((\varphi_2^R)^{(2l)}(u)\right)^\dagger,\\
  K^{L,[2l]}(z,u)&=(\mathbb{I}-w\bar{z}^{-1}){\mathfrak M}^{R,+}_w\left[K^{L,[2l-1]}\right](z,u)+
\left((\varphi_2^L)^{(2l-1)}(z)\right)^\dagger {\mathfrak M}^{R,+}_w\left[(\varphi_1^L)^{(2l-1)}\right](u),\\
 K^{R,[2l+1]}(z,u)&=(\mathbb{I}-wz){\mathfrak M}^{R,+}_w\left[K^{R,[2l]}\right](z,u)+
(\varphi_1^R)^{(2l)}(z) \left({\mathfrak M}^{R,+}_w\left[(\varphi_2^R)^{(2l)}\right](u)\right)^{\dagger},\\
K^{L,[2l+1]}(z,u)&=(\mathbb{I}-w\bar{z}){\mathfrak M}^{R,-}_w\left[K^{L,[2l]}\right](z,u)+
\left((\varphi_2^L)^{(2l)}(z)\right)^\dagger {\mathfrak M}^{R,-}_w\left[(\varphi_1^L)^{(2l)}\right](u),\\
K^{R,[2l]}(z,u)&=(\mathbb{I}-wz^{-1}){\mathfrak M}^{R,-}_w\left[K^{R,[2l-1]}\right](z,u)+
(\varphi_1^R)^{(2l-1)}(z) \left({\mathfrak M}^{R,-}_w\left[(\varphi_2^R)^{(2l-1)}\right](u)\right)^{\dagger}.
 \end{align*}
\end{theorem}
\begin{proof}
We just give the main ideas of the proof not dealing with details. Let us consider \eqref{eq:miwa.moment} at the light of the Gauss--Borel factorizations \eqref{eq:fac1} and \eqref{eq:fac2}
\begin{align*}
{\mathfrak M}^{L,\pm}_w\left[S_2\right] S_2^{-1}&={\mathfrak M}^{L,\pm}_w\left[S_1\right]\left[\mathbb{I}-w\Upsilon^{\pm 1} \right]S_1^{-1}, &
\left({\mathfrak M}^{L,\pm}_w\left[Z_1\right]\right)^{-1}Z_1&=\left({\mathfrak M}^{L,\pm}_w\left[Z_2\right]\right)^{-1}\left[\mathbb{I}-w\Upsilon^{\mp 1} \right]Z_2,\\
S_1 \left( {\mathfrak M}^{R,\pm}_w\left[S_1\right]\right)^{-1}&=S_2 \left[\mathbb{I}-w\Upsilon^{\pm 1} \right]\left(  {\mathfrak M}^{R,\pm}_w\left[S_2\right]\right)^{-1}, &
Z_2^{-1} {\mathfrak M}^{R,\pm}_w\left[Z_2\right]&=Z_1^{-1}\left[\mathbb{I}-w\Upsilon^{\mp 1} \right] {\mathfrak M}^{R,\pm}_w\left[Z_1\right].
\end{align*}
Each of these equalities defines a semi-infinite matrix relating shifted and non shifted polynomials.
At this point it is important to stress that the LHS in the two first equations are upper triangular semi-infinite  matrices, while the two last equations have in the RHS upper triangular semi-infinite matrices.
Observe also that in the two first equations, because of the RHS  only the main, the first and the second block diagonals over the first have non zero
blocks while in the LHS of the two last equations only the main diagonal and the two immediate diagonals below it have non zero blocks.
Then we proceed as in the proof of the Christoffel--Darboux formula in Theorem \ref{th:CD}. To get a glance of the technique let us illustrate it for the first equation.
In the one hand we  have for the $2l$-th and $(2l+1)$-th block rows
\begin{align*}
{\mathfrak M}^{L,+}_w\left[S_2\right] S_2^{-1}&={\mathfrak M}^{L,+}_w\left[S_1\right]\left[\mathbb{I}-w\Upsilon \right]S_1^{-1}
\\&=\begin{pmatrix}
                       \ddots  &  \ddots & & & &\\
                                    \cdots    &         0    &{\mathfrak M}^{L,+}_w\left[h_{2l}^L\right](h_{2l}^L)^{-1}   &  w\alpha_{1,2l+2}^L                                            & -w & 0&\cdots\\
                  \cdots   &          0   &                            0                                  &{\mathfrak M}^{L,+}_w\left[h_{2l+1}^R\right](h_{2l+1}^R)^{-1}       & -\left({\mathfrak M}^{L,+}_w\left[\alpha_{2,2l+1}^R\right]\right)^\dagger w & * & 0\\
                     &             &                                                              &                                                                      &  \\
                     &             &                                                              &                                                                      &\ddots
        \end{pmatrix}.
\end{align*}
In the other hand
\begin{align*}
\left({\mathfrak M}^{L,+}_w\left[\phi_2^L\right](z)\right)^\dagger{\mathfrak M}^{L,+}_w\left[S_2\right] S_2^{-1}&=
\left(\phi_2^L(z)\right)^\dagger, &
{\mathfrak M}^{L,+}_w\left[S_1\right]\left[\mathbb{I}-w\Upsilon \right]S_1^{-1}\phi_1^L(u)&=
{\mathfrak M}^{L,+}_w\left[\phi_1^L\right](u)(\mathbb{I}-wu).
\end{align*}
Then, by appropriate scalar product pairings we get the result.
\end{proof}
An appropriated choice of the variables allows us to express the rows or columns of the kernel in terms of
the rows or columns of a product of a shifted and a non shifted polynomial
\begin{cor}\label{KSnoKS}
If $w=\diag (w_1,\dots,w_m)$, $w_k\in\C$, we have
\begin{align*}
K^{L,[2l+1]}(z,w_k^{-1})E_{k,k}&=
  \left({\mathfrak M}^{L,+}_w\left[(\varphi_2^L)^{(2l)}\right](z)\right)^\dagger
 {\mathfrak M}^{L,+}_w\left[h_{2l}^L\right](h_{2l}^L)^{-1}(\varphi_1^L)^{(2l)}(w_k^{-1}) E_{k,k},\\
  K^{R,[2l]}(z,\bar{w}_k)E_{k,k}&=
{\mathfrak M}^{L,+}_w\left[(\varphi_1^R)^{(2l-1)}\right](z)
{\mathfrak M}^{L,+}_w\left[h_{2l-1}^L\right](h_{2l-1}^L)^{-1}\left((\varphi_2^R)^{(2l-1)}(\bar{w}_{k})\right)^\dagger E_{k,k},\\
   K^{L,[2l]}(z,w_k)E_{k,k}&=
  \left({\mathfrak M}^{L,-}_w\left[(\varphi_2^L)^{(2l-1)}\right](z)\right)^\dagger
 {\mathfrak M}^{L,-}_w\left[h_{2l-1}^R\right](h_{2l-1}^R)^{-1}(\varphi_1^L)^{(2l-1)}(w_{k})E_{k,k},\\
   K^{R,[2l+1]}(z,\bar{w}_{k}^{-1})E_{k,k}&=
{\mathfrak M}^{L,-}_w\left[(\varphi_1^R)^{(2l)}\right](z)
{\mathfrak M}^{L,-}_w\left[h_{2l}^L\right](h_{2l}^L)^{-1}\left((\varphi_2^R)^{(2l)}(\bar{w}_{k}^{-1})\right)^\dagger E_{k,k},\\
  E_{k,k}K^{L;[2l]}(\bar{w}_{k},u)&=
E_{k,k}\left((\varphi_2^L)^{(2l-1)}(\bar{w}_{k})\right)^\dagger {\mathfrak M}^{R,+}_w\left[(\varphi_1^L)^{(2l-1)}\right](u),\\
 E_{k,k}K^{R,[2l+1]}(w_{k}^{-1},u)&=
E_{k,k}(\varphi_1^R)^{(2l)}(w_{k}^{-1}) \left( {\mathfrak M}^{R,+}_w\left[(\varphi_2^R)^{(2l)}\right](u)\right)^{\dagger},\\
E_{k,k}K^{L,[2l+1]}(\bar{w}_{k}^{-1},u)&=
E_{k,k}\left((\varphi_2^L)^{(2l)}(\bar{w}_{k}^{-1})\right)^\dagger {\mathfrak M}^{R,-}_w\left[(\varphi_1^L)^{(2l)}\right](u),\\
E_{k,k}K^{R,[2l]}(w_{k},u)&=
E_{k,k}(\varphi_1^R)^{(2l-1)}(w_{k}) \left( {\mathfrak M}^{R,-}_w\left[(\varphi_2^R)^{(2l-1)}\right](u)\right)^{\dagger}.
\end{align*}
\end{cor}

Let us consider what happens when instead of a diagonal matrix $w$ is proportional to the identity matrix. In this case \eqref{eq:miwa.measure} informs us that
left and right handed Miwa shifts coincide. We only have two Miwa shifts ${\mathfrak M}^\pm_w$ where now $w\in\C$
  \begin{align}\label{eq:miwa.scalar.measure}
\d {\mathfrak M}^{\pm}_{w}\left[\mu \right]&=(1-wz^{\pm 1})\d \mu.
 \end{align}
In this case Corolary \ref{KSnoKS} would be written in much a simpler
way (closer to the scalar case),
\begin{pro}\label{clave teorema}
The following relations hold
\begin{align*}
[(\varphi^L_2)^{(2l-1)}(\bar w)]^\dagger {\mathfrak M}^+_w (h^{R}_{2l-1})&=(\varphi^R_1)^{(2l)}\big(w^{-1}\big)h^{R}_{2l},&
 {\mathfrak M}^-_w (h^{R}_{2l-1}) (h^{R}_{2l-1})^{-1}(\varphi^L_1)^{(2l-1)}(w)&=[(\varphi^R_2)^{(2l)}(\bar w^{-1})]^\dagger, \\
(\varphi^R_1)^{(2l-1)}(w){\mathfrak M}^-_w (h^{L}_{2l-1})&=[(\varphi^L_2)^{(2l)}(\bar w^{-1})]^\dagger  h^{L}_{2l},&
{\mathfrak M}^+_w  (h^{L}_{2l-1}) (h^{L}_{2l-1})^{-1}[(\varphi^R_2)^{(2l-1)}(\bar w)]^\dagger &=(\varphi^L_1)^{(2l)}\big(w^{-1}\big),\\
(\varphi^R_1)^{(2l)}\big(w^{-1}\big) {\mathfrak M}^+_w  (h^{R}_{2l})&=[\bar w(\varphi^L_2)^{(2l+1)}(\bar w)]^\dagger  h^{R}_{2l+1},&
 {\mathfrak M}^-_w(h^{ R}_{2l}) (h^{R}_{2l})^{-1}[(\varphi^R_2)^{(2l)}(\bar w^{-1})]^\dagger  &=z(\varphi^L_1)^{(2l+1)}(w),\\
[(\varphi^L_2)^{(2l)}(\bar w^{-1})]^\dagger  {\mathfrak M}^-_w(h^{ L}_{2l})&=w (\varphi^R_1)^{(2l+1)}(z) h^{L}_{2l+1},&
{\mathfrak M}^+_w( h^{ L}_{2l}) (h^{L}_{2l})^{-1} (\varphi^L_1)^{(2l)}\big(w^{-1}\big)&= [\bar w(\varphi^R_2)^{(2l+1)}(\bar w)]^\dagger.
\end{align*}
\end{pro}
\begin{proof}
  See Appendix \ref{proofs}
\end{proof}
Now, we can state
\begin{theorem}\label{elteorema}
 The CMV matrix Laurent orthogonal  polynomials can be expressed as follows
\begin{align}
 (\varphi^L_1)^{(2l)}(z)&= z^{l} \left[ {\mathfrak M}^+_{z^{-1}}(h^{ L}_{2l-1})(h^{L}_{2l-1})^{-1}\right]\cdots  \left[ {\mathfrak M}^+_{z^{-1}}(h^{ L}_{0}) (h^{L}_{0})^{-1}\right],\\
 (\varphi^L_1)^{(2l+1)}(z)&= z^{-(l+1)} \left[ {\mathfrak M}^-_{z}(h^{R}_{2l}) (h^{R}_{2l})^{-1}\right] \cdots \left[ {\mathfrak M}^-_{z}(h^{R}_{0} )(h^{R}_{0})^{-1}\right],\\
 [(\varphi^L_2)^{(2l)}(\bar z^{-1})]^\dagger&= z^{-l} \left[ (h^{L}_{0})^{-1}{\mathfrak M}^-_z(h^{ L}_{0})\right] \cdots \left[( h^{L}_{2l-1}) {\mathfrak M}^-_{z}(h^{ L}_{2l-1})\right] (h^{L}_{2l})^{-1},\\
 [(\varphi^L_2)^{(2l+1)}(\bar z^{-1})]^\dagger&= z^{l+1} \left[ (h^{R}_{0})^{-1} {\mathfrak M}^+_{z^{-1}}(h^{R}_{0})\right] \cdots \left[ (h^{R}_{2l})^{-1}{\mathfrak M}^+_{z^{-1}}(h^{ R}_{2l})\right] (h^{R}_{2l+1})^{-1},
\end{align}
\begin{align}
 (\varphi^R_1)^{(2l)}(z)&= z^{l} \left[ (h^{R}_{0})^{-1} {\mathfrak M}^+_{z^{-1}}(h^{\ R}_{0})\right]  \cdots\left[ (h^{R}_{2l-1})^{-1}{\mathfrak M}^+_{z^{-1}}( h^{\ R}_{2l-1})\right] (h^{R}_{2l})^{-1},\\
 (\varphi^R_1)^{(2l+1)}(z)&= z^{-(l+1)} \left[ (h^{L}_{0})^{-1} {\mathfrak M}^-_{z}(h^{ L}_{0})\right] \cdots \left[ (h^{L}_{2l})^{-1}] {\mathfrak M}^-_{z}(h^{\ L}_{2l})\right] (h^{L}_{2l+1})^{-1},\\
 [(\varphi^R_2)^{(2l)}(\bar z^{-1})]^\dagger&=z^{-l}\left[ {\mathfrak M}^-_{z}( h^{R}_{2l-1})(h^{R}_{2l-1})^{-1}\right] \cdots  \left[{\mathfrak M}^-_{z}( h^{ R}_{0})(h^{R}_{0})^{-1}\right],\\
 [(\varphi^R_2)^{(2l+1)}(\bar z^{-1})]^\dagger&= z^{l+1} \left[ {\mathfrak M}^+_{z^{-1}}(h^{\ L}_{2l})(h^{L}_{2l})^{-1}\right]  \cdots \left[ {\mathfrak M}^+_{z^{-1}}(h^{L}_{0}) (h^{L}_{0})^{-1}\right].
\end{align}
\end{theorem}
\begin{proof}
  See Appendix \ref{proofs}
\end{proof}
This is the furthest we have managed to take our $\tau$ description of the MOLPUC search.
The reader may have noticed that forgetting about the $R$ and $L$ labels and the noncommutativity of the matrix norms
we would be left with a quotient of Miwa shifted and non shifted norms which in the scalar case
coincides whith the quotient of the determinants of the truncated Miwa shifted and non shifted moment matrices.


\begin{appendices}
\section{Proofs}\label{proofs}
\paragraph{ Proposition \ref{Pro-existence LU}}
\begin{proof}
Assuming  $\det A\neq 0$ for any block matrix $M=\left(\begin{smallmatrix}
    A & B\\
    C & D
   \end{smallmatrix}\right)$ we can write in terms of Schur complements
\begin{align*}
M&=\begin{pmatrix}
    \mathbb{I} & 0\\
    CA^{-1} & \mathbb{I}
   \end{pmatrix}
\begin{pmatrix}
    A & 0\\
    0 & M\diagup A
   \end{pmatrix}
\begin{pmatrix}
    \mathbb{I} & A^{-1}B\\
    0 & \mathbb{I}
   \end{pmatrix}.
\end{align*}
Thus, as  $g^H$ is given for a  matrix quasi-definite measure
\begin{align*}
( g^H)^{[l+1]}&=\left(\begin{BMAT}{c|c}{c|c}
    \mathbb{I}_{l\times l}               &    0    \\
    v_{[l]}      &        \mathbb{I}      \\
    \end{BMAT}\right)
\left(\begin{BMAT}{c|c}{c|c}
( g^H)^{[l]}       & 0       \\
       0                 &              (g^H)^{[l+1]}\diagup (g^H)^{[l]}     \\
      \end{BMAT}\right)
\left(\begin{BMAT}{c|c}{c|c}
\mathbb{I}_{l\times l}           &  w^{[l]}      \\
           0                   &  \mathbb{I}      \\
      \end{BMAT}\right),
\end{align*}
 where $v_{[l]}=( v_0,    \dots,  v_{l-1})$ and $w^{[l]}=\left(\begin{smallmatrix} w^0 \\ w^1 \\ \vdots \\ w^{l-1} \end{smallmatrix}\right)$
are two matrix vectors.
Applying the  same factorization   to $(g^H)^{[l]}$
we get
\begin{align*}
(g^H)^{[l+1]}&=\left(\begin{BMAT}{c|cc}{c|cc}
     \mathbb{I}_{(l-1)\times (l-1)}              &     0       &  0  \\
      r_{[l-1]}                  & \mathbb{I} &  0    \\
      v'_{[l-1]}                  &      *      &  \mathbb{I}    \\
      \end{BMAT}\right)
\left(\begin{BMAT}{c|cc}{c|cc}
   (g^H)^{[l-1]}                        &      0     &    0\\
      0              & (g^H)^{[l]}\diagup (g^H)^{[l-1]} &  0    \\
      0              &          0      &  (g^H)^{[l+1]}\diagup (g^H)^{[l]}
      \end{BMAT}\right)
\left(\begin{BMAT}{c|cc}{c|cc}
   \mathbb{I}_{(l-1)\times (l-1)}                & s^{[l-1]}  &  w'^{[l-1]}  \\
             0              & \mathbb{I} &    * \\
              0              &      0      &  \mathbb{I}    \\
      \end{BMAT}\right).
\end{align*}
Finally, the iteration of these factorizations leads to
\begin{align}\label{Gss}
(g^H)^{[l+1]}&=\begin{pmatrix}
  \mathbb{I}&     0     &\dots&0\\
         *   &\mathbb{I}&    \ddots  &\vdots\\
\vdots    &\ddots          &\ddots&0\\
*       &\dots   &*   &\mathbb{I}\\
\end{pmatrix}
\begin{pmatrix}
(g^H)^{[1]}\diagup(g^H)^{[0]}&0&  \dots&0\\
     0        &       (g^H)^{[2]}\diagup(g^H)^{[1]}&    \ddots &  \vdots  \\
     \vdots &  \ddots          & \ddots&0   \\
                  0    &      \dots      &    0 &(g^H)^{[l+1]}\diagup(g^H)^{[l]}\\
\end{pmatrix}
\begin{pmatrix}
  \mathbb{I}&   * &\dots&*\\
          0  &\mathbb{I}&  \ddots&\vdots\\
           \vdots & \ddots& \ddots&*\\
            0&\dots  & 0&\mathbb{I}\\
\end{pmatrix},
\end{align}
for $H=L,R$.
Since this would have been valid for any $l$ it would also hold for the direct limit $\lim\limits_{\longrightarrow}(g^H)^{[l]}$; i. e.,  for $g^H$ with $H=L,R$.
\end{proof}

\paragraph{ Lemma \ref{lemma-alternative}}
\begin{proof}
Notice is that the third equality of each expression is  the second one written in terms of
Schur complements. Therefore, just the first and second equalities of each expression need to be proven.  The $LU$ factorization leads to
\begin{align*}
S_1^{[\geq l, l]}(g^L)^{[l]}&= -S_1^{[\geq l]}(g^L)^{[\geq l, l]},\\
S_2^{[l]}((g^L)^{[l]})^{-1}&=S_1^{[l]}
  \end{align*}
from where the result  follows immediately. As an illustration let us derive the first expression; in the one hand,
\begin{align*}\begin{aligned}
(\varphi^L)_{1}^{(l)}(z)&=\sum_{j=0}^l(S_1)_{l,j}\chi^{(j)}\\&=\chi^{(l)}+\big(S_1^{[\geq l, l]}\chi^{[l]}\big]^{(0)}\\
&=\chi^{(l)}-\sum_{m=0}^{l-1}\sum_{j=0}^l((g^L)^{[\geq l, l]})_{0,m}(((g^{L})^{[l]})^{-1})_{m,j}\chi^{(j)}\\&=\chi^{(l)}-\begin{pmatrix}
    (g^L)_{l,0}&(g^L)_{l,1}&\cdots &(g^L)_{l,l-1}
  \end{pmatrix}((g^{L})^{[l]})^{-1}\chi^{[l]},
\end{aligned}\end{align*}
and on the other
\begin{align*}\begin{aligned}
(\varphi^L)_{1}^{(l)}(z)&=\sum_{m=0}^{l}\sum_{j=0}^l(S_2^{[l+1]})_{l,m}(((g^L)^{[l+1]})^{-1})_{m,j}\chi^{(j)}\\ &=(S_2^{[l+1]})_{l,l}([(g^L)^{[l+1]}]^{-1})_{l,j}\chi^{(j)}\\&=(S_2)_{ll}\begin{pmatrix}
    0 &0 &\dots &0& \mathbb{I}
  \end{pmatrix}  ((g^L)^{[l+1]})^{-1}\chi^{[l+1]}.
\end{aligned}
\end{align*}

Proceeding in a similar manner one gets all the other identities.\footnote{It is interesting to notice that in order to prove the right case expressions,
once we have worked out the left ones, there is no need to go over the same calculations again. It is enough to realize that
\begin{align*}
 \varphi_1^R(\bar{z})&=\chi^\dagger(z) Z_1 \,\, \mbox{   same structure as   }\,\, [\varphi_2^L(z)]^\dagger=\chi^\dagger(z) S_2^{-1},\\
 \varphi_1^L(z)&=S_1\chi(z) \,\, \mbox{   same structure as   } \,\,  [\varphi_2^R(\bar{z}]^\dagger=Z_2^{-1}\chi(z).
\end{align*}}
\end{proof}
\paragraph{ Theorem \ref{th:CD}}
\begin{proof}
 We will only prove the first equation as the other three are proven in  a similar way.
 In particular, we first prove the  second equality of the first equation. We are interested in evaluating the expression
 \begin{align*}
  [(\varphi^L_2)^{[2k]}(z)]^\dagger\left[(J ^L)^{[2k]}\right](\varphi^L_{1})^{[2k]}(z')
 \end{align*}
in two different ways. In the one hand, we could first let $J$ act to the right. Truncating the expression
\begin{align*}
 J^L \varphi^L_{1}(z)=z\varphi^L_{1}(z)
\end{align*}
we  have
\begin{multline*}
\left[(J^L)^{[2k]}\right](\varphi^L_{1})^{[2k]}(z')\\=
\begin{pmatrix}
z'(\varphi^L_{1})^{(0)}(z')\\
z'(\varphi^L_{1})^{(1)}(z') \\
\vdots \\
-\alpha_{1,2k-1}^L h_{2k-2}^{R}[h_{2k-3}^{R}]^{-1}(\varphi^L_{1})^{(2k-3)}(z')-\alpha_{1,2k-1}^L[\alpha_{2,2k-2}^R]^\dagger(\varphi^L_{1})^{(2k-2)}(z')-\alpha_{1,2k}^L(\varphi^L_{1})^{(2k-1)}(z')\\
h_{2k-1}^{R}[h_{2k-3}^{R}]^{-1}(\varphi^L_{1})^{(2k-3)}(z')+h_{2k-1}^{R}[h_{2k-2}^{R}]^{-1}[\alpha_{2,2k-2}^R]^\dagger(\varphi^L_{1})^{(2k-2)}(z')-[\alpha_{2,2k-1}^R]^\dagger\alpha_{1,2k}^L(\varphi^L_{1})^{(2k-1)}(z')
\end{pmatrix}.
\end{multline*}
But
{\small\begin{align*}
 z'(\varphi^L_{1})^{(2k-2)}(z')=&-\alpha_{1,2k-1}^L h_{2k-2}^{R}[h_{2k-3}^{R}]^{-1}(\varphi^L_{1})^{(2k-3)}(z')-\alpha_{1,2k-1}^L[\alpha_{2,2k-2}^R]^\dagger(\varphi^L_{1})^{(2k-2)}(z')-\alpha_{1,2k}^L(\varphi^L_{1})^{(2k-1)}(z')+(\varphi^L_{1})^{(2k)}(z')\\
 z'(\varphi^L_{1})^{(2k-1)}(z')=&h_{2k-1}^{R}[h_{2k-3}^{R}]^{-1}(\varphi^L_{1})^{(2k-3)}(z')+h_{2k-1}^{R}[h_{2k-2}^{R}]^{-1}[\alpha_{2,2k-2}^R]^\dagger(\varphi^L_{1})^{(2k-2)}(z')\\
 &-[\alpha_{2,2k-1}^R]^\dagger\alpha_{1,2k}^L(\varphi^L_{1})^{(2k-1)}(z')+[\alpha_{2,2k-1}^R]^\dagger(\varphi^L_{1})^{(2k)}(z'),
\end{align*}}
so that we obtain
\begin{multline*}
  [(\varphi^L_2)^{[2k]}(z)]^\dagger\left[J_L^{[2k]}\right](\varphi^L_{1})^{[2k]}(z')\\
=
\left([(\varphi^L_2)^{(0)}(z)]^\dagger,
[(\varphi^L_2)^{(1)}(z)]^\dagger,\dots,
[(\varphi^L_2)^{(2k-1)}(z)]^\dagger\right)
\cdot
\begin{pmatrix}
z'(\varphi^L_{1})^{(0)}(z')\\
z'(\varphi^L_{1})^{(1)}(z') \\
\vdots \\
z'(\varphi^L_{1})^{(2k-2)}(z')-(\varphi^L_{1})^{(2k)}(z')\\
z'(\varphi^L_{1})^{(2k-1)}(z')-[\alpha_{2,2k-1}^R]^\dagger(\varphi^L_{1})^{(2k)}(z')
\end{pmatrix}.
\end{multline*}
On the other hand, we could let $J^L$ act to the left and remember that
\begin{align*}
[(\varphi^L_2)(z)]^\dagger J^L=\bar{z}^{-1} [(\varphi^L_2)(z)]^\dagger.
\end{align*}
So, truncating the expression as we did when $J^L$ acted to the right, we are  left with
\begin{multline*}
[(\varphi^L_2)^{[2k]}(z)]^\dagger (J^L)^{[2k]}\\=
\left(
\bar{z}^{-1}[(\varphi^L_2)^{(0)}(z)]^\dagger,
\dots,
\bar{z}^{-1}[(\varphi^L_2)^{(2k-2)}(z)]^\dagger,
[(\varphi^L_2)^{(2k-2)}(z)]^\dagger(-\alpha_{1,2k}^L) +[(\varphi^L_2)^{(2k-1)}(z)]^\dagger(-[\alpha_{2,2k-1}^R]^\dagger\alpha_{1,2k}^L)
\right).
\end{multline*}
But we also have
\begin{align*}
 \bar{z}^{-1}[(\varphi^L_2)^{(2k-1)}(z)]^\dagger=&
 [(\varphi^L_2)^{(2k-2)}(z)]^\dagger(-\alpha_{1,2k}^L) +[(\varphi^L_2)^{(2k-1)}(z)]^\dagger(-[\alpha_{2,2k-1}^R]^\dagger\alpha_{1,2k}^L)\\
 &+[(\varphi^L_2)^{(2k)}(z)]^\dagger[-\alpha_{1,2k+1}^L h_{2k}^R (h_{2k-1}^R)^{-1}]+
 [(\varphi^L_2)^{(2k+1)}(z)]^\dagger[h_{2k+1}^R (h_{2k-1}^R)^{-1}].
\end{align*}
So, inserting it into the equation we are interested in, we have
\small\begin{multline*}
[(\varphi^L_2)^{[2k]}(z)]^\dagger\left[(J^L)^{[2k]}\right](\varphi^L_{1})^{[2k]}(z')
\\=
\left(
\bar{z}^{-1}[(\varphi^L_2)^{(0)}(z)]^\dagger,
\dots,
\bar{z}^{-1}[(\varphi^L_2)^{(2k-1)}(z)]^\dagger
[(\varphi^L_2)^{(2k)}(z)]^\dagger \alpha_{1,2k+1}^L h_{2k}^R (h_{2k-1}^R)^{-1}
-[(\varphi^L_2)^{(2k+1)}(z)]^\dagger h_{2k+1}^R (h_{2k-1}^R)^{-1}
\right)
\begin{pmatrix}
(\varphi^L_{1})^{(0)}(z')\\
(\varphi^L_{1})^{(1)}(z') \\
\vdots \\
(\varphi^L_{1})^{(2k-1)}(z')
\end{pmatrix}.
\end{multline*}

Hence, we are left with the result we wanted to prove
\begin{align*}
\bar{z}^{-1}[(\varphi^L_2)^{[2k]}(z)]^\dagger\cdot(\varphi^L_{1})^{[2k]}(z')
+\left[
[(\varphi^L_2)^{(2k)}(z)]^\dagger \alpha_{1,2k+1}^L h_{2k}^R (h_{2k-1}^R)^{-1}
-[(\varphi^L_2)^{(2k+1)}(z)]^\dagger h_{2k+1}^R (h_{2k-1}^R)^{-1}
\right](\varphi^L_{1})^{(2k-1)}(z')\\
=[(\varphi^L_2)^{[2k]}(z)]^\dagger\cdot z'(\varphi^L_{1})^{[2k]}(z')
-\left[
[(\varphi^L_2)^{(2k-2)}(z)]^\dagger
+[(\varphi^L_2)^{(2k-1)}(z)]^\dagger[\alpha_{2,2k-1}^{R}]^\dagger
\right](\varphi^L_{1})^{(2k)}(z').
\end{align*}
Finally,  the first equality in the first equation follows from the just proven result and Proposition \ref{sss}.
As was said at the beginning of this proof the rest of the relations are proven in the exact same way.
\end{proof}

\paragraph{ Proposition \ref{integrable}}
\begin{proof}
First of all we have
\begin{align*}
{\frac{\partial}{\partial t^L_{j,a}}} W_0(t^L)&=\left[E_{a,a}\chi(\Upsilon)^{(j)} \right]W_0(t^L) & &\Longrightarrow&
{\partial_{L,j}}W_0(t^L)&=\chi(\Upsilon)^{(j)}W_0(t^L),\\
{\frac{\partial}{\partial t^R_{j,a}}}V_0(t^R) &=V_0(t^R)\left[\chi(\Upsilon^{-1})^{(j)} E_{a,a} \right] &&\Longrightarrow&
{\partial_{R,j}}V_0(t^R) &=V_0(t^R) \chi(\Upsilon^{-1})^{(j)}.
\end{align*}
The previous derivatives make sense and are well defined since the two factors in the results commute. Hence, we have
\begin{align*}
{\frac{\partial}{\partial t^L_{j,a}}} W_1^L(t)&=\left[ \left({\frac{\partial}{\partial t^L_{j,a}}} S_1(t)\right) S_1(t)^{-1}+ S_1(t)\left[E_{a,a}\chi(\Upsilon)^{(j)} \right]S_1(t)^{-1}\right]W_1^L(t),\\
{\frac{\partial}{\partial t^R_{j,a}}} W_1^L(t)&=\left[\left({\frac{\partial}{\partial t^R_{j,a}}} S_1(t) \right)S_1(t)^{-1}\right]W_1^L(t),\\
{\frac{\partial}{\partial t^R_{j,a}}} W_2^L(t)&=\left[ \left({\frac{\partial}{\partial t^R_{j,a}}} S_2(t)\right) S_2(t)^{-1}- S_2(t)\left[E_{a,a}\chi(\Upsilon)^{(j)} \right]S_2(t)^{-1}\right]W_2^L(t),\\
{\frac{\partial}{\partial t^L_{j,a}}} W_2^L(t)&=\left[\left({\frac{\partial}{\partial t^L_{j,a}}} S_2(t) \right)S_2(t)^{-1}\right]W_2^L(t),
\end{align*}
\begin{align*}
{\frac{\partial}{\partial t^L_{j,a}}} W_2^R(t)&=W_2^R(t)\left[ Z_2(t)^{-1} \left({\frac{\partial}{\partial t^L_{j,a}}}Z_2(t) \right)-Z_2(t)^{-1}\left[E_{a,a}\chi(\Upsilon^{-1})^{(j)} \right]Z_2(t)\right],\\
{\frac{\partial}{\partial t^R_{j,a}}}W_2^R(t)&=W_2^R(t)\left[ Z_2(t)^{-1} \left({\frac{\partial}{\partial t^R_{j,a}}}Z_2(t) \right)\right],\\
{\frac{\partial}{\partial t^R_{j,a}}} W_1^R(t)&=W_1^R(t)\left[ Z_1(t)^{-1} \left({\frac{\partial}{\partial t^R_{j,a}}}Z_1(t) \right)+Z_1(t)^{-1}\left[\chi(\Upsilon^{-1})^{(j)}E_{a,a} \right]Z_1(t)\right],\\
{\frac{\partial}{\partial t^L_{j,a}}} W_1^R(t)&=W_1^R(t)\left[ Z_1(t)^{-1} \left({\frac{\partial}{\partial t^L_{j,a}}}Z_1(t) \right)\right].
\end{align*}
Now, if we let ${\frac{\partial}{\partial t^H_{j,a}}}$ act on both sides of
the first expression in \eqref{ye} we obtain
\begin{align*}
 \left({\frac{\partial}{\partial t^L_{j,a}}} S_1(t)\right) S_1(t)^{-1}+ S_1(t)\left[E_{a,a}\chi(\Upsilon)^{(j)} \right]S_1(t)^{-1}&=\left({\frac{\partial}{\partial t^L_{j,a}}} S_2(t)\right) S_2(t)^{-1},\\
 \left({\frac{\partial}{\partial t^R_{j,a}}} S_2(t) \right)S_2(t)^{-1}- S_2(t)\left[\chi(\Upsilon)^{(j)} E_{a,a}\right]S_2(t)^{-1}&=\left({\frac{\partial}{\partial t^R_{j,a}}} S_1(t)\right) S_1(t)^{-1}.
\end{align*}
Which implies
\begin{align*}
\left({\frac{\partial}{\partial t^L_{j,a}}} S_2(t) \right)S_2(t)^{-1}&=\left( S_1(t)\left[E_{a,a}\chi(\Upsilon)^{(j)} \right]S_1(t)^{-1}\right)_{+},\\
\left({\frac{\partial}{\partial t^L_{j,a}}} S_1(t) \right)S_1(t)^{-1}&=-\left( S_1(t)\left[E_{a,a}\chi(\Upsilon)^{(j)} \right]S_1(t)^{-1}\right)_{-},\\
\left({\frac{\partial}{\partial t^R_{j,a}}} S_2(t) \right)S_2(t)^{-1}&=\left( S_2(t)\left[E_{a,a}\chi(\Upsilon)^{(j)} \right]S_2(t)^{-1}\right)_{+},\\
\left({\frac{\partial}{\partial t^R_{j,a}}} S_1(t) \right)S_1(t)^{-1}&=-\left( S_2(t)\left[E_{a,a}\chi(\Upsilon)^{(j)} \right]S_2(t)^{-1}\right)_{-}.
\end{align*}
Similarly let ${\frac{\partial}{\partial t^H_{j,a}}}$ act on both sides of the second expression in \eqref{ye}
\begin{align*}
    Z_2(t)^{-1} \left({\frac{\partial}{\partial t^L_{j,a}}} Z_2(t) \right)-Z_2(t)^{-1}\left[E_{a,a}\chi(\Upsilon^{-1})^{(j)} \right]Z_2(t)&=Z_1^{-1} \left({\frac{\partial}{\partial t^L_{j,a}}} Z_1(t) \right),\\
    Z_1(t)^{-1} \left({\frac{\partial}{\partial t^R_{j,a}}} Z_1(t) \right)+Z_1(t)^{-1}\left[\chi(\Upsilon^{-1})^{(j)}E_{a,a} \right]Z_1(t)&=Z_2^{-1} \left({\frac{\partial}{\partial t^R_{j,a}}}Z_2(t) \right).
\end{align*}
Which means
\begin{align*}
Z_2(t)^{-1} \left({\frac{\partial}{\partial t^L_{j,a}}} Z_2(t) \right)&=\left( Z_2(t)^{-1}\left[E_{a,a}\chi(\Upsilon^{-1})^{(j)} \right]Z_2(t)\right)_{-},\\
Z_1(t)^{-1} \left({\frac{\partial}{\partial t^L_{j,a}}} Z_1(t) \right)&=-\left( Z_2(t)^{-1}\left[E_{a,a}\chi(\Upsilon^{-1})^{(j)} \right]Z_2(t)\right)_{+},\\
Z_2(t)^{-1} \left({\frac{\partial}{\partial t^R_{j,a}}} Z_2(t) \right)&=\left( Z_1(t)^{-1}\left[E_{a,a}\chi(\Upsilon^{-1})^{(j)} \right]Z_1(t)\right)_{-},\\
Z_1(t)^{-1} \left({\frac{\partial}{\partial t^R_{j,a}}} Z_1(t) \right)&=-\left( Z_1(t)^{-1}\left[E_{a,a}\chi(\Upsilon^{-1})^{(j)} \right]Z_1(t)\right)_{+}.
\end{align*}
With all these results is easy to prove both the linear systems for the wave functions and the Lax equations.
For the flows of the intertwining operators we use these relations together with \eqref{Cpt}, the first expression for the right times and the second one for the left times; then just recall \eqref{def:ZS}. Finally,
the Zakharov--Shabat equations are just the compatibility conditions of the Lax equations.
\end{proof}

\paragraph{ Proposition \ref{clave teorema}}
\begin{proof}
First we use \eqref{norms1} and apply a Miwa shift
\begin{align*}
   {\mathfrak M}^+_w(h^R_{2l-1})&=\oint_{\T} {\mathfrak M}^+_w((\varphi^{ L}_1)^{(2l-1)})(u) \frac{\d {\mathfrak M}^+_w(\mu)(u)}{\operatorname {i}u} u^{l}.
\end{align*}
Second, from Theorem \ref{th: Miwa.diag.CD}  ($w_k=w$ for $k=1,\dots,m$) we get
\begin{align*}
\Big(\Big[(\varphi_2^L)^{(2l-1)}(\bar w)\Big]^\dagger\Big)^{-1} K^{L,[2l]}(\bar w,u)=
{\mathfrak M}^+_w \Big[(\varphi_1^L)^{(2l-1)}\Big](u),
\end{align*}
so that
\begin{align*}
\big[(\varphi^L_2)^{(2l-1)}\big(\bar{w}\big)\big]^\dagger {\mathfrak M}^+_w(h^R_{2l-1}) &=\oint_{\T}K^{L,[2l]}(\bar w,u) \frac{\d {\mathfrak M}^+_w(\mu)(u)}{\operatorname {i}u} u^{l}\\&=
\oint_{\T}K^{L,[2l]}(\bar w,u) (1-wu)\frac{\d \mu(u)}{\operatorname {i}u} u^{l} &\text{by \eqref{eq:miwa.scalar.measure}}\\&=
\oint_{\T}\Big((\varphi_1^R)^{(2l)}(w^{-1})h_{2l}^R (h^R_{2l-1})^{-1}(\varphi_1^L)^{(2l-1)}(u)&\text{by \eqref{quer1}}\\
&\hspace*{5cm}-(\varphi_1^R)^{(2l-1)}(w^{-1})(\varphi_1^L)^{(2l)}(u)\Big)\frac{\d \mu(u)}{\operatorname {i}u} u^{l}\\&=
\oint_{\T}(\varphi^R_1)^{(2l)}\big(w^{-1}\big)h^{ R}_{2l}h^{ R}_{2l-1}(\varphi^L_1)^{(2l-1)}(u)\frac{\d \mu(u)}{\operatorname {i}u} u^{l}&\text{by \eqref{ort}}\\&=
(\varphi^R_1)^{(2l)}\big(w^{-1}\big)h^{ R}_{2l}.&\text{by \eqref{norms1}}
\end{align*}
 This same procedure applies for the proof of the remaining formulae.
\end{proof}

\paragraph{ Theorem \ref{elteorema}}
\begin{proof}
Since we can take any value for $w_1,w_2$ let us consider them our variable
and name them $w_1=w_2=z$. Now by iteration of the formulae in Proposition \ref{clave teorema} we get
\begin{align*}
  [(\varphi^R_2)^{(2l+1)}(\bar z^{-1})]^\dagger&= {\mathfrak M}^+_{z^{-1}}(h^{ L}_{2l}) (h^{L}_{2l})^{-1} \cdots {\mathfrak M}^+_{z^{-1}}(h^{ L}_{1}) (h^{L}_{1})^{-1} [(\varphi^R_2)^{(1)}(\bar z^{-1})]^\dagger z^l,\\
  (\varphi^L_1)^{(2l)}(z)&= {\mathfrak M}^+_{z^{-1}}(h^ L_{2l-1}) (h^{L}_{2l-1})^{-1}  \cdots {\mathfrak M}^+_{z^{-1}}(h^{L}_{1})(h^{L}_{1})^{-1} {\mathfrak M}^+_{z^{-1}}(h^{L}_{0})(h^{L}_{0})^{-1}z^l,\\
  [(\varphi^L_2)^{(2l+1)}(\bar z^{-1})]^\dagger&= z^{l} [(\varphi^L_2)^{(1)}(\bar z^{-1})]^\dagger {\mathfrak M}^+_{z^{-1}}(h^{R}_{1})\cdots (h^{R}_{2l})^{-1} {\mathfrak M}^+_{z^{-1}}(h^{R}_{2l}) (h^{R}_{2l+1})^{-1},\\
  (\varphi^R_1)^{(2l)}(z)&= z^{l} [h^{R}_{0}]^{-1}{\mathfrak M}^z_2(h^{ R}_{0}) (h^{R}_{1})^{-1} {\mathfrak M}^+_{z^{-1}}(h^{R}_{1})\cdots (h^{R}_{2l-2})^{-1}{\mathfrak M}^z_2(h^{R}_{2l-2}) (h^{R}_{2l-1})^{-1}{\mathfrak M}^z_2(h^{R}_{2l-1}) (h^{R}_{2l})^{-1},\\
  (\varphi^L_1)^{(2l+1)}(z)&= {\mathfrak M}^-_{z}(h^{R}_{2l})(h^{R}_{2l})^{-1} {\mathfrak M}^-_{z}(h^{R}_{2l-1}) (h^{R}_{2l-1})^{-1}\cdots {\mathfrak M}^-_{z}(h^{R}_{1}) (h^{R}_{1})^{-1}(\varphi^L_1)^{(1)}(z) z^{-l},\\
  [(\varphi^R_2)^{(2l)}(\bar z^{-1})]^\dagger&= {\mathfrak M}^-_{z}(h^{R}_{2l-1}) (h^{R}_{2l-1})^{-1}  \cdots {\mathfrak M}^-_{z}(h^{R}_{0})(h^{R}_{0})^{-1} z^{-l},\\
  (\varphi^R_1)^{(2l+1)}(z)&= z^{-l} (\varphi^R_1)^{(1)}(z) {\mathfrak M}^-_{z}(h^{L}_{1})\cdots (h^{L}_{2l-1})^{-1} {\mathfrak M}^-_{z}(h^{L}_{2l-1})(h^{L}_{2l})^{-1} {\mathfrak M}^-_{z}(h^{L}_{2l})(h^{L}_{2l+1})^{-1},\\
  [(\varphi^L_2)^{(2l)}(\bar z^{-1})]^\dagger&= z^{-l} (h^{L}_{0})^{-1} {\mathfrak M}^-_{z}(h^{L}_{0})\cdots (h^{L}_{2l-1})^{-1} {\mathfrak M}^-_{z}(h^{L}_{2l-1}) (h^{L}_{2l})^{-1}.
\end{align*}
Finally, noticing that
\begin{align*}
 {\mathfrak M}^+_{z^{-1}}( h^{ L}_0)={\mathfrak M}^+_{z^{-1}}(h^{R}_0)&= \mathbb{I}-\frac{1}{z}(g^L)_{01}=\mathbb{I}-\frac{1}{z}(g^R)_{10},\\
 {\mathfrak M}^-_{z}( h^{L}_0)={\mathfrak M}^-_{z}(h^{R}_0)&= \mathbb{I}-z(g^L)_{10}=\mathbb{I}-z(g^R)_{01},
\end{align*}
we get
\begin{align*}
[(\varphi^R_2)^{(1)}(\bar z^{-1})]^\dagger&=z(\mathbb{I}-\frac{1}{z}(g^R)_{10})=z{\mathfrak M}^+_{z^{-1}}(h^{L}_0),\\
[(\varphi^L_2)^{(1)}(\bar z^{-1})]^\dagger&=z(\mathbb{I}-\frac{1}{z}(g^L)_{01})(h_1^R)^{-1}=z{\mathfrak M}^+_{z^{-1}}(h^{R}_0 )(h_1^{R})^{-1},\\
(\varphi^L_1)^{(1)}(z)&=\frac{1}{z}(\mathbb{I}-z(g^L)_{10})=\frac{1}{z} {\mathfrak M}^-_{z}(h^{R}_0),\\
(\varphi^R_1)^{(1)}(z)&=\frac{1}{z}(\mathbb{I}-z(g^R)_{01})(h_1^L)^{-1}=\frac{1}{z}{\mathfrak M}^-_{z}(h^{L}_0) (h_1^{L})^{-1},
\end{align*}
and the result is proven.
\end{proof}

\section{Explicit coefficients of $J$ and $C$}\label{explicit}
\begin{pro}
The following expressions correspond to the block non zero elements of $(J^H)^{\pm1}$
  \begin{align*}
&(J^L)_{2k,2k-1}=-h^L_{2k} \alpha_{1,2k+1}^{R} (h^R_{2k-1})^{-1},                      & &(J^R)_{2k,2k-1}=h^R_{2k} [\alpha_{2,2k+1}^L]^\dagger(h^L_{2k-1})^{-1},   \\
&(J^L)_{2k,2k}=-h^L_{2k} \alpha_{1,2k+1}^R [\alpha_{2,2k}^L]^\dagger(h^L_{2k})^{-1},     & &(J^R)_{2k,2k}=-h^R_{2k} [\alpha_{2,2k+1}^L]^\dagger \alpha_{1,2k}^R (h^R_{2k})^{-1} ,  \\
&(J^L)_{2k,2k+1}=-\alpha_{1,2k+2}^L ,                                                  & &(J^R)_{2k,2k+1}=-[\alpha_{2,2k+2}^R]^\dagger,\\
&(J^L)_{2k,2k+2}=\mathbb{I},                                                           & &(J^R)_{2k,2k+2}=\mathbb{I},\\
&(J^L)_{2k+1,2k-1}=h^R_{2k+1} (h^R_{2k-1})^{-1},                                       & &(J^R)_{2k+1,2k-1}=h^L_{2k+1} (h^L_{2k-1})^{-1},   \\
&(J^L)_{2k+1,2k}=h^R_{2k+1}[\alpha_{2,2k}^L]^\dagger (h^L_{2k})^{-1},                  & &(J^R)_{2k+1,2k}=h^L_{2k+1}\alpha_{1,2k}^R (h^R_{2k})^{-1},   \\
&(J^L)_{2k+1,2k+1}=-[\alpha_{2,2k+1}^R]^\dagger \alpha_{1,2k+2}^L,                     & &(J^R)_{2k+1,2k+1}=-\alpha_{1,2k+1}^L [\alpha_{2,2k+2}^R]^\dagger,\\
&(J^L)_{2k+1,2k+2}=[\alpha_{2,2k+1}^R]^\dagger,                                        & &(J^R)_{2k+1,2k+2}=\alpha_{1,2k+1}^L,\\
\end{align*}
\begin{align*}
&(J^L)_{0,0}=-h^L_0\alpha_{1,1}^R (h^{L}_0)^{-1},                         & &(J^R)_{0,0}=-h^R_0 [\alpha_{2,1}^L]\dagger (h^R_0)^{-1},\\
&(J^L)_{0,1}=-\alpha_{1,2}^L,                                          & &(J^R)_{0,1}=-[\alpha_{2,2}^R]^\dagger,\\
&(J^L)_{0,2}=\mathbb{I}                                                & &(J^R)_{0,2}=\mathbb{I},\\
&(J^L)_{1,0}=h^R_1 (h^{L}_0)^{-1},                                        & &(J^R)_{1,0}=h^L_1 (h^R_0)^{-1},\\
&(J^L)_{1,1}=-[\alpha_{1,2}^R]^\dagger \alpha_{1,2}^L,                  & &(J^R)_{1,1}=-\alpha_{1,1}^L [\alpha_{2,2}^R]^\dagger,\\
&(J^L)_{1,2}=[\alpha_{1,2}^R]^\dagger,                                  & &(J^R)_{1,2}=\alpha_{1,2}^L,\\
\end{align*}
\begin{align*}
&((J^L)   ^{-1})_{2k-1,2k}=-[\alpha_{2,2k+1}^R]^\dagger,                                                      &   &((J^R)   ^{-1})_{2k-1,2k}=-\alpha_{1,2k+1}^R, \\
&((J^L)   ^{-1})_{2k,2k}= -\alpha_{1,2k}^L[\alpha_{2,2k+1}^R]^\dagger,                                         &   &((J^R)   ^{-1})_{2k,2k}= -[\alpha_{2,2k}^R]^\dagger \alpha_{1,2k+1}^L,\\
&((J^L)   ^{-1})_{2k+1,2k}= -h^R_{2k+1} [\alpha_{2,2k+2}^L]^\dagger (h^L_{2k})^{-1},                           &   &((J^R)   ^{-1})_{2k+1,2k}= -h^L_{2k+1} \alpha_{1,2k+2}^R (h^R_{2k})^{-1},    \\
&((J^L)   ^{-1})_{2k+2,2k}= (h^L)_{2k+2} (h^L_{2k})^{-1},                                                        &   &((J^R)   ^{-1})_{2k+2,2k}= (h^R)_{2k+2} (h^R_{2k})^{-1},   \\
&((J^L)   ^{-1})_{2k-1,2k+1}= \mathbb{I},                                                                      &   &((J^R)   ^{-1})_{2k-1,2k+1}= \mathbb{I},\\
&((J^L)   ^{-1})_{2k,2k+1}=\alpha_{1,2k}^L,                                                                    &   &((J^R)   ^{-1})_{2k,2k+1}=[\alpha_{2,2k}^R]^\dagger,\\
&((J^L)   ^{-1})_{2k+1,2k+1}=-h^R_{2k+1} [\alpha_{2,2k+2}^L]^\dagger \alpha_{1,2k+1}^R (h^{R}_{2k+1})^{-1},      &   &((J^R)   ^{-1})_{2k+1,2k+1}=-h^L_{2k+1} \alpha_{1,2k+2}^R [\alpha_{2,2k+1}^L]^\dagger (h^L_{2k+1})^{-1},   \\
&((J^L)   ^{-1})_{2k+2,2k+1}=h^L_{2k+2}\alpha_{1,2k+1}^R(h^R_{2k+1})^{-1},                                     &   &((J^R)   ^{-1})_{2k+2,2k+1}=(h^R)_{2k+2}[\alpha_{2,2k+1}^L]^\dagger (h^L_{2k+1})^{-1},
\end{align*}
\begin{align*}
&((J^L)   ^{-1})_{0,0}= -[\alpha_{2,1}^R]^\dagger,                                                         &  &((J^R)   ^{-1})_{0,0}= -\alpha_{1,1}^L,\\
&((J^L)   ^{-1})_{1,0}= -h^R_1 [\alpha_{2,2}^L]^\dagger (h^{L}_0)^{-1},                                       &  &((J^R)   ^{-1})_{1,0}= -h^L_1 \alpha_{1,2}^R (h^{R}_0)^{-1},\\
&((J^L)   ^{-1})_{2,0}=h^L_2 (h^{L}_0)^{-1},                                                                 &  &((J^R)   ^{-1})_{2,0}=h^R_2 (h^{R}_0)^{-1},\\
&((J^L)   ^{-1})_{0,1}= \mathbb{I},                                                                         &  &((J^R)   ^{-1})_{0,1}= \mathbb{I}\\
&((J^L)   ^{-1})_{1,1}= -h^R_1 [\alpha_{2,2}^L]^\dagger \alpha_{1,1}^R (h^{R}_1)^{-1},                        &  &((J^R)   ^{-1})_{1,1}= -h^L_1 \alpha_{1,2}^R [\alpha_{2,1}^L]^\dagger (h^{L}_1)^{-1},\\
&((J^L)   ^{-1})_{2,1}= h^L_2 \alpha_{1,1}^R (h^{R}_1)^{-1},                                                  &  &((J^R)   ^{-1})_{2,1}= h^R_2 [\alpha_{2,1}^L]^\dagger (h^{L}_1)^{-1}.\\
\end{align*}
\end{pro}

\begin{pro}
The following expressions correspond to the block non zero elements of $C_{[0]}^{[\pm1]}$
\end{pro}
\begin{align*}
 &(C_{[0]})_{2k,2k-1}=h^R_{2k} [(h^R)_{2k-1}]^{-1}=\mathbb{I}-[\alpha_{2,2k}^R]^\dagger \alpha_{1,2k}^L,  \\
 &(C_{[0]})_{2k,2k}=[\alpha_{2,2k}^R]^\dagger=h^R_{2k}[\alpha_{2,2k}^L]^\dagger[h^L_{2k}]^{-1},  \\
 &(C_{[0]})_{2k+1,2k+1}=-\alpha_{1,2k+2}^L=-h^L_{2k+1}\alpha_{1,2k+2}^R[h^R_{2k+1}]^{-1},  \\
 &(C_{[0]})_{2k+1,2k+2}=\mathbb{I}=h^L_{2k+1}[\mathbb{I}-\alpha_{1,2k+2}^R[\alpha_{2,2k+2}^L]^\dagger][(h^L)_{2k+2}]^{-1},
\end{align*}
\begin{align*}
 &(C_{[0]}^{-1})_{2k,2k-1}=h^L_{2k} [(h^L)_{2k-11}]^{-1}=\mathbb{I}-\alpha_{1,2k}^L[\alpha_{2,2k}^R]^\dagger, \\
 &(C_{[0]}^{-1})_{2k,2k}=\alpha_{1,2k}^L=h^L_{2k}\alpha_{1,2k}^R[h^R_{2k}]^{-1}, \\
 &(C_{[0]}^{-1})_{2k+1,2k+1}=-[\alpha_{2,2k+2}^R]^\dagger=-h^R_{2k+1}[\alpha_{2,2k+2}^L]^\dagger[h^L_{2k+1}]^{-1},\\
 &(C_{[0]}^{-1})_{2k+1,2k+2}=\mathbb{I}=h^R_{2k+1}[\mathbb{I}-[\alpha_{2,2k+2}^L]^\dagger \alpha_{1,2k+2}^R][h^R_{2k+2}]^{-1}.
\end{align*}

\begin{pro}
The following expressions correspond to the block non zero elements of $C_{[-1]}^{[\pm1]}$
\end{pro}
\begin{align*}
  &(C_{[-1]})_{2k,2k}=-[\alpha_{2,2k+1}^R]^\dagger=-h^R_{2k}[\alpha_{2,2k+1}^L]^\dagger[h^L_{2k}]^{-1}  &
  &(C_{[-1]^{-1}})_{2k,2k}=-\alpha_{1,2k+1}^L=-h^L_{2k}\alpha_{2,2k+1}^L [h^R_{2k}]^{-1} \\
  &(C_{[-1]})_{2k,2k+1}=\mathbb{I} &
  &(C_{[-1]}^{-1})_{2k,2k+1}=\mathbb{I} \\
  &(C_{[-1]})_{2k+1,2k}=\mathbb{I}-\alpha_{1,2k+1}^L[\alpha_{2,2k+1}^R]^\dagger]=h^L_{2k+1}[h^L_{2k}]^{-1}          &
  &(C_{[-1]}^{-1})_{2k+1,2k}=\mathbb{I}-[\alpha_{2,2k+1}^R]^\dagger\alpha_{1,2k+1}^L]=h^R_{2k+1}[h^R_{2k}]^{-1} \\
  &(C_{[-1]})_{2k+1,2k+1}=\alpha_{1,2k+1}^L=h^L_{2k+1}\alpha_{1,2k+1}^R [h^R_{2k+1}]^{-1}   &
  &(C_{[-1]}^{-1})_{2k+1,2k+1}=[\alpha_{2,2k+1}^R]^\dagger=h^R_{2k+1}[\alpha_{2,2k+1}^L]^\dagger[h^L_{2k+1}]^{-1}
\end{align*}
\section{Complete recursion relations}\label{zeta-1}
Here we give a more complete set of recursion relations for the MOLPUC
\begin{pro} The five term CMV recursion relations are
{\small\begin{align*}
 z(\varphi_1^L)^{(2k)}(z)&=-\alpha_{1,2k+1}^L(\mathbb{I}-[\alpha_{2,2k}^R]^\dagger \alpha_{1,2k}^L)(\varphi_1^L)^{(2k-1)}
-\alpha_{1,2k+1}^L[\alpha_{2,2k}^R]^\dagger (\varphi_1^L)^{(2k)}-\alpha_{1,2k+2}^L (\varphi_1^L)^{(2k+1)}+(\varphi_1^L)^{(2k+2)},\\
z(\varphi_1^L)^{(2k+1)}(z)&=(\mathbb{I}-[\alpha_{2,2k+1}^R]^\dagger \alpha_{1,2k+1}^L)(\mathbb{I}-[\alpha_{2,2k}^R]^\dagger \alpha_{1,2k}^L)(\varphi_1^L)^{(2k-1)},\\
&\quad\quad+(\mathbb{I}-[\alpha_{2,2k+1}^R]^\dagger \alpha_{1,2k+1}^L)[\alpha_{2,2k}^R]^\dagger(\varphi_1^L)^{(2k)}-[\alpha_{2,2k+1}^R]^\dagger\alpha_{1,2k+2}^L(\varphi_1^L)^{(2k+1)}+[\alpha_{2,2k+1}^R]^\dagger(\varphi_1^L)^{(2k+2)},\\
z(\varphi_1^L)^{(0)}(z)&=-\alpha_{1,1}^R(\varphi_1^L)^{(0)}-\alpha_{1,2}^L(\varphi_1^L)^{(1)}+(\varphi_1^L)^{(2)},\\
z(\varphi_1^L)^{(1)}(z)&=(\mathbb{I}-[\alpha_{2,1}^R]^\dagger \alpha_{1,1}^L)(\varphi_1^L)^{(0)}-[\alpha_{2,1}^R]^\dagger\alpha_{1,2}^L(\varphi_1^L)^{(1)}+[\alpha_{2,1}^R]^\dagger (\varphi_1^L)^{(2)},
\end{align*}
\begin{align*}
z^{-1}(\varphi_1^L)^{(2k)}(z)&=(\mathbb{I}-\alpha_{1,2k}^L[\alpha_{2,2k}^R]^\dagger)(\mathbb{I}-\alpha_{1,2k-1}^L[\alpha_{2,2k-1}^R]^\dagger)(\varphi_1^L)^{(2k-2)}\\
&\quad\quad+(\mathbb{I}-\alpha_{1,2k}^L[\alpha_{2,2k}^R]^\dagger)\alpha_{1,2k-1}^L(\varphi_1^L)^{(2k-1)}-\alpha_{1,2k}^L [\alpha_{2,2k+1}^R]^\dagger(\varphi_1^L)^{(2k)}+[\alpha_{1,2k}^L]^\dagger(\varphi_1^L)^{(2k+1)},\\
z^{-1}(\varphi_1^L)^{(2k+1)}(z)&=-[\alpha_{2,2k+2}^R]^\dagger (\mathbb{I}-\alpha_{1,2k+1}^L[\alpha_{2,2k+1}^R]^\dagger )(\varphi_1^L)^{(2k)}
-[\alpha_{2,2k+2}^R]^\dagger\alpha_{1,2k+1}^L (\varphi_1^L)^{(2k+1)}-[\alpha_{2,2k+3}^R]^\dagger (\varphi_1^L)^{(2k+2)}+(\varphi_1^L)^{(2k+3)},\\
z^{-1}(\varphi_1^L)^{(0)}(z)&=-[\alpha_{2,1}^R]^\dagger(\varphi_1^L)^{(0)}+(\varphi_1^L)^{(1)},
\end{align*}
\begin{align*}
 [z(\varphi_2^L)^{(2k)}(z)]^\dagger=& -[(\varphi_2^L)^{(2k-1)}(z)]^\dagger [\alpha_{2,2k+1}^R]^\dagger- [(\varphi_2^L)^{(2k)}]^\dagger \alpha_{1,2k}^L[\alpha_{2,2k+1}^R]^\dagger-[(\varphi_2^L)^{(2k+1)}(z)]^\dagger [\alpha_{2,2k+2}^R]^\dagger (\mathbb{I}-\alpha_{1,2k+1}^L [\alpha_{2,2k+1}^R]^\dagger)\\
&\quad\quad+[(\varphi_2^L)^{(2k+2)}(z)]^\dagger (\mathbb{I}-\alpha_{1,2k+2}^L [\alpha_{2,2k+2}^R]^\dagger)(\mathbb{I}-\alpha_{1,2k+1}^L [\alpha_{2,2k+1}^R]^\dagger),\\
[z(\varphi_2^L)^{(2k+1)}(z)]^\dagger&= [(\varphi_2^L)^{(2k-1)}(z)]^\dagger + [(\varphi_2^L)^{(2k)}(z)]^\dagger \alpha_{1,2k}^L\\
&\quad\quad-[(\varphi_2^L)^{(2k+1)}(z)]^\dagger [\alpha_{2,2k+2}^R]^\dagger \alpha_{1,2k+1}^L +[(\varphi_2^L)^{(2k+2)}(z)]^\dagger (\mathbb{I}-\alpha_{1,2k+2}^L [\alpha_{2,2k+2}^R]^\dagger)\alpha_{1,2k+1}^L
\end{align*}
\begin{align*}
[z^{-1}(\varphi_2^L)^{(2k)}(z)]^\dagger=& [(\varphi_2^L)^{(2k-2)}(z)]^\dagger - [(\varphi_2^L)^{(2k-1)}(z)]^\dagger [\alpha_{2,2k-1}^R]^\dagger,\\
&-[(\varphi_2^L)^{(2k)}(z)]^\dagger \alpha_{1,2k+1}^L[\alpha_{2,2k}^R]^\dagger  +[(\varphi_2^L)^{(2k+1)}(z)]^\dagger (\mathbb{I}-[\alpha_{2,2k+1}^R]^\dagger\alpha_{1,2k+1}^L )[\alpha_{2,2k}^R]^\dagger,\\
[z^{-1}(\varphi_2^L)^{(2k+1)}(z)]^\dagger=& -[(\varphi_2^L)^{(2k)}(z)]^\dagger \alpha_{1,2k+2}^L- [(\varphi_2^L)^{(2k+1)}]^\dagger [\alpha_{2,2k+1}^R]^\dagger\alpha_{1,2k+2}^L-[(\varphi_2^L)^{(2k+2)}(z)]^\dagger \alpha_{1,2k+3}^L(\mathbb{I}-[\alpha_{2,2k+2}^R]^\dagger\alpha_{1,2k+2}^L )\\
&\quad\quad+[(\varphi_2^L)^{(2k+3)}(z)]^\dagger (\mathbb{I}-[\alpha_{2,2k+3}^R]^\dagger\alpha_{1,2k+3}^L )(\mathbb{I}- [\alpha_{2,2k+2}^R]^\dagger \alpha_{1,2k+2}^L),\\
\end{align*}
\begin{align*}
 z(\varphi_1^R)^{(2k)}=&-(\varphi_1^R)^{(2k-1)}\alpha_{1,2k+1}^L-(\varphi_1^R)^{(2k)} [\alpha_{2,2k}^R]^\dagger \alpha_{1,2k+1}^L\\
&\quad-(\varphi_1^R)^{(2k+1)}\alpha_{1,2k+2}^L(\mathbb{I}-[\alpha_{2,2k+1}^R]^\dagger\alpha_{1,2k+1}^L )+(\varphi_1^R)^{(2k+2)}(\mathbb{I}-[\alpha_{2,2k+2}^R]^\dagger\alpha_{1,2k+2}^L )(\mathbb{I}-[\alpha_{2,2k+1}^R]^\dagger\alpha_{1,2k+1}^L )\\
z(\varphi_1^R)^{(2k+1)}=&(\varphi_1^R)^{(2k-1)}+(\varphi_1^R)^{(2k)} [\alpha_{2,2k}^R]^\dagger-(\varphi_1^R)^{(2k+1)}\alpha_{1,2k+2}^L [\alpha_{2,2k+1}^R]^\dagger+(\varphi_1^R)^{(2k+2)}(\mathbb{I}-[\alpha_{2,2k+2}^R]^\dagger\alpha_{1,2k+2}^L )[\alpha_{2,2k+1}^R]^\dagger\\
\end{align*}
\begin{align*}
z^{-1}(\varphi_1^R)^{(2k)}=&(\varphi_1^R)^{(2k-2)}+(\varphi_1^R)^{(2k-1)} \alpha_{1,2k-1}^L-\\
&-(\varphi_1^R)^{(2k)}\alpha_{1,2k}^L[\alpha_{2,2k+1}^R]^\dagger+ (\varphi_1^R)^{(2k+1)}(\mathbb{I}-\alpha_{1,2k+1}^L [\alpha_{2,2k+1}^R]^\dagger)\alpha_{1,2k}^L,\\
z^{-1}(\varphi_1^R)^{(2k-1)}=&-(\varphi_1^R)^{(2k-2)}[\alpha_{2,2k}^R]^\dagger-(\varphi_1^R)^{(2k-1)} \alpha_{1,2k-1}^L[\alpha_{2,2k}^R]^\dagger\\
&-(\varphi_1^R)^{(2k)}[\alpha_{1,2k+1}^L]^\dagger(\mathbb{I}-\alpha_{1,2k}^L [\alpha_{2,2k}^R]^\dagger)+(\varphi_1^R)^{(2k+1)}(\mathbb{I}-\alpha_{1,2k+1}^L[\alpha_{2,2k+1}^R]^\dagger )(\mathbb{I}-\alpha_{1,2k}^L[\alpha_{2,2k}^R]^\dagger ),\\
\end{align*}
\begin{align*}
 [z(\varphi_2^R)^{(2k)}(z)]^\dagger=&-[\alpha_{2,2k+1}^R]^\dagger(\mathbb{I}-\alpha_{1,2k}^L [\alpha_{2,2k}^R]^\dagger)[(\varphi_2^R)^{(2k+1)}(z)]^\dagger-\\
&-[\alpha_{2,2k+1}^R]^\dagger \alpha_{1,2k}^L [(\varphi_2^R)^{(2k)}(z)]^\dagger-[\alpha_{2,2k+2}^R]^\dagger [(\varphi_2^R)^{(2k+1)}(z)]^\dagger+[(\varphi_2^R)^{(2k+2)}(z)]^\dagger\\
[z(\varphi_2^R)^{(2k+1)}(z)]^\dagger&=(\mathbb{I}-\alpha_{1,2k+1}^L [\alpha_{2,2k+1}^R]^\dagger)(\mathbb{I}-\alpha_{1,2k}^L [\alpha_{2,2k}^R]^\dagger)[(\varphi_2^R)^{(2k-1)}(z)]^\dagger+(\mathbb{I}-\alpha_{1,2k+1}^L [\alpha_{2,2k+1}^R]^\dagger)\alpha_{1,2k}^L [(\varphi_2^R)^{(2k)}(z)]^\dagger\\
&-\alpha_{1,2k+1}^L [\alpha_{2,2k+1}^R]^\dagger [(\varphi_2^R)^{(2k+1)}(z)]^\dagger+\alpha_{1,2k+1}^L [(\varphi_2^R)^{(2k+2)}(z)]^\dagger\\
[z(\varphi_2^R)^{(0)}(z)]^\dagger=&-[\alpha_{2,1}^R]^\dagger[(\varphi_2^R)^{(0)}(z)]^\dagger-[\alpha_{2,2}^R]^\dagger [(\varphi_2^R)^{(1)}(z)]^\dagger+[(\varphi_2^R)^{(2)}(z)]^\dagger\\
[z(\varphi_2^R)^{(1)}(z)]^\dagger=&(\mathbb{I}-\alpha_{1,1}^L [\alpha_{2,1}^R]^\dagger)[(\varphi_2^R)^{(0)}(z)]^\dagger-\alpha_{1,1}^L [\alpha_{2,2}^R]^\dagger[(\varphi_2^R)^{(1)}(z)]^\dagger \alpha_{1,1}^L [(\varphi_2^R)^{(2)}(z)]^\dagger
\end{align*}
\begin{align*}
[z^{-1}(\varphi_2^R)^{(2k)}(z)]^\dagger=&(\mathbb{I}-[\alpha_{2,2k}^R]^\dagger\alpha_{1,2k}^L )(\mathbb{I}-[\alpha_{2,2k-1}^R]^\dagger\alpha_{1,2k-1}^L )[(\varphi_2^R)^{(2k-2)}(z)]^\dagger+(\mathbb{I}-[\alpha_{2,2k}^R]^\dagger \alpha_{1,2k}^L )[\alpha_{2,2k-1}^R]^\dagger [(\varphi_2^R)^{(2k-1)}(z)]^\dagger\\
&-[\alpha_{2,2k}^R]^\dagger \alpha_{1,2k+1}^L  [(\varphi_2^R)^{(2k)}(z)]^\dagger+[\alpha_{2,2k}^R]^\dagger [(\varphi_2^R)^{(2k+1)}(z)]^\dagger\\
[z^{-1}(\varphi_2^R)^{(2k-1)}(z)]^\dagger=&-\alpha_{1,2k}^L(\mathbb{I}- [\alpha_{2,2k-1}^R]^\dagger\alpha_{1,2k-1}^L)[(\varphi_2^R)^{(2k-2)}(z)]^\dagger\\
&-\alpha_{1,2k}^L [\alpha_{2,2k-1}^R]^\dagger [(\varphi_2^R)^{(2k-1)}(z)]^\dagger-\alpha_{1,2k+1}^L [(\varphi_2^R)^{(2k)}(z)]^\dagger+[(\varphi_2^R)^{(2k+1)}(z)]^\dagger\\
[z^{-1}(\varphi_2^R)^{(0)}(z)]^\dagger=&-\alpha_{1,1}[(\varphi_2^R)^{(0)}(z)]^\dagger+[(\varphi_2^R)^{(1)}(z)]^\dagger
\end{align*}}
\end{pro}

\section{Projections in modules} \label{modules}
For a ring  $\mathbb M$  and  left and right modules $V$ and $W$  over $\mathbb M$, respectively,  bilinear forms  are  applications
\begin{align*}
   G&:V \times W \longrightarrow \mathbb{M}
\end{align*}
such that
\begin{align*}
  G(m_1v_1+m_2v_2,w)&=m_1G(v_1,w)+m_2G(v_2,w), & \forall m_1,m_2 &\in\mathbb{M}, &v,v_1,v_2&\in V,\\
  G(v,w_1m_1+w_2m_2)&=G(v,w_1)m_1+G(v,w_2)m_2,&
  \forall m_1,m_2 &\in\mathbb{M}, & w,w_1,w_2&\in W.
\end{align*}

 In free modules any such bilinear form can be represented by a unique $l\times r$ matrix denoted also by $G$, with coefficients in the ring $\mathbb M$, as follows
\begin{align*}
 G&:V \times W \longrightarrow \mathbb{M},\\
 G(v,w)&:=\begin{pmatrix}v_0 & \ldots & v_{l-1} \end{pmatrix} \,G\,\begin{pmatrix} w_0 \\  \vdots \\ w_{l-1}\end{pmatrix}.
\end{align*}
Given free submodules $\tilde V\subset V$ and $\tilde W\subset W$ of the modules (not necessarily free) $V, W$ and two bases $\{e_0,\dots,e_{\tilde l-1}\}\subset \tilde V$ and $\{f_0,\dots,f_{\tilde r-1}\}\subset \tilde W$  of $\tilde V$ and $\tilde W$, respectively, we denote $G_{i,j}=G(e_i,f_j)$. For the  same rank, $\tilde l=\tilde r$,  the matrix $\tilde G=(G_{i,j})$ can be assumed to be invertible, $\tilde G\in\operatorname{GL}(\tilde l,\mathbb M)\cong \operatorname{GL}(\tilde l m,\mathbb C)$. In such case we introduce the $G$-dual vectors to $e_i,f_j$ defined as
\begin{align*}
  e_i^*&=\sum_{j=0}^{\tilde l-1} f_j (\tilde G^{-1})_{j,i},&
  f_j^*=\sum_{i=0}^{\tilde r-1} (\tilde G^{-1})_{j,i}e_i.
\end{align*}

These vectors have some interesting properties
\begin{enumerate}
 \item If we change basis $\hat e_j=\sum_{i=0}^{\tilde l-1}a_{j,i}e_i$ and $\hat f_j=\sum_{i=0}^{\tilde l-1}f_i b_{i,j}$ then
 \begin{align*}
\hat e_j^*&=\sum_{i=0}^{\tilde l-1}e_i^*(a^{-1})_{i,j}&
\hat f_i^*&=\sum_{i=0}^{\tilde l-1}(b^{-1})_{i,j}f_j^*,
 \end{align*}
 where we have used the matrices $a=(a_{i,j})$ and $b=(b_{i,j})$, $a,b \in\operatorname{GL}(\tilde l,\mathbb M)$.
 \item The set of dual vectors $\{e_i^*\}_{i=0}^{\tilde l-1}$ and $\{f_i^*\}_{i=0}^{\tilde l-1}$ are bases with duals given by
 \begin{align*}
   (e^*_i)^*&=e_i, & (f_j^*)^*&=f_j.
 \end{align*}
  \item It is easy to see that they satisfy the bi-ortogonal type identity
 \begin{align*}
 G(e_i, e_j^*)=G(f_i^*,f_j)&=\delta_{i,j},&\forall i,j&=0,\dots,\tilde l-1.
 \end{align*}
\end{enumerate}
Given the bilinear form $G$ we can construct the associated projections on these
\begin{align*}
\begin{aligned}
  p&:V\rightarrow \tilde V,& p(v)&:=\sum_{i=0}^{\tilde l-1}G(v,e^*_i)e_i,\\
  q&:W\rightarrow \tilde W,& q(w)&:=\sum_{j=0}^{\tilde l-1}f_jG(f^*_j,w).
\end{aligned}
\end{align*}
These constructions are relevant when considering the Christoffel--Darboux operators and formulae in the matrix context.

\end{appendices}

\end{document}